\documentclass[11pt]{amsart}

\usepackage[paperwidth=21cm,paperheight=29.7cm,left=2cm,top=2cm]{geometry}

\usepackage[alphabetic]{amsrefs}
\usepackage[T1]{fontenc}
\usepackage{txfonts}
\usepackage{times}
\usepackage{amssymb,verbatim,amscd}
\usepackage[all]{xy}
\usepackage{subfig}
\usepackage{tikz}
\usetikzlibrary{shapes,arrows,shadows}
\usetikzlibrary{decorations.markings}
\usepackage{color}
\usepackage[normalem]{ulem}

\usepackage{hyperref}
\usepackage{wrapfig}
\usetikzlibrary{arrows}

\long\def\symbolfootnote[#1]#2{\begingroup%
\def\thefootnote{\fnsymbol{footnote}}\footnote[#1]{#2}\endgroup}

\newtheorem{theorem}{Theorem}[section]
\newtheorem*{NoNumberTheorem}{Theorem}
\newtheorem{proposition}[theorem]{Proposition}
\newtheorem{corollary}[theorem]{Corollary}
\newtheorem{lemma}[theorem]{Lemma}

\newtheorem*{convention}{Convention}

\theoremstyle{definition}

\newtheorem{remark}[theorem]{Remark}

\newtheorem{definition}[theorem]{Definition}

\newtheorem{example}[theorem]{Example}

\def\F{\mathcal{F}}
\def\Dn{\mathbf{D}_{n}}
\def\pa{\psi}
\def\Tor{\mathrm{Tor}}
\def\Id{\mathrm{Id}}

\def\Q{\mathbb{Q}}

\def\Z{\mathbb{Z}}
\def\R{\mathbb{R}}
\def\N{\mathbb{N}}

\def\ttau{\tilde{\tau}}
\def\fib{\Sigma\longrightarrow M_{\pa}\longrightarrow S^{1}}
\def\fibb{\Sigma\longrightarrow M\longrightarrow S^{1}}

\newcommand{\mcg}{\mathrm{Mod}}

\begin{document}
\title{Computing the Teichm\"uller polynomial}

\author{Erwan Lanneau and Ferr\'an Valdez}
\date{\today}

\address{
Erwan Lanneau \newline
UMR CNRS 5582 \newline
Institut Fourier, Universit\'e de Grenoble I, BP 74, 38402 Saint-Martin-d'H\`eres, France
}
\email{erwan.lanneau@ujf-grenoble.fr}
\address{
Ferr\'an Valdez \newline
Centro de Ciencias Matem\'aticas, UNAM, Campus Morelia, C.P. 58190, Morelia, \newline 
Michoac\'an,  M\'exico.
}
\email{ferran@matmor.unam.mx}

\begin{abstract}
The Teichm\"uller polynomial of a fibered 3-manifold, defined in~\cite{Mc}, plays a useful role in 
the construction of mapping class having small stretch factor. 
We provide an algorithm that computes the Teichm\"uller polynomial
of the fibered face associated to a pseudo-Anosov mapping class of a 
disc homeomorphism. As a byproduct, our algorithm allows us to derive all the relevant informations 
on the topology of the different fibers that belong to the fibered face.
\end{abstract}
\maketitle

\section{Introduction}
A fibered hyperbolic 3-manifold $M$ is a rich source
for pseudo-Anosov mapping classes: Thurston's theory of fibered faces tells us that integer points in the \emph{fibered cone} $\R^{+} \cdot F \subset H^1(M,\R)$ over the \emph{fibered face} $F$ of the Thurston norm unit ball correspond to fibrations of $M$ over the circle. Given that $M$ is hyperbolic, the monodromy of each such fibration is a pseudo-Anosov class $[\psi]$ with stretch factor $\lambda(\psi)>1$. These stretch factors are packaged in the Teichm\"uller polynomial, defined in ~\cite{Mc}. This is an element $\Theta_F=\sum_{g\in G} a_{g}g$ in the group ring $ \Z[H_1(M,\Z)/\mathrm{Torsion}]$, which is
associated to the fibered face $F$ and that is used to compute effectively the stretch factor $\lambda(\psi)$. More precisely,  if $[\alpha]\in H^{1}(M,\Z)$ is the integer class corresponding to $\pa$ in the fibered cone and $\xi_{\alpha}\in H_{1}(M,\Z)$ is its dual, then the largest root of the Laurent polynomial $\Theta_{F}(\alpha):=\sum_{g\in G}a_{g}\cdot t^{\hspace{1mm}\xi_{\alpha}(g)}\in\Z[t,t^{-1}]$ 
(in absolute value)  is the stretch factor $\lambda(\pa)$. The Teichm\"uller polynomial  has been used as a natural source of pseudo-Anosov 
homeomorphism having small {\em normalized stretch factors}: infinite families of 
pseudo-Anosov homeomorphisms $[\pa] \in \mathrm{Mod}(\Sigma_g)$ satisfying
$\lambda(\pa)^g = O(1)$ as $g\to \infty$. In particular, it has been intensively used in the papers by Hironaka~\cite{Hi-LT}, Hironaka-Kin~\cite{Hi}, 
Kin-Takasawa~\cite{KT1}, Kin-Takasawa and Mitsuhiko~\cite{KT}, Kin-Kojima-Takasawa~\cite{KKT}.
Most of known pseudo-Anosov homeomorphisms having a small normalized stretch factor are coming from 
fibrations of two very particular hyperbolic manifolds: the manifold coming from the simplest hyperbolic braid~\cite{Mc,Hi-LT} and the
``magic manifold'' see {\em e.g.}~\cite{KT}.

The Teichm\"uller polynomial was originally defined as the generator of the Fitting ideal of a module of transversals (defined by a lamination) over $\Z[H_1(M,\Z)/\mathrm{Torsion}]$. However, it is a result of McMullen~\cite{Mc} that this polynomial can also be defined in terms of the transition matrix of an infinite train track associated to a fibration on the fibered face $F$.

The main goal of our paper, based on this second definition, is to present an algorithm to compute explicitly the 
Teichm\"uller polynomial and to give a unified presentation of the aforementioned papers. More precisely we will denote the mapping torus of $[\pa]\in \mcg(S)$ by
$$
M_{\pa}:=S\times [0,1]/(x,1)\sim(\pa(x),0)
$$
and we will suppose that the first Betti number of $M_{\pa}$ is at least 2. \medskip

Based on the results of Penner and Papadopoulos \cite{PaPe} on train tracks and elementary operations
(folding operations in the present paper), we provide an algorithm that
\begin{enumerate}
\item computes the Teichm\"uller polynomial $\Theta_F$ of the fibered face $F$ of $M_\pa$ where
$[\pa]\in {\rm Mod}(S)$ is a pseudo-Anosov class.
\item computes the topology (genus, number of singularities and type) of the fibers of fibrations in the cone $\mathbb{R}^{+}\cdot F$.
\end{enumerate}

We will present our algorithm in the case where $S$ is the $n$-punctured disc $\mathbf{D}_{n}$. Then 
${\rm Mod}(S)$ is naturally isomorphic to the braid group $B(n)$. Let $\beta\in B(n)$ and let 
$[f_{\beta}]$ be the corresponding mapping class in $\mathrm{Mod}(\mathbf{D}_n)$. We shall show:

\begin{theorem}
For any pseudo-Anosov class $[f_\beta]\in \mathrm{Mod}(\mathbf{D}_n)$ 
represented by a path in some automaton
$$
\tau_{1}\stackrel{T_{1}}\longrightarrow\tau_{2}\stackrel{T_{2}}{\longrightarrow}\cdots
\stackrel{T_{n-1}}\longrightarrow\tau_{n} \stackrel{T_n}\longrightarrow\tau_{n+1},
$$
with transition matrices $M_i=M(T_i)\in \mathrm{GL}(\Z^r)$, the Teichm\"uller polynomial $\Theta_{F}({\bf t},u)$ of the associated fibered face $F$ determined by $[f_\beta]$ is:
$$
\Theta_{F}({\bf t},u)=\mathrm{det}\left(u\cdot \Id-M_1D_1 \cdot  M_{2} D_{2} \cdots M_{n}D_n R\right)
$$
where the diagonal matrices $D_i\in \mathrm{GL}(\Z[{\bf t}]^r)$ 
are uniquely determined  in terms of the path in the automaton and $R\in \mathrm{GL}(\Z^r)$ is a relabeling matrix.
\end{theorem}

For a more precise statement, in particular the nature of the variables $u$ and $\bf t$, see Theorem~\ref{thm:second:main} (see also Section~\ref{sec:automaton}
for the definition of the automaton). Observe that Bestvina and Handel~\cite{BH} have introduced an effective algorithm that determines 
whether a given homeomorphism $f\in \mathrm{Mod}(\mathbf{D}_n)$ is pseudo-Anosov. See also~\cite{B} and~\cite{H}
for implementation of the algorithm: in the pseudo-Anosov case it generates the train tracks and a path in
some corresponding automaton. 

\subsection*{Reader's guide}

In Section~\ref{sec:background} we recall Thurston's theory of fibered faces
and we review basic definitions and properties of the Teichm\"uller polynomial and its 
relation with the stretch factor associated to the monodromy of a fibration $\fibb$.
In Section~\ref{sec:tt} we describe a general strategy to compute the Teichm\"uller polynomial $\Theta_F$ 
from a train track and a train track map (after~\cite{Mc}). In Section~\ref{sec:automaton} we introduce the
notion of automaton and we give several relevant examples. In particular we use the convention 
of labelled train-tracks (similarly to Kerckhoff and Marmi-Moussa-Yoccoz did for interval exchange
transformations). Section~\ref{sec:main:theo} is devoted 
to the statement and the proof of our main theorem. Finally as a byproduct, our algorithm allows us also to derive 
all the relevant informations on the topology of the different fibers that belong to the face. This is the 
content of Section~\ref{ss:TN} (and correspond to Proposition~\ref{prop:cc:boundary}, Corollary~\ref{cor:slope} 
and Propositions~\ref{prop:cc:singu}-\ref{prop:cc:singu:2}). In Sections~\ref{sec:examples}, 
~\ref{EXB4} and Appendix~\ref{APENDIXA} we apply our results to produce several examples, recovering the ones of 
McMullen~\cite{Mc} and Hironaka~\cite{Hi-LT}, but also giving infinitely many new examples of  Teichm\"uller polynomials defined by 
pseudo-Anosov braids in $B_{n}$, for $n\geq 4$.

\subsection*{Acknowledgments}

We would like to thank Ursula Hamenst\"adt, Eriko Hironaka, J\'er\^ome Los, Curt McMullen, and Jean-Luc Thiffeault for very useful and stimulating discussions.

  We would also 
thank Centro de Ciencias Matem\'aticas, UNAM in Morelia and Institut Fourier in Grenoble for the hospitality during 
the preparation of this work. Some of the research visits which made this collaboration possible were supported by 
the ANR Project GeoDyM. The authors are partially supported by the ANR Project GeoDyM. The second author was generously supported by LAISLA, CONACYT CB-2009-01 127991 and PAPIIT projects IN100115, IN103411 \& IB100212 during the realization of this project.

\section{Thurston's theory of fibered faces and the Teichm\"uller polynomial}
\label{sec:background}

In this section we recall Thurston's theory of fibered faces. We also review the construction of 
the Teichm\"uller polynomial and its relation with the stretch factor associated 
to the monodromy of a fibration $\fibb$.

We begin by fixing several notations. Let $S$ be a surface (for which one might have
$\partial S\neq \emptyset$). Let $[\pa]$ be a class in $\mcg(S)$. 
A deep result by Thurston (see {\em e.g.}~\cite[\S 13, Thm. 13.4]{FM}) tell us that $M_{\pa}$ admits a hyperbolic metric if and only if the mapping class $[\pa]$ is pseudo-Anosov. By Mostow's rigidity theorem the isometry  class of 
$M_{\pa}$ does not depend on the choice of the representative or the conjugacy class of $[\pa]\in\mcg(S)$.

\subsection{Thurston norm and fibered faces}
Thurston introduced a very effective tool for studying essential surfaces in 3-manifolds: a 
norm on $H_2(M,\R)$. For practical reasons, we will define this norm on $H^{1}(M,\R)$. 
For nice references see {\em e.g.}~\cite[expos\'e 14]{FLP},~\cite{Cal,Thur}.

For a compact connected surface $S$, let $\chi_{-}{(S)}=|\min\{0,\chi(S)\}|$. In general, 
if a surface $S$ has $r$ connected components $S_{1},\dots,S_r$ we define 
$\chi_{-}{(S)}$ by $\sum_{i=1}^r\chi_{-}{(S_{i})}$. This determines a function $||\cdot||_{T}:H^{1}(M,\R)\to \N\cup\{0\}$ as follows:
$$
||[\alpha]||_{T}:= \inf \left\{
\chi_{-}(S)\hspace{1mm}|\hspace{1mm}\text{$S$ is a properly embedded oriented surface where $[S]$ is dual to $[\alpha]$}
\right\},
$$
where $[S]\in H_{2}(M,\Z)$ (or $H_{2}(M,\partial M;\Z)$ if $\partial M\neq \emptyset$). So far this function just measures the minimal topological complexity of a surface dual to $[\alpha]$. However, if $M$ is irreducible 
(\emph{i.e.} if every embedded sphere bounds a ball) then $||\cdot||_{T}$
satisfies the pseudo-norm properties. Therefore it has a unique continuous extension to a pseudo norm on $H^{1}(M,\R)$. 
If in addition $M$ is atoroidal and $\chi(\partial M)=0$, this continuous extension is a norm. 
This is the so called \emph{Thurston norm}. The unit ball of this norm has a very special geometry.
\begin{theorem}\cite{Thur}
	\label{Th:TNP}
Let $M$ be an irreducible and atoroidal manifold. Then the unit ball of the Thurston norm is a convex finite sided polytope.
\end{theorem}
An avid reader can consult the proof on the preceding theorem on Calegari's book (see~\cite[Theorem 5.10]{Cal}). The most striking aspect of the Thurston norm is that it provides a very nice picture for homology classes representing fibrations of $M$ over the circle.

\subsection{From homology classes to fibrations}

Let $[M,S^{1}]$ denote the set of homotopy classes of maps from $M$ to $S^{1}$. Given a class $[f]\in [M,S^{1}]$ one can choose a smooth representative $f:M\to S^{1}$ and $d\theta$ the angle form on $S^{1}$. The pullback defines a class $[f^{*}d\theta]$ in $H^{1}(M,\R)$.  This correspondence defines a bijection between $H^{1}(M,\Z)$ and $[M,S^{1}]$. We will call  $[\alpha]\in H^{1}(M,\Z)$ \emph{a fibration} if the corresponding class in $[M,S^{1}]$ is a fibration. 
Let us define:
$$
\Phi(M):=\{[\alpha]\in H^{1}(M,\Z)\hspace{1mm}|\hspace{1mm} [\alpha]\hspace{1mm}\text{is a fibration}\}
$$
and for every face $F$ of the Thurston norm ball let $\R^{+}\cdot F$ denote the positive cone in $H^{1}(M,\R)$ whose basis is $F$. 
\begin{theorem}\cite{Thur}
	\label{Th:FF}	
	Suppose that $b_{1}(M)\geq 2$. 
If $\Phi(M)\cap \R^{+}\cdot F\neq\emptyset$ for some top-dimensional face $F$ of the Thurston norm unit ball, then 
$\Phi(M)\cap \R^{+}\cdot F = H^{1}(M,\Z)\cap\R^{+}\cdot F$.
\end{theorem}
When $\Phi(M)\cap \R^{+}\cdot F\neq\emptyset$ we call $F$ a \emph{fibered face} and $\R^{+}\cdot F$ a \emph{fibered cone}. A fiber of a fibration minimizes the Thurston norm in its homology class (see~\cite[Corollary 5.13]{Cal}).

\subsection{Hyperbolic manifolds}

If the manifold M is hyperbolic, then the monodromy of each fibration $\Sigma\to M\to S^{1}$ defines a pseudo-Anosov class in $\mcg(\Sigma)$. 
Hence we can think of each integer point in a fibered cone $\R^{+}\cdot F$ as a pseudo-Anosov class (on a surface that is not necessary connected). 
We want to compute, for a fixed fibered face $F$, the stretch factors of all pseudo-Anosov maps arising as monodromies of fibrations in the fibered cone $ \R^{+}\cdot F$. This can be done by using an invariant of the fibered face called the \emph{Teichm\"uller polynomial}. Roughly speaking, this polynomial invariant is an element of the group ring $\Z[G]$, where $G=H_{1}(M,\Z)/\Tor$ and $\Tor$ is the torsion subgroup of $H_{1}(M,\Z)$. Following McMullen, let us denote it by $\Theta_{F}$. We will now explain how $\Theta_{F}$ is used to calculate stretching factors of pseudo-Anosov monodromies and we will later deal with its definition.
For any $[\alpha]\in H^{1}(M,\Z)$ we can associate a morphism $(\xi_{\alpha}:H_{1}(M,\Z)\to \Z)\in {\rm Hom}(H_{1}(M,\Z),\Z)$. 
Now $\Theta_{F}$ is an element of the group ring $\Z[G]$, thus it can be written as a formal sum:
$$
\Theta_{F}=\sum_{g\in G}a_{g}g, \hspace{3mm} a_{g}\in\Z\hspace{1mm}\text{for all}\hspace{1mm} g\in G
$$
where at most a finite number of coefficients $a_{g}$ are different from zero. The valuation of $\Theta_{F}$ at $[\alpha]$ is defined as follows:
$$
\Theta_{F}(\alpha):=\sum_{g\in G}a_{g}\cdot t^{\hspace{1mm}\xi_{\alpha}(g)}\in\Z[t,t^{-1}].
$$
Remark that $\Theta_{F}(\alpha)$ is a Laurent polynomial in $\Z[t,t^{-1}]$. Let $\lambda(\alpha)$ be the stretch factor of the pseudo-Anosov class in $\mcg(\Sigma)$ defined by the monodromy of the fibration corresponding to $[\alpha]$. The following theorem relies the Laurent polynomial $\Theta_{F}(\alpha)$ 
to $\lambda(\alpha)$.
\begin{theorem}\cite{Mc}
For any fibration $[\alpha]\in\R^{+}\cdot F$, the stretch factor $\lambda(\alpha)$ is given by the largest root 
(in absolute value) of the equation:
$$
\Theta_{F}(\alpha)=\sum_{g\in G}a_{g}\cdot t^{\hspace{1mm}\xi_{\alpha}(g)}=0.
$$
\end{theorem}
%

%
\section{Teichm\"uller polynomial and train tracks} 
\label{sec:tt}

In this section we recall the construction of the Teichm\"uller polynomial $\Theta_F$ and basic
facts on train tracks.

\subsection{The Teichm\"uller polynomial of a fibered face}
\label{SS:TPF}
In the sequel $G$ will denote $H_{1}(M;\Z)/\Tor$, where $\Tor$ denotes the torsion subgroup of $H_{1}(M;\Z)$. As before we assume that $b_{1}(M)\geq 2$. The pseudo-Anosov monodromy $\psi$ of any fibration $[\alpha]\in\R^{+}\cdot F$ with fiber $\Sigma$ has an expanding invariant lamination $\lambda\subset\Sigma$ which is unique up to isotopy.  Let $\mathcal{L}$ be the mapping torus of $\psi:\lambda\to\lambda$ and $\widetilde{\mathcal{L}}$ be the preimage of the lamination $\mathcal{L}$ on the covering space 
$$
\pi:\widetilde{M}\to M
$$
corresponding to the kernel of the map $\pi_{1}(M)\to G $. As Fried explains (see expos\'ee 14 \cite{FLP}), $\mathcal{L}$ is a compact lamination which, up to isotopy, depends only on the fibered face $F$. Using this fact, McMullen~\cite{Mc} defines the Teichm\"uller polynomial of the fibered face as
$$
\Theta_{F}=gcd(f\hspace{1mm}:\hspace{1mm} f\in I)\in \Z[G]
$$
where $I$ is the \emph{Fitting ideal} of the finitely presented $\Z[G]$-module of transversals of the lamination $\widetilde{\mathcal{L}}$. Remark that $\Theta_{F}$ is well defined up to multiplication by a unit in $\Z[G]$. One of the main results of~\cite{Mc} that we exploit in this article is a formula that allows to compute explicitly  $\Theta_{F}$. We recall how to derive this formula in the sequel.

\subsection{The setting}
	\label{SUBSEC:thesetting}
The formula that allows us to compute explicitly $\Theta_{F}$ needs a particular splitting of the group $G$. 
Fix a fiber $\Sigma\hookrightarrow M$ and let $\psi:\Sigma\to\Sigma$ be the corresponding pseudo-Anosov monodromy. 
We will denote by $H={\rm Hom}(H^{1}(\Sigma,\Z)^{\psi},\Z )\simeq \Z^{b_{1}(M)-1}$ the dual of the $\psi$-invariant cohomology of $\Sigma$.  The natural map $\pi_{1}(S)\to H_{1}(S,\Z)\to H$ determines an infinite $\Z^{b_{1}(M)-1}$-covering:
$$
   \rho:\widetilde{\Sigma}\to\Sigma 
$$
with deck transformation group $H$. We can think of $\widetilde{\Sigma}$ as a component of the preimage of a fixed fiber $\Sigma$ in the covering $\pi:\widetilde{M}\to M$ and $H$ as the subgroup of ${\rm Deck}(\pi)=G$ fixing $\widetilde{\Sigma}$. For every lift
\begin{equation}
   \label{E:lift}
   \widetilde{\psi}:\widetilde{\Sigma}\to\widetilde{\Sigma}
\end{equation}
of $\psi$, the three manifold $\widetilde{M}$ can be easily described in terms of $\widetilde{\Sigma}$ and $\widetilde{\psi}$ 
as follows. For every $k\in\Z$ let $A_{k}$ denote a copy of $\widetilde{\Sigma}\times[0,1]$. Then $\widetilde{M}$ is obtained from $\bigsqcup_{k\in\Z}A_{k}$ by identifying $(s,1)\in A_{k}$ with $(\widetilde{\psi}(s),0)\in A_{k+1}$, for every $k\in \Z$. In this setting, the deck transformation group of $\widetilde{M}$ splits as:
$$
   G=H\oplus \Z \widetilde{\Psi}
$$
where the map $\widetilde{\Psi}$ acts on 
$\widetilde{M}$ as  $\widetilde{\Psi}(s,t)=(\widetilde{\psi}(s),t-1)$. Equipped with these coordinates, 
if $F\subset H^1(M,\R)$ is the fibered face with $[\Sigma]\in \R^+\cdot F$, then we can regard 
$\Theta_F$ as a Laurent polynomial $\Theta_F(t,u)\in \Z[G]=\Z[H]\oplus \Z[u]$ where 
$t=(t_1,\dots,t_{b-1})$ is a basis of $H$ and 
$u=\widetilde{\Psi}$. 

\begin{remark}
	\label{R:dlifts}
If $\widetilde{\psi_{1}}$ and $\widetilde{\psi_{2}}$ are two different lifts of $\psi$ to $\widetilde{\Sigma}$ then $\widetilde{\psi_{1}}= t\cdot \widetilde{\psi_{2}}$ for some $t\in H={\rm Deck(\rho)}$. Hence, taking a different lift in~\eqref{E:lift} is traduced into a change of variables of the form $u'=tu$. On the other hand, since the topology of $\widetilde{M}$ is independent of $\psi$, the topology of the infinite surface $\widetilde{\Sigma}$ is also independent of $\psi$.
\end{remark}

\subsection{Train tracks}
A train track is a connected graph with an additional ``smooth''
structure.  More precisely let $\tau$ be a graph and let $h:\tau \rightarrow \Sigma$ be an
embedding so that the branches are tangent at the vertices. Since
the vertices are smooth, at each vertex the edges can be partitioned
into two sets, called ingoing and outgoing for convenience (the choice of the partition 
is arbitrary). We will also assume that at each vertex of $\tau$ we have a cyclic order
(given by $h$). This gives the notion of \emph{cusps} at a vertex:
this is a region formed by a consecutive pair (in terms of the cyclic
ordering) of either two ingoing or two outgoing edges.

The pair $(\tau,h)$ (often called simply $\tau$ if there is no
confusion) is a \emph{train track} if the additional following properties are
fulfilled:
\begin{enumerate}
\item $\tau$ has no vertex of valence $1$ or $2$;
\item The connected components of $\Sigma \backslash h(\tau)$ are either
  polygons with at least one cusp or annuli with one boundary
  contained in $\partial \Sigma$ and the other boundary with one cusp.
\end{enumerate}

A (transverse) measure $\mu$ on a train track is an assignment of 
a positive real number $\mu(e) \ge 0$ to each edge $e$ of $\tau$ which satisfy the 
switch condition at each vertex: the sum of measures of
edges in the ingoing set is the same as that for the outgoing set. The train track $\tau$ equipped
with a measure $\mu$ will be called a measured train track, and will be denoted $(\tau,h,\mu)$.

\subsection{Measured foliations and train tracks}

We can construct a (class of) measured foliation $\mathcal{F}$ from a measured train track $(\tau,h,\mu)$ as follows.
We replace each edge $e$ of $h(\tau)$ by a rectangle,  of arbitrary width and height $\mu(e)$, 
foliated by horizontal leaves. According to the switch condition, the rectangles glued together
give a subsurface $\tilde \Sigma \subset \Sigma$ (with boundaries) and a measured foliation
$\mathcal{F}$ on $\tilde \Sigma$. Now to define the foliation on the whole surface, one has to collapse the 
complementary regions. By assumption, no complementary components of $\tilde
\Sigma$ into $\Sigma$ are smooth annuli, so that we can contract each boundary
component in order to obtain a well-defined measured foliation on $\Sigma$
(see~\cite {PaPe,FLP} for details). We will call the
sides of the polygons of $\Sigma\backslash \tau$ around the punctures or around the singularities of $\mathcal{F}$
the \emph{infinitesimal edges}.

\begin{remark}
There are many arbitrary parameters in the above construction, but 
the equivalence class $[\mathcal F]$ (up to isotopy and Whitehead moves)
of $\mathcal F$ is well defined.
\end{remark}

There is a converse to the above construction. Let $\mathcal F$ be a measured foliation representing $[\mathcal F]$ and let
$p \in \mathcal F$ be a singularity. Consider a polygon $\Delta_p$ embedded into
the surface $\Sigma$, where each side $\Delta_p$ is contained in a leaf
of $\mathcal F$.  We will say that the subsurface $\Sigma \backslash \cup_{p\,
  \in\,\text{sing.}} \Delta_p$ has a partial foliation (still denoted by
$\mathcal F$) induced by $\mathcal F$. Since all complementary regions of this partial measured foliation
have at least two cusps, we can collapse the leaves of this foliation in order to obtain a measured train track $(\tau,h,\mu)$. \medskip

By considering small segments transversal to the horizontal leaves on the rectangles used in the above procedure, we obtain a  fibered neighborhood $N(\tau) \subset \Sigma$
equipped with a retraction $N(\tau) \rightarrow \tau$. The neighborhood $N(\tau)$ has
cusps on its boundary, and the fibers of the retraction are called
\emph{ties}. The train track $\tau$ can be recovered from $N(\tau)$ by
collapsing every tie to a point. We will say that $\mathcal F$ is \emph{carried} by $\tau$
(denoted by $\mathcal F \prec \tau$) if $\mathcal F$ can be represented by a partial foliation
contained in $N(\tau)$ and transverse to the ties. If in addition
no leaves of $\mathcal F$ connect cusps of $N(\tau)$, we say that $\tau$ is
\emph{suited} to $\mathcal F$. \medskip

The next sections are intended to make explicit some well-known relations
between pseudo-Anosov homeomorphisms and train track morphisms.

\subsection{Invariant train tracks}
\label{SUBSEC:INVARIANTTT}
By definition, any representative of a pseudo-Anosov class $[\psi]\in {\rm Mod}(\Sigma)$ leaves invariant a pair of transverse measured foliations $(\mathcal{F}^{s},\mathcal{F}^{u})$. However 
the action of $\psi$ on these foliations is rather difficult to describe. A good tool to understand this 
action is given by train tracks (see {\em e.g.}~\cite[\S 4]{PaPe}). Let $h:\tau\hookrightarrow \Sigma$  \emph{be suited} to $\mathcal{F}^{u}$.
Since $\mathcal{F}^{u}$ is invariant by $\psi$ it follows that $\tau$ is invariant by $\psi$, namely:
\begin{enumerate}
\item The foliation $\mathcal{F}^{u}$ can be represented by a partial measured foliation $F$ whose support is a fibered neighborhood $N(\tau)$ of $h(\tau)$.
\item The image $\psi(h(\tau))$ can be isotoped to a train track $h'(\tau')$ which is contained in a  \emph{fibered neighborhood} $N(h(\tau))$ of $h(\tau)$, \emph{is transversal to the tie foliation of $h(\tau)$} and has switches that are disjoint from the collection of central ties of $h(\tau)$. 
\end{enumerate}
If (2) holds, we will say that $\psi(\tau)$ is carried by $\tau$ and use the notation $\psi(\tau)\prec\tau$.
\begin{convention}
In this paper, we will work with \emph{labeled train tracks}, that is, triples $(\tau,h,\varepsilon)$, 
where $\varepsilon:E(\tau)\to\mathbb{A}$ is a labeling map from the set of edges of $\tau$ into a fixed 
finite alphabet $\mathbb{A}$. In the sequel we will abuse of the notation and abbreviate $(\tau,h,\varepsilon)$ with $\tau$ whenever the embedding of the 
graph $h:\tau\hookrightarrow\Sigma$ and the labelling are clear from the context.

In order to make a distinction between infinitesimal edges and other edges, we make the following choice: we label 
the infinitesimal edges by capital letters and other edges by minuscule letters. We denote by $\mathbb{A}_{\mathrm{prong}} \subset \mathbb{A}$
the $n$ letters corresponding to infinitesimal edges $E(\tau)_{\mathrm{prong}}$ enclosing the punctures of $\mathbf{D}_n$. We also denote by
$\mathbb{A}_{\mathrm{real}} \subset \mathbb{A}$ the letters corresponding to non-infinitesimal edges $E(\tau)_{\mathrm{real}}$.
 \end{convention}

\subsection{Incidence matrix}

In the above situation, if $\psi(\tau) = \sigma \prec \tau$ we naturally associate an \emph{incidence matrix} $M(\psi) \in \mathrm{GL}(\Z^\mathbb{A})$ to $\psi$ 
in the following way. For any edge $e$ of $\tau$ we make a choice
of a tie above an interior point $e$ (called the central tie associated to the edge $e$). Let $\sigma'$ isotopic 
to $\sigma$ be such that $\sigma' \subset N(\tau)$. For any edge $f$ of $\sigma$ we have a corresponding 
edge $f'$ of $\sigma'$ given by the isotopy. We can furthermore isotope $\sigma'$
slightly so that it is in general position with respect to the central ties of $\tau$. 
Now for any pair $(e,f)$ with labels $(\alpha,\beta)$ ({\em i.e.} $\varepsilon(e)=\alpha$ and $\varepsilon(f)=\beta$) 
we define $M_{\beta,\alpha}(\psi)$ as the geometric intersection between $f'$
with the central tie associated to the edge $e$ of $\tau$. \medskip

A classical theorem (see~\cite[Theorem 4.1]{PaPe}) asserts that in the pseudo-Anosov case, 
the leading eigenvalue of this matrix equals the stretch factor of the pseudo-Anosov class $[\psi]$
(if one restricts to a good set of edges, the corresponding matrix is a Perron-Frobenius matrix).\medskip

{\bf The determinant formula}. Now consider $\tilde{\tau}\subset \widetilde{\Sigma}$ a component of $\rho^{-1}(\tau)$ lying in the infinite surface $\widetilde{\Sigma}$, as defined in \S \ref{SUBSEC:thesetting}. This is an infinite train track whose set of edges and vertices can be identified with $E\times H$ and $V\times H$ respectively. Since $\tau$ is $\psi$-invariant, $\ttau$ is $\widetilde{\psi}$-invariant. This means that 
$\widetilde{\psi}(\tilde{\tau})$ can be isotoped to a train track $\widetilde{\tau}'$ which lies in a tie neighborhood of $N(\tilde{\tau})$ of $\tilde{\tau}$, is transverse to $\widetilde{\tau}$ 's ties and whose switches are disjoint from the collection of central ties of $\tilde{\tau}$. As with the train track $\tau$ and the map $\psi$, we can associate to $\widetilde{\tau}$ and $\widetilde{\psi}$ an incidence matrix $P_{E}(t) \in \mathrm{GL}(\Z[H]^\mathbb{A})$ with entries in $\Z[H]$.
Analogously, there is a matrix $P_{V}(t)$ with entries in $\Z[H]$ associated to the set of vertices of $\widetilde{\tau}$. The next theorem states that the Teichm\"uller polynomial associated to the fibered face $F$ can be recovered from the matrices $P_{E}(t)$ and $P_{V}(t)$.
\begin{theorem}[\cite{Mc}]
	\label{THM:TPformula}
	The Teichm\"uller polynomial of the fibered face $F$ is given by:
	\begin{equation}
	   \label{E:TPform}
	   \Theta_{F}(t,u)=\frac{\det(u\cdot\mathrm{Id}-P_{E}(t))}{\det(u\cdot\mathrm{Id}-P_{V}(t))}.
	\end{equation}
\end{theorem}

\subsection{Train track morphisms} We begin with a classical definition (see {\em e.g.}~\cite{Los}).

 \begin{definition}
A map $T$ between two train tracks $(\tau,h)$ and $(\tau',h')$ is a \emph{train-track morphism} if it is 
cellular and preserves the smooth structure. If in addition $(\tau,h)$ and $(\tau',h')$ are isomorphic as train tracks we call $T$ a \emph{train track map}.\\
A train-track morphism $T:\tau\to\tau'$ is a {\em representative} of $[f]\in{\rm Mod(\Sigma)}$ if in addition
\begin{enumerate}
\item The diagram
$$
\begin{CD}
\tau @>h>> \Sigma\\
@VVT V @VV f V\\
\tau' @>h'>> \Sigma
\end{CD}
$$
commutes, up to isotopy, and 
\item $f\circ h(\tau) \subset N(h'(\tau'))$ and $f\circ h(\tau)$ is transverse to the tie foliation of $h'(\tau')$.
\end{enumerate}
\end{definition}
To any train track morphism $T$ one can associate an incidence matrix $M(T)\in \mathrm{GL}(\Z^\mathbb{A})$ in the following way: 
for any pair $(e,e')$ with labels $(\alpha,\beta)$ {\em i.e.} ($\varepsilon(e)=\alpha$ and $\varepsilon(e')=\beta$) 
we define $M(T)_{\alpha,\beta}$ as the number of occurrences of $e'$ in the edge path 
$T(e)$. It is clear from the definitions that if $T:\tau\to\tau$ is the representative map of a 
homeomorphism $f: \Sigma\to\Sigma$ and if $f(\tau) \prec \tau$ then the incidence matrix 
$M(f)$ defined  in the preceding section and the incidence matrix $M(T)$ are equal.

\begin{theorem}\cite{PaPe}
\label{thm:PaPe}
Let $\psi$ be pseudo-Anosov homeomorphism and let $\mathcal F^{u}$ be in the 
class of its unstable foliation. There exists a train track $\tau$ suited to
$\mathcal F^{u}$ so that $\psi(\tau) \prec \tau$. Furthermore $\psi(\tau)$ 
is isotopic to a train track $\tau' \subset N(\tau)$ which is transverse to the ties
so that the matrix describing the linear map from the space of weights on real
edges of $\tau'$ to the space of weights on real edges of $\tau$ is primitive irreducible
({\em i.e.} some iterate all of whose entries are strictly positive).
\end{theorem}
 
\subsection{Elementary operations} 

One of the main difficulties to use the aforementioned formulas (computing matrix $M(\psi)$ and Formula~\eqref{E:TPform}) is 
that $\psi(\tau)$ (or $\widetilde{\psi}(\ttau)$) might look very complicated so that the isotopy needed to embed this 
train track in a tie neighborhood of $\tau$ transverse to the ties might be difficult to find. There is 
a general strategy that will simplify calculation, involving two natural (dual) operations defined on a train track
and usually called \emph{folding} and \emph{splitting}. \medskip

Roughly speaking, they are defined by folding or splitting a fibered neighborhood $N(\tau)$ along a cusp. 
For a more precise definition, see~\cite{Los,PaPe} (for splitting operation) and~\cite{KLS}
(for folding operation). In this paper we shall make use of {\em folding operation} 
which will produce from a train track $\tau$ a new train track $\tau'$ with the property that
$\tau \prec \tau'$. We now describe briefly the combinatorial folding operations. Observe that 
these operations first appear as (dual) right/left splits described in~\cite{PaPe}. \medskip

Let $\tau$ be a train track. Let $e_1,e_2$ be two edges of $\tau$ that are issued from the same vertex $v_1$ 
and that form a cusp $C$. We assume that one of the two edges is not infinitesimal (say it is $e_1$).
We describe the folding where edge {\em $e_1$ is folded onto edge $e_2$} (the other case being similar). 
The edge $e_2$ (respectively, $e_1$) is incident to two  vertices $v_1$ and $v_2$ (respectively, $v_1$ and $v_3$).
The orientation around $v_1$ determines an edge $e$ attached to $v_2$ (see Figure~\ref{fig:folding}). If the cusp
determined by $e$ is on the same side as the cusp $C$ then we cannot fold $e_1$ onto $e_2$. In the
other case we form a new graph $\tau'$ in the following way: we identify the edges $e_2$ and $e_1$ so that the new
graph we obtain has a new edge: $e'_1$ from $v_3$ to $v_2$. If $e$ is an infinitesimal edge we then fold again $e'_1$ on $e$. 
The new train track $(\tau',h')$ naturally inherits a labelling $\varepsilon'$ induced from the one of $\tau$: every edge of $\tau'$
share the same label than the corresponding edge of $\tau$.

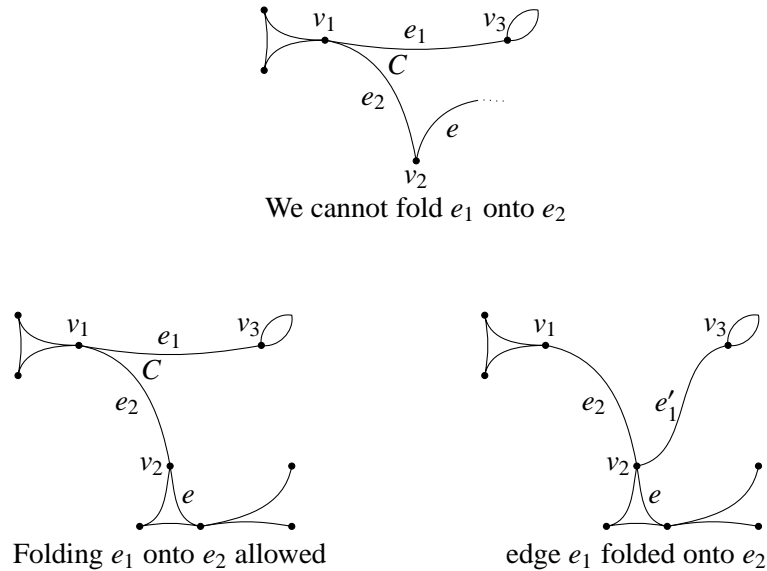
\begin{figure}[htbp]
\begin{minipage}[l]{0.23\linewidth}
\centering
\begin{tikzpicture}[scale=0.8]
\draw[smooth] (-5,2.5) to[out=280,in=180] (-4,2) to[out=180,in=80](-5,1.5) to[out=80,in=280] (-5,2.5) ;
\draw[smooth] (-4,2) to[out=350,in=190] (-1,2);
\draw[smooth] (-1,2) to[out=100,in=170] (-0.5,2.5) to[out=280,in=350] (-1,2);
\draw[smooth] (-4,2) to[out=350,in=100] (-2.5,0) to[out=80,in=190] (-1.5,1);
\draw[smooth, dotted] (-1.5,1) to[out=10,in=190] (-1,1);
\draw [fill=black] (-5,2.5) circle (0.05);
\draw [fill=black] (-5,1.5) circle (0.05);
\draw [fill=black] (-4,2) circle (0.05);
\draw [fill=black] (-1,2) circle (0.05);
\draw [fill=black] (-2.5,0) circle (0.05);
\draw (-4,2.3) node{$v_{1}$};
\draw (-2.5,2.1) node{$e_{1}$};
\draw (-1.2,2.3) node{$v_{3}$};
\draw (-3.2,1) node{$e_{2}$};
\draw (-2.5,-.3) node{$v_{2}$};
\draw (-2.8,1.6) node{$C$};
\draw (-1.9,.5) node{$e$};
\draw (-2.5,-.8) node{We cannot fold $e_1$ onto $e_2$};
\end{tikzpicture}
\end{minipage}
\vskip 10mm
\begin{minipage}[l]{0.30\linewidth}
\centering
\begin{tikzpicture}[scale=0.8]
\draw[smooth,xshift=5.5cm] (-5,2.5) to[out=280,in=180] (-4,2) to[out=180,in=80](-5,1.5) to[out=80,in=280] (-5,2.5) ;
\draw[smooth,xshift=5.5cm] (-4,2) to[out=350,in=190] (-1,2);
\draw[smooth	,xshift=5.5cm] (-1,2) to[out=100,in=170] (-0.5,2.5) to[out=280,in=350] (-1,2);
\draw[smooth,xshift=5.5cm] (-4,2) to[out=350,in=100] (-2.5,0);
\draw[smooth] (3,0) to[out=280,in=170] (3.5,-1) to[out=170,in=10] (2.5,-1) to[out=10,in=260] (3,0);
\draw[smooth] (3.5,-1) to[out=10,in=260] (5,0);
\draw[smooth] (3.5,-1) to[out=10,in=170] (5,-1);
\draw [fill=black,xshift=5.5cm] (-5,2.5) circle (0.05);
\draw [fill=black,xshift=5.5cm] (-5,1.5) circle (0.05);
\draw [fill=black,xshift=5.5cm] (-4,2) circle (0.05);
\draw [fill=black,xshift=5.5cm] (-1,2) circle (0.05);
\draw [fill=black,xshift=5.5cm] (-2.5,0) circle (0.05);
\draw [fill=black] (2.5,-1) circle (0.05);
\draw [fill=black] (3.5,-1) circle (0.05);
\draw [fill=black] (5,0) circle (0.05);
\draw [fill=black] (5,-1) circle (0.05);
\draw[xshift=5.5cm] (-4,2.3) node{$v_{1}$};
\draw[xshift=5.5cm] (-2.5,2.1) node{$e_{1}$};
\draw[xshift=5.5cm] (-1.2,2.3) node{$v_{3}$};
\draw[xshift=5.5cm] (-3.2,1) node{$e_{2}$};
\draw[xshift=5.5cm] (-2.8,0) node{$v_{2}$};
\draw[xshift=5.5cm] (-2.8,1.6) node{$C$};
\draw (3.3,-0.5) node{$e$};
\draw (3,-1.5) node{Folding $e_1$ onto $e_2$ allowed};
\end{tikzpicture}
\end{minipage}
\hskip 10mm
\begin{minipage}[l]{0.30\linewidth}
\centering
\begin{tikzpicture}[scale=0.8]
\draw[smooth,xshift=5.5cm,yshift=-5cm] (-5,2.5) to[out=280,in=180] (-4,2) to[out=180,in=80](-5,1.5) to[out=80,in=280] (-5,2.5) ;
\draw[smooth,xshift=5.5cm,yshift=-5cm] (-1,2) to[out=100,in=170] (-0.5,2.5) to[out=280,in=350] (-1,2);
\draw[smooth,xshift=5.5cm,yshift=-5cm] (-4,2) to[out=350,in=100] (-2.5,0);
\draw[smooth,yshift=-5cm] (3,0) to[out=280,in=170] (3.5,-1) to[out=170,in=10] (2.5,-1) to[out=10,in=260] (3,0);
\draw[smooth,yshift=-5cm] (3.5,-1) to[out=10,in=260] (5,0);
\draw[smooth,yshift=-5cm] (3.5,-1) to[out=10,in=170] (5,-1);
\draw[smooth,yshift=-5cm] (3,0) to[out=10,in=190] (4.5,2);
\draw [fill=black,xshift=5.5cm,yshift=-5cm] (-5,2.5) circle (0.05);
\draw [fill=black,xshift=5.5cm,yshift=-5cm] (-5,1.5) circle (0.05);
\draw [fill=black,xshift=5.5cm,yshift=-5cm] (-4,2) circle (0.05);
\draw [fill=black,xshift=5.5cm,yshift=-5cm] (-1,2) circle (0.05);
\draw [fill=black,xshift=5.5cm,yshift=-5cm] (-2.5,0) circle (0.05);
\draw [fill=black,yshift=-5cm] (2.5,-1) circle (0.05);
\draw [fill=black,yshift=-5cm] (3.5,-1) circle (0.05);
\draw [fill=black,yshift=-5cm] (5,0) circle (0.05);
\draw [fill=black,yshift=-5cm] (5,-1) circle (0.05);
\draw[xshift=5.5cm,yshift=-5cm] (-4,2.3) node{$v_{1}$};
\draw(3.5,-4) node{$e_{1}'$};
\draw[xshift=5.5cm,yshift=-5cm] (-1.2,2.3) node{$v_{3}$};
\draw[xshift=5.5cm,yshift=-5cm] (-3.2,1) node{$e_{2}$};
\draw[xshift=5.5cm,yshift=-5cm] (-2.8,0) node{$v_{2}$};
\draw (3.3,-5.5) node{$e$};
\draw (3,-6.5) node{edge $e_1$ folded onto $e_2$};
\end{tikzpicture}
\end{minipage}
\caption{The folding operation: edge $e_1$ folded onto $e_2$ produces a new train track.
\label{fig:folding}
}
\end{figure}

\begin{definition}
We will say that a train track $\tau$ refines to a train track $\sigma$ if there exists a sequence
\begin{equation}
	\label{E:refine}
\sigma=\tau_{1}\prec\tau_{2}\prec\cdots\prec\tau_{k-1}\prec\tau_{k}=\tau
\end{equation}
where $\tau_{i}$ is obtained from $\tau_{i-1}$ by a folding operation.
\end{definition}

\begin{proposition}\cite{PaPe}
	\label{P:Pape1}
Suppose that $\mathcal{F}\prec\sigma\prec\tau$ where $\sigma$ is contained in a fibered neighborhood $N(\tau)$ and $\tau$ is suited to $\mathcal{F}$. Then $\tau$ refines to $\sigma$.
\end{proposition}

\begin{proof}[Sketch of the proof of Proposition~\ref{P:Pape1}]
We prove the proposition by using splitting instead of folding, it is easier to explain and the corresponding sequence of foldings easy to derive.\\
Up to making an isotopy, one can find a fibered neighborhood $N(\sigma)$ contained in the interior of $N(\tau)$ whose tie foliation is formed by sub arcs of the tie foliation of $N(\tau)$. The number of cusps of $N(\sigma)$ and $N(\tau)$ is the same and one can define a pairing between these two sets of cusps with a family of disjointly embedded arcs $\{\Gamma_{i}\}_{i=1}^{I}$ contained in $N(\tau)\setminus \mathrm{Int}(N(\sigma))$ which are transverse to the ties 
(\cite[Lemma 2.1]{PaPe}). The sequence of splittings that defines the refinement is obtained by cutting $N(\sigma)$  along $\Gamma_{i}$, $i=1,\ldots,n$. Each time the arc $\Gamma_{i}$ crosses a singular tie of $N(\tau)$, the cutting along $\Gamma_{i}$ defines a splitting operation on $\tau$. The concatenation of these operations produces 
the sequence~\eqref{E:refine}.
\end{proof}

The previous proposition has a simple but important consequence: the refinement of 
$\tau$ to $\psi(\tau)$ allows us to factorize the incidence matrix $M(\psi)$ as a product of 
matrices associated to folding operations. In the sequel we explain how this can be done.

\subsection{Folding operations and train track morphisms}

Each folding operation from a train track $(\tau,h,\varepsilon)$ to a train track $(\tau',h',\varepsilon')$ produces a train track morphism $T:\tau\to\tau'$ that represents 
some $[f]\in{\rm Mod(\Sigma)}$ such that $f(h(\tau))\prec h'(\tau')$. Hence our preceding discussions
can be reformulated as follows.
\begin{lemma}[Penner-Papadoupoulos~\cite{PaPe}]
\label{L:factor}
Every (labeled) train track map representing a class $[f]\in{\rm Mod (\Sigma)}$ is obtained by a finite sequence of 
train track maps induced by folding operations and then followed by a relabeling operation.
\end{lemma}

Hence to any pseudo-Anosov class $[\psi]$ and any invariant train track $\tau$, one can define
a (non unique) sequence of folding operations defined by the refinement sequence
$$
\psi(\tau)=\tau_{1}\prec\tau_{2}\prec\cdots\prec\tau_{k-1}\prec\tau_{k}=\tau
$$
The sequence of folding operations defines a sequence of train track maps:
$$
\psi(\tau)=(\tau_{1},\varepsilon_1)\stackrel{T_{1}}\longrightarrow(\tau_{2},\varepsilon_{2})\stackrel{T_{2}}{\longrightarrow}\cdots
\stackrel{T_{k-1}}\longrightarrow(\tau_{k},\varepsilon_{k}) \stackrel{T_{k}}\longrightarrow(\tau_{1},\varepsilon_{k+1})\stackrel{R}\longrightarrow(\tau_{1},\varepsilon_1)
$$
Here $R$ is just a relabeling map. Therefore Lemma~\ref{L:factor} in this context implies that $T=R\circ T_{k}\circ T_{k-1}\circ\cdots\circ T_{2}\circ T_{1}$.
We draw the incidence matrix $M(\psi)$ or $M(T)$ as
\begin{equation}
	\label{E:pross}
M(T) = M(R\circ T_{k}\circ T_{k-1}\circ\cdots\circ T_{2}\circ T_{1}) = M(T_{1})M(T_{2})\cdots M(T_{k})M(R).
\end{equation}

\begin{remark}
Observe that since we work with {\em labelled train tracks}, all the transition matrices $M(T_i)$ have the 
form $\mathrm{Id}+E$ where $E$ is a non negative matrix.
\end{remark}

To summarize, to each pseudo-Anosov homeomorphisms, one can associate a train track
and a sequence of folding operations such that the corresponding product of matrices is irreducible
{\em i.e.} it has some power such that every entry has positive coefficients
(Theorem~\ref{thm:PaPe}). In general the converse is not true, however under mild assumption one has:

\begin{theorem}
\label{thm:converse}
Let $\tau=\tau_{1}\prec\tau_{2}\prec\cdots\prec\tau_{k-1}\prec\tau_{k}=\tau$ be a refinement sequence defined by folding operations
such that the corresponding incidence matrix is irreducible and the Perron-Frobenius eigenvector satisfies the 
switch conditions of the train track $\tau$. Then the train track map $T$ associated to this sequence 
is the representative of a pseudo-Anosov homeomorphism.
\end{theorem}

\subsection{Elementary operations and standardness} 
As we have seen in the preceding paragraphs, every train track map $T:\tau \longrightarrow \tau$ representing a class in ${\rm Mod}(\Sigma)$ is the product of train track maps defined by elementary operations. When $\Sigma$ is the $n$-punctured disc $\mathbf{D}_{n}$ this product can be refined by asking that every train track in~\eqref{E:refine} is \emph{standardly embedded}. Since all the calculations that we present in the present paper are 
described in this context, we will discuss these notions in detail. \medskip

In \S \ref{SUBSEC:INVARIANTTT} we made the convention that the $n$-letters in $\mathbb{A}_{\rm prong}$ label the infinitesimal edges enclosing punctures of $\mathbf{D}_{n}$, hence any labeling using these letters defines a labeling of the punctures of $\mathbf{D}_{n}$. We consider $\mathbf{D}_{n}$ to be modeled on the unit disc in $\mathbf{C}$ with $n$ punctures along the real line $\mathbf{R}$. 
Let $l_{\alpha}$ be a vertical segment joining the puncture labeled by $\alpha\in\mathbb{A}_{\rm prong}$ 
to the boundary of the disc. Now consider a train track $h:\tau\hookrightarrow\mathbf{D}_{n}$. If all the edges except these infinitesimal edges are embedded in the upper (or lower) half disc, then we say that $(\tau,h)$ is \emph{standard}. If only one open real edge of $h(\tau)$ intersects only once $\cup l_{i}$, 
and all 

the other real edges are embedded in the upper (or lower) half disc, then we say that $(\tau,h)$ is \emph{almost standard}. These notions were first introduced in \cite{KLS}.

We consider $B_{n}$ the n-th braid group with standard generators $\sigma_{1},\ldots, \sigma_{n-1}$ and consider the natural map 
$B_{n}\longrightarrow{\rm Mod}(\mathbf{D}_{n})$ which associates to each braid $\beta$ the mapping class $f_{\beta}$. 
If $(\tau,h)$ is standard and we perform a folding operation on $h(\tau)$, then the resulting train track $(\tau_{1},h_{1})$ is
either standard or almost standard. In the latter situation we can easily turn $(\tau_{1},h_{1})$ into a standard marking.
\begin{lemma}~\cite{KLS}
	\label{Lemma:StantardizingBraid}
	Let $(\tau,h)$ be an almost standard train track in $\mathbf{D}_{n}$. Then there exist some $n$-braid of the form $\delta^{\pm}_{[l,m]}=(\sigma_{m-1}\sigma_{m-2}\cdots\sigma_{l})^{\pm}$ (called standardizing braid) such that $(\tau_{1},f_{\beta}\circ h_{1})$ is standard. 
\end{lemma}
In this context we say that $f_{\beta}$ is a {\em standardizing homeomorphism} for $(\tau_{1},h_{1})$.

\begin{definition}
\label{def:standardizing}
Any infinitesimal edge around a puncture determines a cusp (enclosing a puncture).
Any standardizing homeomorphism $f_{\beta}$ acts on those edges by a permutation $\pi\in \mathbb{A}_{\mathrm{prong}}$.
Moreover if $e,f$ are two infinitesimal edges (with labeling $\alpha,\beta$ respectively) and if $\pi(\alpha)=\beta$ then $f_\beta$ 
acts as a rotation whose support is contained in a small neighborhood of the punctures. We encode this action by the rotation number
$k\in \Z$ (under the convention that a counterclockwise rotation has positive sign) and we will use 
the notation $\pi(\alpha)=\beta^k$.
\end{definition}

\begin{example}
\label{ex:standardizing}
In Figure~\ref{isotopy} the two standardizing homeomorphisms corresponding to $\sigma^{-1}_1$ and
$\sigma_2$ act on the punctures as follows (we identify punctures and infinitesimal edges enclosing punctures for the labelling):
$$
\pi(\sigma_2) = \begin{pmatrix} A & B & C \\ A & C^{+1} & B \end{pmatrix} \qquad \textrm{and} \qquad \pi(\sigma^{-1}_1) 
= \begin{pmatrix} A & B & C \\ B&A^{-1}&C\end{pmatrix}.
$$
In particular, for any $k\in\Z$, $k>0$, the permutation associated to $\sigma_2^{2k}$ is $\pi(\sigma_2^{2k})=\begin{pmatrix} A & B & C \\ A & B^{k} & C^k \end{pmatrix}$.
\end{example}

\section{Construction of the automaton}
\label{sec:automaton}

In this section, for simplicity, we specify to the case of the punctured disc. However all 
the discussion can be done for surfaces of higher genera. Let us fix $n>2$, the number of 
punctures, and the singularity data of train tracks ({\em i.e.} the number and type of prongs).
We fix an alphabet~$\mathbb{A}$.

\subsection{Graphs of foldings}
	\label{SUBSEC:GraphFolds}
We start with the following observation: the number of 
labelled train tracks $(\tau,h,\varepsilon)$ of $\mathbf{D}_{n}$ where
\begin{itemize}
\item $(\tau,h)$ is standard,
\item $\tau$ has prescribed singularity data and labelling $\varepsilon : E(\tau) \rightarrow \mathbb{A}$
\end{itemize}
(up to isotopy of $\mathbf{D}_{n}$ fixing the punctures) is finite. 

Moreover this set is also (set-wise) invariant by folding operations followed by standardness operations. Finally, given a tuple 
$(\tau,h,\varepsilon)$ into this finite set, since the number of cusps and edges is finite, there are only finitely many possible 
folding operations on $h(\tau)$. These three finiteness ingredients allow us to construct a graph  in the following way.
\begin{enumerate}
\item Vertices are tuples $(\tau,h,\varepsilon_{|E(\tau)_{\mathrm{real}}})$ where $h: \tau \rightarrow \mathbf{D}_{n}$ is standard (up to isotopy
fixing the punctures).
\item There is an edge between $(\tau_1,h_1,\varepsilon_1)$ and $(\tau_2,h_2,\varepsilon_2)$ if there 
is a folding operation from $(\tau_1,h_1,\varepsilon_1)$ to $(\tau_2,h'_1,\varepsilon_2)$ and either:
\begin{enumerate}
\item $h'_1(\tau_2)$ is standard: then $h_2=h'_1$, or
\item $h'_1(\tau_2)$ is almost standard: $h_2=f_{\beta}\circ h'_1$ where $f_\beta$ is a standardizing braid.
\end{enumerate}
\item There is an edge between $(\tau,h,\varepsilon)$ and $(\tau,f_{\beta}\circ h,\varepsilon)$ 
where $\beta\in B_n$ and $f_{\beta}\circ h(\tau)$ is also standard.
\end{enumerate}

The resulting directed graph is called the {\em folding automaton} associated to the number of marked points of $\mathbf{D}_{n}$ and the
type of singularities. Observe that this graph is not necessarily strongly connected (even not connected). It would be 
nice to have a description of the topology of these graphs in general. \medskip

For any train track $(\tau,h,\varepsilon)$ we will denote by $\mathcal{N}^{\mathrm{lab}}(\tau,h,\varepsilon)$ the  
connected component of the folding automaton containing $(\tau,h,\varepsilon)$.
One can also perform the same construction without labelling: the connected components containing $(\tau,h)$ are then 
denoted by $\mathcal{N}(\tau,h)$.

\begin{example}
See Figures~\ref{fg0}-\ref{fg6}-\ref{fig11} for examples of automata.
\end{example}

\subsection{Closed loops in $\mathcal{N}(\tau,h)$ and pseudo-Anosov homeomorphisms}
	\label{SS:ClosedLoopsAutomaton}

The labelling allows us to define for each edge of $\mathcal{N}^{\mathrm{lab}}(\tau,h,\varepsilon)$ a train track map and its associate
transition matrix. Hence given a path $\eta$ in $\mathcal{N}^{\mathrm{lab}}(\tau,h,\varepsilon)$ (not necessarily closed) represented by
$$
(\tau_{1},\varepsilon_1)\stackrel{T_{1}}\longrightarrow(\tau_{2},\varepsilon_{2})\stackrel{T_{2}}{\longrightarrow}\cdots
\stackrel{T_{k-1}}\longrightarrow(\tau_{k},\varepsilon_{k}) \stackrel{T_k}\longrightarrow(\tau_{1},\varepsilon_{k+1})
$$
one defines the matrix $M(\eta)$ by using the formula~\eqref{E:pross}:
$$
M(\eta) = M(T_{k}\circ T_{k-1}\circ\cdots\circ T_{2}\circ T_{1}) = M(T_{1})M(T_{2})\cdots M(T_{k}).
$$
Now if $\gamma$ is a {\em loop} in $\mathcal{N}(\tau,h)$ starting at some point $(\tau_i,h_i)$ we can lift $\gamma$ to some 
path $\hat{\gamma}$ in $\mathcal{N}^{\mathrm{lab}}(\tau,h,\varepsilon)$ starting at $(\tau_i,h_i,\varepsilon_i)$. Here $\varepsilon$
is any labelling of $(\tau,h)$. 
The end point of $\hat{\gamma}$ (that is $(\tau_i,h_i,\varepsilon'_i)$) defines a train track map
$$
R : (\tau_i,h_i,\varepsilon'_i) \longrightarrow (\tau_i,h_i,\varepsilon_i).
$$
The associated matrix $M(R)\in \mathrm{GL}(\Z^\mathbb{A})$ is induced by a permutation, namely
$R_{\alpha,\beta}=1$ if $\pi(\alpha)=\beta$ and $0$ otherwise, where $\pi=\varepsilon'_i\circ(\varepsilon_i)^{-1}\in{\rm Sym}(\mathbb{A})$.
We then define:
$$
M(\gamma) := M(\hat{\gamma})\cdot M(R).
$$
Obviously the conjugacy class of the matrix $M(\gamma)$ does not depend on the choice of the labelling $\varepsilon$.

\begin{remark}
The above discussion allows us to reformulate Theorem~\ref{thm:PaPe} and Lemma~\ref{L:factor} as follows: 
any pseudo-Anosov homeomorphism is obtained from a closed loop in some graph $\mathcal{N}(\tau,h)$ by using the above 
construction. The converse is almost true: this is Theorem~\ref{thm:converse}.
\end{remark}

We end this section with a useful description of the train track map representing lift of 
homeomorphism to $\widetilde{\mathbf{D}_{n}}$.

\subsection{Lifting train tracks} 
\label{SLEO}
We denote by $\rho:\widetilde{\mathbf{D}_{n}}\to\mathbf{D}_{n}$ the $H=\Z^{b_{1}(M)-1}$-covering of the punctured disc 
(see Section~\ref{SS:TPF}). The infinite surface $\widetilde{\mathbf{D}_{n}}$ can be constructed by glueing $H$ copies of the simply 
connected domain obtained by cutting the base $\mathbf{D}_{n}$ along $n$ disjoint segments from the punctures to the exterior 
boundary. The way one should glue is dictated by the monodromy of the covering. 
We call each of these simply connected domains a \emph{leaf} of the covering $\rho:\widetilde{\mathbf{D}_{n}}\to\mathbf{D}_{n}$.
Henceforth, we choose a leaf in $\widetilde{\mathbf{D}_{n}}$ and we label it with $e_{H}$, the identity element in $H$. 
We call it \emph{the leaf at level zero}. \medskip

For each standard $(\tau,h,\varepsilon)$, there is a natural way to define an infinite train track $\widetilde{h}:\widetilde{\tau}\to\widetilde{\mathbf{D}_{n}}$
by $\widetilde{h}(\widetilde{\tau}):=\rho^{-1}(h(\tau))$. The edges and vertices of $\widetilde{\tau}$ are in bijection with 
$E(\tau)\times H$ and $V(\tau)\times H$ respectively and there are several ways to label the edges of $\widetilde{\tau}$. \medskip

Every permutation $\eta\in{\rm Sym}(\mathbb{A}_{\mathrm{prong}})$ defines a labeling of the edges of  $\widetilde{\tau}$ as follows. For every edge $\mathbf{e}$ of 
$\widetilde{\tau}$ whose image under $\widetilde{h}$ is properly contained in the leaf at level zero we define $\widetilde{\varepsilon}(\mathbf{e})=\varepsilon(e)$, where $\rho(\mathbf{e})=e$. Now by the way we defined the leaves of the covering, and given that we are working with standardly embedded train tracks, there are exactly $2n$ edges of $\widetilde{\tau}$ whose image under 
$\widetilde{h}$  are not properly contained in the leaf of level zero. Moreover, these edges can be grouped in pairs $\{e^{1},e^{2}\}$
where $\rho(e^{1})=\rho(e^{2})=e$, and $e$ is an infinitesimal edge of $\tau$ around a puncture. For every such edge $e$
we define $\widetilde{\varepsilon}(e^{1})=\eta(\varepsilon(e))$ where $\varepsilon(e)\in \mathbb{A}_{\mathrm{prong}}$. 
Finally, we extend $\widetilde{\varepsilon}$ to the remaining edges of $\widetilde{\tau}$  by using the $H$-monodromy action of the covering.

\subsection{Lifting train track maps}
\label{sub:sec:lift}
Let us consider $(\tau,h)$ and $(\tau',h')$ two train tracks in $\mathbf{D}_n$ and let $T:\tau\longrightarrow\tau'$ be a train track map representing a class $[f]\in{\rm Mod}(\mathbf{D}_{n})$. Now consider: $\widetilde{f}:\widetilde{\mathbf{D}_{n}}\to\widetilde{\mathbf{D}_{n}}$ a lift of  $[f]$ to the $H=\Z^{b_{1}(M)-1}$-covering of the punctured disc $\rho:\widetilde{\mathbf{D}_{n}}\to\mathbf{D}_{n}$, $(\hspace{.5mm}\widetilde{\tau},\widetilde{h})$ and $(\widetilde{\tau'},\widetilde{h'})$ lifts of $(\tau,h)$ and $(\tau',h')$ to this covering respectively. As with finite train tracks, a cellular map $\widetilde{T}:\widetilde{\tau}\to\widetilde{\tau'}$ that preservers the smooth structure will be called a \emph{train track morphism}. If in addition the domain and image train tracks of the morphism are isomorphic as train tracks we speak of a \emph{train track map}. A train track morphism $\widetilde{T}:\widetilde{\tau}\to\widetilde{\tau'}$ is a representative of the lift $\widetilde{f}$ if:
\begin{enumerate}
\item The diagram
$$
\begin{CD}
\widetilde{\tau} @>\widetilde{h}>> \widetilde{\mathbf{D}_{n}}\\
@VV\widetilde{T} V @VV \widetilde{f} V\\
\widetilde{\tau'} @>\widetilde{h'}>> \widetilde{\mathbf{D}_{n}}
\end{CD}
$$
commutes, up to isotopy, and 
\item $\widetilde{f}\circ \widetilde{h}(\hspace{.5mm}\widetilde{\tau}\hspace{.5mm}) \subset N(\widetilde{h'}(\widetilde{\tau'}))$ and $\widetilde{f}\circ \widetilde{h}(\hspace{.5mm}\widetilde{\tau}\hspace{.5mm})$ is transverse to the tie foliation of $\widetilde{h'}(\widetilde{\tau'})$.
\end{enumerate}
It is clear that for every lift $\widetilde{f}$ of  $f$ there is a train track map representing it. 

Let $\eta\in{\rm Sym}(\mathbb{A}_{\mathrm{prong}})$ be any permutation and $\pi\in{\rm Sym}(\mathbb{A}_{\mathrm{prong}})$ be the permutation defined by $f$. These permutations define labelings $(\widetilde{\tau},\widetilde{h},\widetilde{\varepsilon})$ and $(\widetilde{\tau'},\widetilde{h'},\widetilde{\varepsilon'})$ respectively, 
by $\eta$ and $\pi\circ \eta$. As in the case of finite train tracks we can associate to the train track map $\widetilde{T}:\widetilde{\tau}\to\widetilde{\tau'}$ representing $\widetilde{f}$ an incidence matrix. 
The matrix $M(\widetilde{T}) \in \mathrm{GL}(\Z[H]^{\mathbb{A}})$ records  how the edges of $\widetilde{f}\circ \widetilde{h}(\hspace{.5mm}\widetilde{\tau}\hspace{.5mm})$ intersect the central ties of  $\widetilde{h'}(\widetilde{\tau'})$. Obviously by construction one has 
$M(\widetilde{T}) = M(T)\cdot \mathrm{Diag}(v)$ for a suitable vector $v\in\Z[H]^{\mathbb{A}}$. 
In the next section we explain how to compute this vector in the particular situation where $T: \tau\longrightarrow\tau'$ 
is an edge of the folding automaton.

\section{Computing the Teichm\"uller polynomial}
\label{sec:main:theo}

In this section we shall prove our main result. The statement uses what we call the 
{\em decorated folding automaton}. 
The idea is to enrich the folding automaton by adding additional  information to each of its edges
so that the computation of the Teichm\"uller polynomial can be carried out using just the decorated folding automaton. 
This represents a simplification of the problem of computing $\Theta_{F}$, for with the method we propose there is no need to pass to an abelian infinite 
cover.
\subsection{The decorated folding automaton}
In the next paragraphs we define the extra piece of information needed to obtain the decorated folding automaton. Roughly speaking, this extra piece of information is a vector $v$ with entries in $\Z[H]$ that encodes the incidence matrix $M(\widetilde{T})$ of the lift of a train track map $T$ coming from a folding operation, see \S \ref{sub:sec:lift}.

Recall that when defining the folding automaton in \S \ref{SUBSEC:GraphFolds}, the labelling map $\varepsilon$ in $(\tau, h, \varepsilon)$ is restricted to 
the set of real edges $E(\tau)_{\rm real}$ of $\tau$. We will often choose the convention that, for {\em any} train track, the infinitesimal edges enclosing 
punctures are labeled by $\{A,B,C,\ldots\}=\mathbb{A}_{\rm prong}$ 
where the alphabetical order is set to match the order on the punctures of $\mathbf{D}_n$ induced by the natural order of $\R$. \medskip

Let $(\tau,h,\varepsilon)\stackrel{T}\longrightarrow(\tau',h',\varepsilon')$  be a train track map 
associated to an edge in the folding automaton which corresponds to a {\em folding operation} $F$ and that represents a standardizing homeomorphism $f_\beta$, where $\beta$ is a braid in $B_n$. If $h'(\tau')$ is standardly embedded we say that \emph{the folding $F$ is standard}. For every edge in the folding automaton corresponding to a standard folding we define $v\in \Z[H]^{\mathbb{A}_{\rm real}}$ as the constant vector on which each entry is equal to 1.

Henceforth we assume that the folding $F$ is \emph{not standard}.
There are two real edges $\{e,e'\}\subset E(\tau)$ and three vertices $\{v_0,v_1,v_2\}\subset V(\tau)$ involved when performing  $F$. We observe that:
\begin{enumerate}
\item there is a unique edge $f\in\{e,e'\}$ in $f_\beta(h(\tau))$ which is not properly embedded, \emph{i.e.} that traverses to the lower half of the punctured disc $\mathbf{D}_n$,
\item there exists a unique vertex $v_{fix}\in\{v_0,v_1,v_2\}$ which is fixed by $f_\beta$, and
\item after performing the folding operation on $F$, a new cusp in $(\tau',h',\varepsilon')$ appears. This cusp is incident to a vertex $v_{end}\in\{v_0,v_1,v_2\}$.
Let $X\in\mathbb{A}_{\rm prong}$ be the label of the unique infinitesimal edge of $\tau$ enclosing a puncture that is incident to $v_{end}$.
\end{enumerate}
%
\begin{definition}
	\label{DEF:VECTV}
We denote by $N(T)$ the connected component of $\tau\setminus f$ which \emph{does not} contain the vertex $v_{fix}$
(possibly $N(T)=\emptyset$). We define $f'\in\{e,e'\}$ by $f'\neq f$. There are two cases to consider:
\begin{itemize}
\item \textbf{Case 1}: $f'\notin N(T)$. We define $v\in \Z[H]^{\mathbb{A}_{\rm real}}$ as:
$$
 \left\{
\begin{array}{l}
v_{\alpha}= X^{\pm 1}\hspace{.5mm}\textrm{if $\varepsilon(e)=\alpha\in\mathbb{A}_{\rm real}$ and $e\in N(T)\cup f$}, \textrm{ and} \\
v_{\alpha}=1 \hspace{2mm}\textrm{otherwise.}
\end{array}
\right.
$$
\item \textbf{Case 2}: $f'\in N(T)$. We define $v\in \Z[H]^{\mathbb{A}_{\rm real}}$ as:
$$
 \left\{
\begin{array}{l}
v_{\alpha}= X^{\pm 1}\hspace{.5mm}\textrm{if $\varepsilon(e)=\alpha\in\mathbb{A}_{\rm real}$ and $e\in N(T)$}, \textrm{ and} \\
v_{\alpha}=1 \hspace{2mm}\textrm{otherwise.}
\end{array}
\right.
$$
\end{itemize}
\end{definition}

The sign of the exponent in $X^{\pm 1}$ is determined by the choice of the counterclockwise direction as positive direction for rotations on the disc. 
The {\em decorated folding automaton} $\mathcal{N}^{\mathrm{aug}}(\tau,h)$ is $\mathcal{N}^{\mathrm{lab}}(\tau,h,\varepsilon)$
where we add the information $(\pi,v)$ at each edge.
\begin{remark}
	\label{Rem:vectorv}
A priori one would expect the vector $v$ encoding the matrix $M(\widetilde{T})$ to be larger, that is, $v\in \Z[H]^{\mathbb{A}}$. 
However, as we will see in the next section, the contribution of infinitesimal edges to the determinant formula~\ref{E:TPform} cancels out 
with the contribution of the matrix $P_V(t)$, and hence one can restrict the computation of the Teichm\"uller polynomial to the subset of 
$\mathbb{A}_{\rm real}\subset \mathbb{A}$ formed by real edges.
\end{remark}

If there is an edge between $(\tau,h,\varepsilon)$ and $(\tau,f_{\beta}\circ h,\varepsilon)$ where 
$\beta\in B_n$ and $f_{\beta}\circ h(\tau)$ is also standard then the vector $v\in  \Z[H]^{\mathbb{A}}$
is defined by $v=(X^{\pm 1},X^{\pm 1},\ldots,X^{\pm 1})$ depending the orientation of the braid $\beta$, and
$X$ is the label associated to the first, respectively last, prong of $\mathbf{D}_n$.

\begin{example}
We consider the folding automaton for the train track depicted in  Figure~\ref{fig:decorated:B3}. We have depicted only real edges. Infinitesimal edges enclosing punctures are labeled by $\{A,B,C\}=\mathbb{A}_{\rm prong}$, where the alphabetical order is set to match the order on the punctures of $\mathbf{D}_3$ induced by the natural order of $\R$. This automaton has two edges. To the right the one corresponding to the non standard folding $F_{ab}$ that folds the real edge labeled with $a$ over the real edge labeled with $b$. The standardizing homeomorphism in this case is given by $f_{\sigma_2}$ and a direct computation shows that:
\begin{itemize}
\item $v_{\rm fix}$ is the vertex on which the infinitesimal edge $A$ is incident and $v_{end}$ is the vertex on which the infinitesimal edge $C$ is incident. Therefore the label of the unique infinitesimal edge of $\tau$ enclosing a puncture that is incident to $v_{end}$ is given by $X=C$.
\item The edge $f$ in definition  \ref{DEF:VECTV} is the edge labeled with $a$. Therefore $N(T_{ab})$ is the graph containing the infinitesimal edges $B$ and $C$, the vertices to which they are incident and the real edge $b$. Therefore, according to definition \ref{DEF:VECTV},  we are in case 2 .

\end{itemize}
Hence we deduce that the vector corresponding to this edge of the automaton is given by $v_{ab}=(1,C^{+1})$. We leave to the reader the rest of the computations. To deduce the signed permutations corresponding to the edges of the automaton it is useful to look at Figure~\ref{isotopy}.
 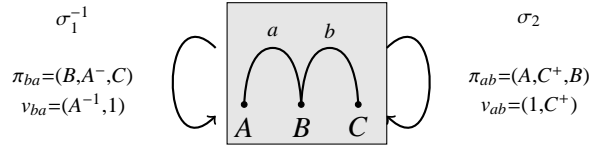
\begin{figure}[htbp]
 \begin{tikzpicture}[scale=0.75, inner sep=1mm]
\filldraw[fill=white!90!black,smooth]
(-1.3,-.7) rectangle (1.5,1.8);

\foreach \x in {4.5,3.5,2.5}
	\draw [xshift=3.5cm, fill=black] (-\x,0) circle (0.05);
\draw [thick, xshift=3.5cm,smooth] plot [tension=1.75] coordinates{(-4.5,0) (-4,1) (-3.5,0)};	
\draw [thick, xshift=3.5cm,smooth] plot [tension=1.75] coordinates{(-3.5,0) (-3,1) (-2.5,0)};

\draw (-.5,1.3) node{$\scriptstyle a$};
\draw (.5,1.3) node{$\scriptstyle b$};
\draw (-1,-.4) node{$A$};
\draw (0,-.4) node{$B$};
\draw (1,-.4) node{$C$};

\draw (-4,1.5) node{$\scriptstyle \sigma_1^{-1}$};
\draw (-4,0.5) node{$\scriptstyle \pi_{ba}=(B,A^{-},C)$};
\draw (-4,0) node{$\scriptstyle v_{ba}=(A^{-1},1)$};

\draw (4,1.5) node{$\scriptstyle \sigma_2$};
\draw (4,0.5) node{$\scriptstyle \pi_{ab}=(A,C^{+},B)$};
\draw (4,0) node{$\scriptstyle v_{ab}=(1,C^+)$};
     
\draw[thick,->] (-1.5,1) .. controls  +(-1,1) and +(-1,-1) .. (-1.5,-0.2);
\draw[thick,->] (1.5,1) .. controls  +(1,1) and +(1,-1) .. (1.5,-0.2);

 \end{tikzpicture}
 \caption{
 \label{fig:decorated:B3}
The {\em decorated folding automaton} for $B_{3}$. The two edges represent $\sigma^{-1}_1$ (left) and $\sigma_2$ (right).}
\end{figure}
\end{example}

\subsection{Main result}
A pseudo-Anosov class $[f_\beta]\in \mathrm{Mod}(\mathbf D_n)$ leaving invariant 
a train track $(\tau,h)$ defines a map
\begin{equation}
	\label{E:tmap}
t:\mathbb{A}_{\mathrm{prong}}\to H=\mathrm{Hom}(H^1(\mathbf D_n,\Z)^{f_\beta},\Z)	
\end{equation}
in the following way. We have a collection of cycles $s_\alpha=[\partial U_\alpha]$, $\alpha\in \mathbb{A}_{\mathrm{prong}}$, that form a basis for 
$H_1(\mathbf{D}_n,\Z)$. Moreover $f_{\beta}$ acts on this basis, indeed for every $\alpha\in \mathbb{A}_{\mathrm{prong}}$ we have 
$f_\beta(s_\alpha)=s_{\beta(\alpha)}$, where $\beta \in \mathrm{Sym}(\mathbb{A}_{\mathrm{prong}})$ is the permutation that  $\beta$ defines  on the punctures. 
For each cycle $\sigma$ of $\beta$ let $t_\sigma = \sum_{\alpha\in \mathrm{Supp}(\sigma)} s_\alpha$. This defines a multiplicative basis for $H$, that we denote by $t_1,\dots,t_r$ for simplicity. The map $t:\mathbb{A}_{\mathrm{prong}}\to H$ is given by $t(\alpha)=t_\sigma$ provided that $\alpha\in \mathrm{Supp}(\sigma)$.

\begin{convention}
Let $\mathbb{A}_{\rm prong}^{\pm 1}:=\{\alpha^{+1},\alpha^{-1}\}_{\alpha\in\mathbb{A}_{\rm prong}}$. We extend the function $t:\mathbb{A}_{\mathrm{prong}}\to H$ (respectively, $\pi \in \mathrm{Sym}(\mathbb{A}_{\mathrm{prong}})$)
to a function $\{1\}\cup \mathbb{A}_{\mathrm{prong}}^{\pm 1}\to H$ (respectively, permutation of $\{1\}\cup \mathbb{A}_{\mathrm{prong}}^{\pm 1}$) 
by $t(\alpha^{\pm 1})=t(\alpha)^{\pm 1}$ if $\alpha \in\mathbb{A}_{\mathrm{prong}}$ and $t(1)=1$
(respectively, $s(\alpha^{\pm 1})=s(\alpha)^{\pm 1}$ and $s(1)=1$). \medskip

Finally, if $v$ is a vector with entries in  $\{1,\alpha^{\pm 1}\hspace{1mm}|\hspace{1mm} \alpha\in \mathbb{A}_{\mathrm{prong}}\}$, 
we define $t(v)$ (respectively, $\pi(v)$) as the vector that results from the evaluation $t$ (respectively, $\pi$) on each coordinate. 
\end{convention}

\begin{theorem}
\label{thm:second:main}
Let $[f_\beta]$ be a pseudo-Anosov class given by the loop
$$
(\tau_{1},\varepsilon_1)\stackrel{T_{1}}\longrightarrow(\tau_{2},\varepsilon_{2})\stackrel{T_{2}}{\longrightarrow}\cdots
\stackrel{T_{k-1}}\longrightarrow(\tau_{k},\varepsilon_{k}) \stackrel{T_{k}}\longrightarrow(\tau_{1},\varepsilon_{k+1})\stackrel{R}\longrightarrow(\tau_{1},\varepsilon_1)
$$
in the decorated folding automaton. We assume that the matrix describing the linear map on the space of weights on real
edges is primitive irreducible. Then the Teichm\"uller polynomial $\Theta_{F}(t,u)$ of the associated fibered face $F$ 
determined by $[f_\beta]$ is:
\begin{equation}
	\label{E:DetFormulaTh}
\Theta_{F}(t_1,\dots,t_r,u)=\det(u\cdot \Id-M)
\end{equation}
where
\begin{equation}
	\label{E:ProdFormulaTh}
M = M(T_{1})D_1 \cdot  M(T_{2}) D_{2} \cdots M(T_{k})D_k \cdot M(R).
\end{equation}
and, for $i=1,\ldots,k$:
\begin{equation}
\left\{
\label{E:DiagFormulaTh}
\begin{array}{l}
D_i= {\rm Diag}(t(w_i)) \in \mathrm{GL}(\Z[H]^{\mathbb{A}_{\mathrm{real}}}), \\
w_i=\eta_i(v_i),\\
\eta_1=\mathrm{Id}_{\mathbb{A}_{\mathrm{prong}}} \textrm{and for } i\geq 2:\ \eta_i=\pi_{i-1}\circ \eta_{i-1}.
\end{array}
\right.
\end{equation}

\end{theorem}

\begin{proof}[Proof of Theorem~\ref{thm:second:main}]
We first observe that the assumption on real edges implies that the contribution of infinitesimal edges to the numerator in the determinant 
formula~\eqref{E:TPform} cancels out with the denominator. This fact, together with the discussion we did in sections \ref{SS:ClosedLoopsAutomaton}-\ref{sub:sec:lift}, imply formulas  (\ref{E:DetFormulaTh}) and (\ref{E:ProdFormulaTh}). We need then to show that each diagonal matrix $D_i$ is given by formula (\ref{E:DiagFormulaTh}). 
We will prove first that given an edge 
\begin{equation}
	\label{E:edgeatproof}
	(\tau_i,\varepsilon_i,h_{i})\stackrel{T_i}{\to}(\tau_{i+1},\varepsilon_{i+1},h_{i+1})
\end{equation}
of the decorated  automaton representing a standardizing homeomorphism $f_{\beta}$, the incidence matrix $M(\widetilde{T_i})$ with entries in $\Z[H]$ associated to a lift of $T_i$ to $\widetilde{\mathbf{D}_n}$ is of the form  $M(T_i){\rm Diag (t(\eta_i(v_i))}$. This is the longest part of the proof and is done by cases that depend on the train track $(\tau_i,\varepsilon_i,h_{i})$.  We deal with the recursive nature of formula (\ref{E:DiagFormulaTh}) at the end. In order to present the list of cases, 
we observe that
\begin{enumerate}
\item \emph{No train track of the decorated folding automaton defines a polygon whose sides are real edges}. More precisely, if $(\tau,h,\varepsilon)\in\mathcal{N}^{\mathrm{lab}}(\tau,h,\varepsilon)$, then there is no connected component of $\mathbf{D}_n\setminus h(\tau)$ homeomorphic to a disc whose boundary is formed by real edges of $\tau$.
\item \emph{All standardizing homeomorphisms are 'simple' in the sense of Ko-Los-Song \cite{KLS}}. More precisely, let $\delta_{[l,m]}:=\sigma_{m-1}\sigma_{m-2}\cdots\sigma_l$, where the $\sigma_i$'s are the standard Artin generators of the braid group $B_n$. Then we can suppose that all the homeomorphism used to standardize train tracks in the augmented folding automaton are of the form $\delta^{\pm 1}_{[l,m]}$. 
\end{enumerate}

It will be useful to understand the proof to describe more precisely the covering $\widetilde{\mathbf{D}_{n}}$ discussed in section (\ref{SLEO}).
Let $D$ be the leaf of this covering 
obtained by cutting the base $\mathbf{D}_n$ along n disjoint vertical segments that go from the punctures to the lower part of the exterior boundary. Any labelling by elements of $\mathbb{A}_{\mathrm{prong}}$ of the infinitesimal edges surrounding punctures defines a natural labelling of these vertical segments. Let us denoted them by 
$\iota_\alpha$ where 
$\alpha\in \mathbb{A}_{\mathrm{prong}}$. The infinite covering $\widetilde{\mathbf{D}_{n}}$ is obtained by glueing the disjoint family of copies of $D$ in the family $\{D_h\}_{h\in H}$ as follows: for every $h\in H$, crossing in $D_h$ the segment 
$\iota_\alpha$ in the counter clockwise direction takes you to $D_{t(\alpha)h}$, where $t$ is the map defined in~\eqref{E:tmap}. From this detailed description we deduce that if the train track map corresponding to the edge (\ref{E:edgeatproof}) comes from a standard folding $F_i$ then $M(\widetilde{T_i})=M(T_i)$. Indeed, it is sufficient to remark that no real edge of $\widetilde{f_{\beta}}(\widetilde{\tau_i})$ intersects a vertical segment $\iota_\alpha$.

Now let us suppose that the edge (\ref{E:edgeatproof}) comes from a non standard folding $F_i$. We justify in detail the equality $M(\widetilde{T_i})=M(T_i){\rm Diag (t(\eta_i(v_i)))}$ in two illustrative cases. Then we explain how to proceed with all the cases that remain, see Appendix \ref{APPENDIX:CASES}.\\
\\
\textbf{Case A.1}. This case is formed by an infinite family of train tracks arising from a graph $\Gamma$ embedded in $\mathbf{D}_{n}$. This graph $\Gamma$ will be called a \emph{basic type} and consists of: two real edges $\{e,e'\}$, three vertices $\{v_{0},v_{1},v_{2}\}$ and at most three infinitesimal edges, each of which encloses a puncture of $\mathbf{D}_{n}$ and is incident to a vertex in $\{v_{0},v_{1},v_{2}\}$. 
For the particular situation of case A.1, the graph $\Gamma$ is depicted in figure \ref{fig:basictype}. 
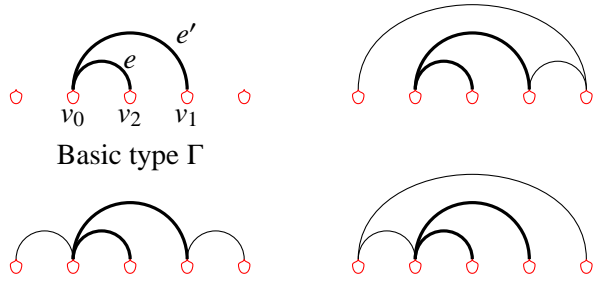
\begin{figure}[htbp]
 \begin{tikzpicture}[scale=0.75, inner sep=1mm]

\foreach \x in {-5,-4,-3,-2,-1}
	\draw[smooth, red] (\x,0) to[out=270,in=90] (\x-.1,-.1) to[out=270,in=180] (\x,-.25) to[out=0,in=270] (\x+.1,-.1) to[out=90,in=270] (\x,0);
\draw[smooth, very thick] (-4,0) to[out=90,in=180] (-3,1) to[out=0,in=90] (-2,0); 	
\draw[smooth, very thick] (-4,0) to[out=90,in=180] (-3.5,.5) to[out=0,in=90] (-3,0); 


\draw(-3,.5) node{$e$};
\draw(-2,1) node{$e'$};
\draw(-4,-.5) node{$v_{0}$};
\draw(-3,-.5) node{$v_{2}$};
\draw(-2,-.5) node{$v_{1}$};
\draw(-3,-1.2) node{Basic type $\Gamma$};

\foreach \x in {-5,-4,-3,-2,-1}
	\draw[xshift=6cm,smooth, red] (\x,0) to[out=270,in=90] (\x-.1,-.1) to[out=270,in=180] (\x,-.25) to[out=0,in=270] (\x+.1,-.1) to[out=90,in=270] (\x,0);
\draw[xshift=6cm,smooth, very thick] (-4,0) to[out=90,in=180] (-3,1) to[out=0,in=90] (-2,0); 	
\draw[xshift=6cm,smooth, very thick] (-4,0) to[out=90,in=180] (-3.5,.5) to[out=0,in=90] (-3,0); 
\draw[smooth] (1,0) to[out=90,in=180] (3,1.5) to[out=0,in=90] (5,0); 
\draw[smooth] (4,0) to[out=90,in=180] (4.5,.5) to[out=0,in=90] (5,0);

\foreach \x in {-5,-4,-3,-2,-1}
	\draw[xshift=6cm,yshift=-3cm,smooth, red] (\x,0) to[out=270,in=90] (\x-.1,-.1) to[out=270,in=180] (\x,-.25) to[out=0,in=270] (\x+.1,-.1) to[out=90,in=270] (\x,0);
\draw[xshift=6cm,yshift=-3cm,smooth, very thick] (-4,0) to[out=90,in=180] (-3,1) to[out=0,in=90] (-2,0); 	
\draw[xshift=6cm,yshift=-3cm,smooth, very thick] (-4,0) to[out=90,in=180] (-3.5,.5) to[out=0,in=90] (-3,0); 
\draw[smooth,yshift=-3cm] (1,0) to[out=90,in=180] (3,1.5) to[out=0,in=90] (5,0); 
\draw[smooth, yshift=-3cm] (1,0) to[out=90,in=180] (1.5,.5) to[out=0,in=90] (2,0);

\foreach \x in {-5,-4,-3,-2,-1}
	\draw[yshift=-3cm,smooth, red] (\x,0) to[out=270,in=90] (\x-.1,-.1) to[out=270,in=180] (\x,-.25) to[out=0,in=270] (\x+.1,-.1) to[out=90,in=270] (\x,0);
\draw[yshift=-3cm,smooth, very thick] (-4,0) to[out=90,in=180] (-3,1) to[out=0,in=90] (-2,0); 	
\draw[yshift=-3cm,smooth, very thick] (-4,0) to[out=90,in=180] (-3.5,.5) to[out=0,in=90] (-3,0); 
\draw[smooth, yshift=-3cm] (-5,0) to[out=90,in=180] (-4.5,.5) to[out=0,in=90] (-4,0); 
\draw[smooth,xshift=3cm, yshift=-3cm] (-5,0) to[out=90,in=180] (-4.5,.5) to[out=0,in=90] (-4,0); 

\end{tikzpicture}
 \caption{
 \label{fig:basictype}
Basic type A.1 and some train tracks that arise from it.}
\end{figure}


The idea now is to add edges and vertices to $\Gamma$ to form vertices $(\tau_i,\varepsilon_i,h_{i})$ of the folding automaton in $\mathbf{D}_{n}$. There are many ways to do this, some of which are depicted in figure \ref{fig:basictype}.  In fact, given that the vertices of the automaton are properly embedded train tracks, \emph{there are only finitely many types} of train tracks that can be obtained this way. To illustrate what we mean by a \emph{type of a train track} we display in figure \ref{fig:typesA.1} all possible types of train tracks arising from the basic type $\Gamma$ in figure \ref{fig:basictype} \footnote{For this particular case we depict $\Gamma$ embedded in $\mathbf{D}_{4}$.}. In figure \ref{fig:typesA.1}, each one of the small boxes represents a subgraph of $(\tau_i,\varepsilon_i,h_{i})$.\\
Now let us consider an edge of the automaton (\ref{E:edgeatproof}) where $(\tau_{i},h_{i},\varepsilon_{i})$ is the  type A.1.1 depicted in figure \ref{fig:typesA.1} and the train track map $T_{i}$ represents the homeomorphism $f_{\beta}$ that standardizes the train track $F_{ee'}(\tau_{i})$ that arises when folding edge $e$ over edge $e'$. In figure \ref{fig:edgeautomaton} we depict this edge of the automaton in detail. The numbers $0,1,2$ in this figure represent vertices $\{v_{0},v_{1},v_{2}\}$ respectively. Remark that for this edge of the automaton we have $v_{\rm fix}=v_{0}$ and $v_{\rm end}=v_{1}$. Moreover, edges $f$, $f'$ from definition 
\ref{DEF:VECTV} are given by $f=e$, $f'=e'$ and the subgraph $N(T_{i})$ is highlighted in bold. We observe that $f'\notin N(T_{i})$, hence we are in case  1 presented in the same definition.


\begin{figure}[htbp]
 \begin{tikzpicture}[scale=0.75, inner sep=1mm]

	\filldraw[xshift=-3.2cm,yshift=-.2cm,fill=white!90!black,smooth] (0,0) rectangle (0.4,0.4);
	\filldraw[xshift=-4.2cm,yshift=-.2cm,fill=white!90!black,smooth] (0,0) rectangle (0.4,0.4);
	\filldraw[xshift=-5.2cm,yshift=-.2cm,fill=white!90!black,smooth] (0,0) rectangle (0.4,0.4);	
	\filldraw[xshift=-2.2cm,yshift=-.2cm,fill=white!90!black,smooth] (0,0) rectangle (0.6,0.4);	
\draw[smooth, very thick] (-4,0) to[out=90,in=180] (-3,1) to[out=0,in=90] (-2,0); 	
\draw[smooth, very thick] (-4,0) to[out=90,in=180] (-3.5,.5) to[out=0,in=90] (-3,0); 
\draw[smooth] (-5,0) to[out=90,in=180] (-3.5,1.5) to[out=0,in=90] (-1.8,0); 

	\filldraw[xscale=-1,xshift=-3.2cm,yshift=-.2cm,fill=white!90!black,smooth] (0,0) rectangle (0.4,0.4);
	\filldraw[xscale=-1,xshift=-4.2cm,yshift=-.2cm,fill=white!90!black,smooth] (0,0) rectangle (0.4,0.4);
	\filldraw[xscale=-1,xshift=-5.2cm,yshift=-.2cm,fill=white!90!black,smooth] (0,0) rectangle (0.4,0.4);	
	\filldraw[xscale=-1,xshift=-2.2cm,yshift=-.2cm,fill=white!90!black,smooth] (0,0) rectangle (0.6,0.4);	
\draw[xscale=-1,smooth, very thick] (-4,0) to[out=90,in=180] (-3,1) to[out=0,in=90] (-2,0); 	
\draw[xscale=-1,xshift=1cm, smooth, very thick] (-4,0) to[out=90,in=180] (-3.5,.5) to[out=0,in=90] (-3,0); 
\draw[xscale=-1,smooth] (-5,0) to[out=90,in=180] (-3.5,1.5) to[out=0,in=90] (-1.8,0); 

\draw(-3.5,-.8) node{Type A.1.1};

\draw(3.5,-.8) node{Type A.1.2};

\end{tikzpicture}
 \caption{
 \label{fig:typesA.1}
Types A.1.1 and A.1.2 arising from simple type A.1 .}
\end{figure}
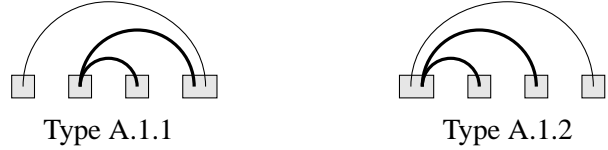


Now consider the lift $\widetilde{\mathbf{D}_{n}}\stackrel{\widetilde{f_{\beta}}}{\to}{\widetilde{\mathbf{D}_{n}}}$ which fixes the point in the fiber over $v_{\rm fix}$ contained in $D_{e_{H}}$ (the leaf of level zero). Such a lift always exists, for $f_{\beta}$ fixes downstairs $v_{\rm fix}$ by definition. 
We now consider the lift
\begin{equation}
	\label{E:liftTi}
(\widetilde{\tau_i},\widetilde{\varepsilon_i},\widetilde{h_{i}})\stackrel{\widetilde{T_i}}{\to}(\widetilde{\tau_{i+1}},\widetilde{\varepsilon_{i+1}},\widetilde{h_{i+1}})
\end{equation}
representing $\widetilde{f_{\beta}}$.
To compute $M(\widetilde{T_{i}})$ we present figure \ref{fig:edgeautomatonUPSTAIRS} where $\widetilde{f_{\beta}}(\widetilde{\tau_{i}})$ and $\widetilde{h_{i+1}}(\widetilde{\tau_{i+1}})$ are both depicted at the leaf $D_{e_{H}}$ of level zero. Remark that \emph{at the leaf of level zero}, except for real edges in $\widetilde{f_{\beta}}(\widetilde{\tau_{i}})$ depicted in bold, all real edges of $\widetilde{f_{\beta}}(\widetilde{\tau_{i}})$ are labeled by $\widetilde{\varepsilon_{i}}$ with letters in $\mathbb{A}$. Labels given to the edges in bold by $\widetilde{\varepsilon_{i}}$  are of the form $t_{i}^{-1}x$, where $x$ ranges among labels in $\mathbb{A}$ reserved for real edges in $N(T_{i})\cup f$ and  $t(X)=t_{i}\in H$. Here, $X\in\mathbb{A}_{\rm prong}$ is the label (given by $\varepsilon_{i}$)  of the unique infinitesimal edge of $\tau_{i}$ enclosing a puncture and that is incident to $v_{\rm end}$. Now remark that 
the projection of the real edges depicted in bold in figure \ref{fig:edgeautomatonUPSTAIRS} to the base $\mathbf{D}_{n}$ are precisely the edges in  $f_{\beta}(\tau_{i})$  contained in $N(T_{i})\cup f$.



\begin{figure}[htbp]
 \begin{tikzpicture}[scale=0.75, inner sep=1mm]

	\filldraw[very thick,xshift=-3.2cm,yshift=-.2cm,fill=white!90!black,smooth] (0,0) rectangle (0.4,0.4);
	\filldraw[xshift=-4.2cm,yshift=-.2cm,fill=white!90!black,smooth] (0,0) rectangle (0.4,0.4);
	\filldraw[xshift=-5.2cm,yshift=-.2cm,fill=white!90!black,smooth] (0,0) rectangle (0.4,0.4);	
	\filldraw[xshift=-1.2cm,yshift=-.2cm,fill=white!90!black,smooth] (0,0) rectangle (0.6,0.4);	
\draw[smooth] (-4,0) to[out=90,in=180] (-3,1) to[out=0,in=90] (-2,0); 	
\draw[smooth] (-4,0) to[out=90,in=180] (-3.5,.5) to[out=0,in=90] (-3,0); 
\draw[smooth] (-5,0) to[out=90,in=180] (-3.5,1.5) to[out=0,in=90] (-.8,0); 
\draw[smooth](-2,0) to[out=90,in=180] (-1.5,.5) to[out=0,in=90] (-1,0);

\draw(-3,.5) node{$e$};
\draw(-2,.8) node{$e'$};
\draw(-4,-.5) node{$0$};
\draw(-3,-.5) node{$2$};
\draw(-2,-.5) node{$1$};
\draw(0.1,1.2) node{$F_{ee'}$};
\draw(-2.5,-1.5) node{$f_{\beta}$};
\draw(-5,1.3) node{$\tau_{i}$};


\draw[smooth,->] (-.5,.8) to[out=25,in=150] (.5,.8);

\draw[smooth,->] (-3,-1) to[out=280,in=80] (-3,-2);


\filldraw[xshift=.8cm,yshift=-.2cm,fill=white!90!black,smooth] (0,0) rectangle (0.4,0.4);
\filldraw[xshift=1.8cm,yshift=-.2cm,fill=white!90!black,smooth] (0,0) rectangle (0.4,0.4);
\filldraw[very thick,xshift=2.8cm,yshift=-.2cm,fill=white!90!black,smooth] (0,0) rectangle (0.4,0.4);
	\filldraw[xshift=4.8cm,yshift=-.2cm,fill=white!90!black,smooth] (0,0) rectangle (0.6,0.4);
	\draw[smooth] (2,0) to[out=90,in=180] (3,1) to[out=0,in=90] (4,0);
\draw[xshift=6cm,smooth] (-5,0) to[out=90,in=180] (-3.5,1.5) to[out=0,in=90] (-.8,0); 
\draw[xshift=6cm,smooth](-2,0) to[out=90,in=180] (-1.5,.5) to[out=0,in=90] (-1,0);
\draw[smooth] (3,0) to[out=90,in=180] (3.5,.5) to[out=0,in=90] (3.75,0)to[out=270,in=180] (4,-.3)to[out=0,in=270] (4.2,0)to[out=90,in=0] (4.1,.1)to[out=180,in=90] (4,0);


\draw(4,1) node{$e'$};
\draw(3,.5) node{$e$};
\draw[xshift=6cm](-4,-.5) node{$0$};
\draw[xshift=6cm](-3,-.5) node{$2$};
\draw[xshift=6cm](-2,-.5) node{$1$};
\draw(3.5,-1.5) node{$f_{\beta}$};

\draw[smooth,->](3,-1)to[out=280,in=80] (3,-2);

\filldraw[xshift=.8cm,yshift=-4.2cm,fill=white!90!black,smooth] (0,0) rectangle (0.4,0.4);
\filldraw[xshift=1.8cm,yshift=-4.2cm,fill=white!90!black,smooth] (0,0) rectangle (0.4,0.4);
\filldraw[very thick,xshift=3.8cm,yshift=-4.2cm,fill=white!90!black,smooth] (0,0) rectangle (0.4,0.4);
\filldraw[xshift=4.8cm,yshift=-4.2cm,fill=white!90!black,smooth] (0,0) rectangle (0.6,0.4);
\draw[smooth](2,-4) to[out=90,in=180] (2.5,-3.5) to[out=0,in=90] (3,-4);
\draw[xshift=1cm,smooth](2,-4) to[out=90,in=180] (2.5,-3.5) to[out=0,in=90] (3,-4);
\draw[smooth](2,-4) to[out=90,in=180] (2.5,-3.5) to[out=0,in=90] (3,-4);
\draw[smooth](3,-4) to[out=90,in=180] (4,-3) to[out=0,in=90] (5,-4);
\draw[smooth](1,-4) to[out=90,in=180] (3,-2.5) to[out=0,in=90] (5.2,-4);


\draw(2.5,-3.1) node{$e'$};
\draw(3.8,-3.3) node{$e$};
\draw(2,-4.5) node{$0$};
\draw(3,-4.5) node{$1$};
\draw(4,-4.5) node{$2$};
\draw(1,-2.5) node{$\tau_{i+1}$};


\filldraw[xshift=-5.2cm,yshift=-4.2cm,fill=white!90!black,smooth] (0,0) rectangle (0.4,0.4);
\filldraw[xshift=-4.2cm,yshift=-4.2cm,fill=white!90!black,smooth] (0,0) rectangle (0.4,0.4);
\filldraw[very thick,xshift=-2.2cm,yshift=-4.2cm,fill=white!90!black,smooth] (0,0) rectangle (0.4,0.4);
\filldraw[xshift=-1.2cm,yshift=-4.2cm,fill=white!90!black,smooth] (0,0) rectangle (0.4,0.4);

\draw[xshift=-6cm,smooth](2,-4) to[out=90,in=180] (2.5,-3.5) to[out=0,in=90] (3,-4);
\draw[xshift=-6cm,smooth](3,-4) to[out=90,in=180] (4,-3) to[out=0,in=90] (5,-4);
\draw[xshift=-6cm,smooth](1,-4) to[out=90,in=180] (3,-2.5) to[out=0,in=90] (5.2,-4);
\draw[smooth] (-4,-4) to[out=90,in=180] (-3.5,-3.7) to[out=0,in=180] (-3,-4.5)to[out=0,in=270] (-2.5,-4)to[out=90,in=180] (-2.25,-3.5)to[out=0,in=90] (-2,-4);


\draw(-3.5,-3.1) node{$e'$};
\draw(-2,-3.3) node{$e$};
\draw(-4,-4.8) node{$0$};
\draw(-3,-4.8) node{$1$};
\draw(-2,-4.8) node{$2$};

\end{tikzpicture}
 \caption{
 \label{fig:edgeautomaton}
A detailed edge of the automaton in case A.1.1.}
\end{figure}
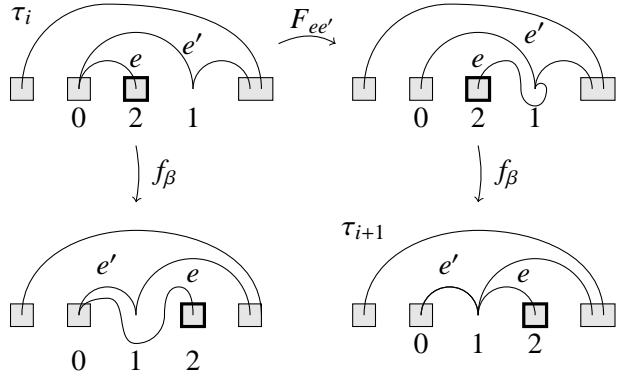

 From this data we can perform a straightforward calculation which shows that in this case $M(\widetilde{T_i})=M(T_i){\rm Diag (t(\eta_i(v_i)))}$, where all entries in the diagonal matrix ${\rm Diag (t(\eta_i(v_i)))}$ different from 1 are equal to $t_{i}$. The case when $T_{i}$ represents the homeomorphism $f_{\beta}$ that standardizes the train track $F_{ee'}(\tau_{i})$ that arises when folding $e'$ over $e$ is treated in the same way. \\


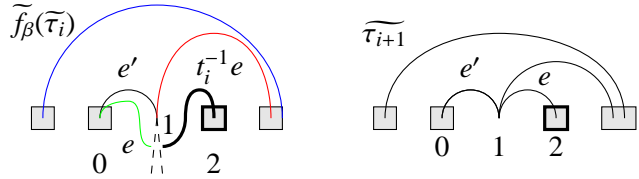
\begin{figure}[htbp]
 \begin{tikzpicture}[scale=0.75, inner sep=1mm]

\filldraw[xshift=.8cm,yshift=-4.2cm,fill=white!90!black,smooth] (0,0) rectangle (0.4,0.4);
\filldraw[xshift=1.8cm,yshift=-4.2cm,fill=white!90!black,smooth] (0,0) rectangle (0.4,0.4);
\filldraw[very thick,xshift=3.8cm,yshift=-4.2cm,fill=white!90!black,smooth] (0,0) rectangle (0.4,0.4);
\filldraw[xshift=4.8cm,yshift=-4.2cm,fill=white!90!black,smooth] (0,0) rectangle (0.6,0.4);
\draw[smooth](2,-4) to[out=90,in=180] (2.5,-3.5) to[out=0,in=90] (3,-4);
\draw[xshift=1cm,smooth](2,-4) to[out=90,in=180] (2.5,-3.5) to[out=0,in=90] (3,-4);
\draw[smooth](2,-4) to[out=90,in=180] (2.5,-3.5) to[out=0,in=90] (3,-4);
\draw[smooth](3,-4) to[out=90,in=180] (4,-3) to[out=0,in=90] (5,-4);
\draw[smooth](1,-4) to[out=90,in=180] (3,-2.5) to[out=0,in=90] (5.2,-4);


\draw(2.5,-3.1) node{$e'$};
\draw(3.8,-3.3) node{$e$};
\draw(2,-4.5) node{$0$};
\draw(3,-4.5) node{$1$};
\draw(4,-4.5) node{$2$};
\draw(1,-2.5) node{$\widetilde{\tau_{i+1}}$};
\draw(-5,-2.3) node{$\widetilde{f_{\beta}}(\widetilde{\tau_{i}}$)};

\filldraw[xshift=-5.2cm,yshift=-4.2cm,fill=white!90!black,smooth] (0,0) rectangle (0.4,0.4);
\filldraw[xshift=-4.2cm,yshift=-4.2cm,fill=white!90!black,smooth] (0,0) rectangle (0.4,0.4);
\filldraw[very thick,xshift=-2.2cm,yshift=-4.2cm,fill=white!90!black,smooth] (0,0) rectangle (0.4,0.4);
\filldraw[xshift=-1.2cm,yshift=-4.2cm,fill=white!90!black,smooth] (0,0) rectangle (0.4,0.4);

\draw[xshift=-6cm,smooth](2,-4) to[out=90,in=180] (2.5,-3.5) to[out=0,in=90] (3,-4);

\draw[red,xshift=-6cm,smooth](3,-4) to[out=90,in=180] (4,-2.5) to[out=0,in=90] (5,-4);
\draw[blue,xshift=-6cm,smooth](1,-4) to[out=90,in=180] (3,-2) to[out=0,in=90] (5.2,-4);
\draw[green,smooth] (-4,-4) to[out=90,in=180] (-3.5,-3.7) to[out=0,in=180] (-3.1,-4.5);
\draw[very thick,smooth](-2.9,-4.5)
to[out=0,in=270] (-2.5,-4)to[out=90,in=180] (-2.25,-3.5)to[out=0,in=90] (-2,-4);
\draw[dashed] (-3,-4) to (-3.1,-5);
\draw[dashed] (-3,-4) to (-2.9,-5);


\draw(-3.5,-3.1) node{$e'$};
\draw(-1.9,-3.1) node{$t_{i}^{-1}e$};
\draw(-4,-4.8) node{$0$};
\draw(-2.8,-4.1) node{$1$};
\draw(-2,-4.8) node{$2$};
\draw(-3.5,-4.5) node{$e$};

\end{tikzpicture}
 \caption{
 \label{fig:edgeautomatonUPSTAIRS}
Lifting an edge of the automaton in case A.1.1 .}
\end{figure}


\textbf{Case B.1}. This case is very similar to the preceding one. Let us consider the basic type $\Gamma$ given by figure \ref{fig:BASICTYPEB.1.1}. We consider all possible types of train tracks arising from $\Gamma$, which are displayed in the same figure. Among these, let us consider the edge of the automaton  (\ref{E:edgeatproof}) where $(\tau_{i},h_{i},\varepsilon_{i})$ is the  type B.1.1 depicted in figure \ref{fig:BASICTYPEB.1.1} and the train track map $T_{i}$ represents the homeomorphism $f_{\beta}$ that standardizes the train track $F_{ee'}(\tau_{i})$ that arises when folding edge $e$ over edge $e'$. 


\begin{figure}[htbp]
 \begin{tikzpicture}[scale=0.75, inner sep=1mm]


\foreach \x in {-4,-3,-2,-1}
	\draw[smooth, red] (\x,0) to[out=270,in=90] (\x-.1,-.1) to[out=270,in=180] (\x,-.25) to[out=0,in=270] (\x+.1,-.1) to[out=90,in=270] (\x,0);
\draw[smooth, very thick] (-3,0) to[out=90,in=180] (-2.5,.5) to[out=0,in=90] (-2,0); 	
\draw[smooth, very thick] (-4,0) to[out=90,in=180] (-3.5,.5) to[out=0,in=90] (-3,0); 


\draw(-3.5,.8) node{$e$};
\draw(-2.5,.8) node{$e'$};
\draw(-4,-.5) node{$v_{2}$};
\draw(-3,-.5) node{$v_{0}$};
\draw(-2,-.5) node{$v_{1}$};
\draw(-3,-1) node{Basic type $\Gamma$};

\begin{scope}
\filldraw[xshift=.8cm,yshift=-.2cm,fill=white!90!black,smooth] (0,0) rectangle (0.4,0.4);
\filldraw[xshift=1.8cm,yshift=-.2cm,fill=white!90!black,smooth] (0,0) rectangle (0.4,0.4);
\filldraw[xshift=2.8cm,yshift=-.2cm,fill=white!90!black,smooth] (0,0) rectangle (0.4,0.4);
\filldraw[xshift=3.8cm,yshift=-.2cm,fill=white!90!black,smooth] (0,0) rectangle (0.4,0.4);


\draw[xshift=5cm,smooth, very thick] (-3,0) to[out=90,in=180] (-2.5,.5) to[out=0,in=90] (-2,0); 	
\draw[xshift=5cm,smooth, very thick] (-4,0) to[out=90,in=180] (-3.5,.5) to[out=0,in=90] (-3,0); 
\draw[smooth, ] (1,0) to[out=90,in=180] (2.5,1.5) to[out=0,in=90] (4,0);
\draw(2,-.8) node{Type B.1.1};
\end{scope}

\begin{scope}[xshift=-3cm,yshift=-3cm]
\filldraw[xshift=.8cm,yshift=-.2cm,fill=white!90!black,smooth] (0,0) rectangle (0.4,0.4);
\filldraw[xshift=1.8cm,yshift=-.2cm,fill=white!90!black,smooth] (0,0) rectangle (0.4,0.4);
\filldraw[xshift=2.8cm,yshift=-.2cm,fill=white!90!black,smooth] (0,0) rectangle (0.4,0.4);
\filldraw[xshift=3.8cm,yshift=-.2cm,fill=white!90!black,smooth] (0,0) rectangle (0.4,0.4);


\draw[xshift=5cm,smooth, very thick] (-3,0) to[out=90,in=180] (-2.5,.5) to[out=0,in=90] (-2,0); 	
\draw[xshift=5cm,smooth, very thick] (-4,0) to[out=90,in=180] (-3.5,.5) to[out=0,in=90] (-3,0); 
\draw[smooth, ] (1,0) to[out=90,in=180] (2.5,1.5) to[out=0,in=90] (4,0);
\draw(2,-.8) node{Type B.1.2};
\end{scope}

\end{tikzpicture}
 \caption{
 \label{fig:BASICTYPEB.1.1}
Types B.1.1 and B.1.2 arising from simple type B.1 .}
\end{figure}
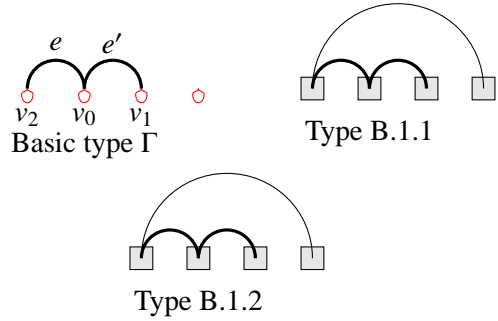


In figure \ref{fig:B.1.1AUTOMATONEDGE} we depict this edge of the automaton in detail. Remark that $v_{\rm fix}=v_{2}$, $v_{\rm end}=v_{1}$, $f=e$, $f'=e'$ and the subgraph $N(T_{i})$ is highlighted in bold. We remark that $f\in N(T_{i})$, hence we are in case 2 from definition \ref{DEF:VECTV}. This is the main difference with the case we treated in case A.1 .

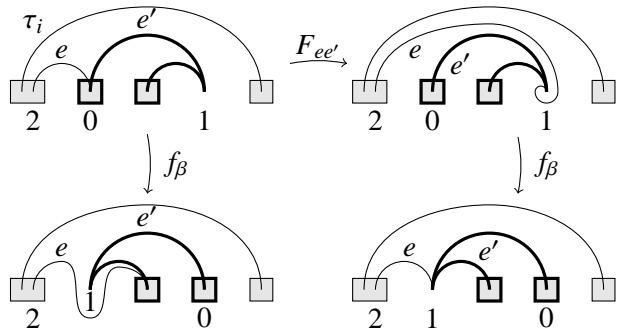
\begin{figure}[htbp]
 \begin{tikzpicture}[scale=0.75, inner sep=1mm]

\begin{scope}
\filldraw[xshift=-5.4cm,yshift=-.2cm,fill=white!90!black,smooth] (0,0) rectangle (0.6,0.4);
\filldraw[very thick,xshift=-4.2cm,yshift=-.2cm,fill=white!90!black,smooth] (0,0) rectangle (0.4,0.4);
\filldraw[very thick,xshift=-3.2cm,yshift=-.2cm,fill=white!90!black,smooth] (0,0) rectangle (0.4,0.4);
\filldraw[xshift=-1.2cm,yshift=-.2cm,fill=white!90!black,smooth] (0,0) rectangle (0.4,0.4);
\draw[smooth] (-5.2,0) to[out=90,in=180] (-3,1.5) to[out=0,in=90] (-1,0); 
\draw[smooth] (-5,0) to[out=90,in=180] (-4.5,.5) to[out=0,in=90] (-4,0); 
\draw[smooth,very thick] (-4,0) to[out=90,in=180] (-3,1) to[out=0,in=90] (-2,0);
\draw[smooth,very thick] (-3,0) to[out=90,in=180] (-2.5,.5) to[out=0,in=90] (-2,0);  
\draw(-5,-.5) node{$2$};
\draw(-4,-.5) node{$0$};
\draw(-2,-.5) node{$1$};
\draw(-4.5,.7) node{$e$};
\draw(-3,1.3) node{$e'$};
\draw(0,.8) node{$F_{ee'}$};
\draw(-5,1.2) node{$\tau_{i}$};
\draw(-2.5,-1.3) node{$f_{\beta}$};

\draw[smooth,->] (-.5,.5) to[out=10,in=170] (.5,.5);
\draw[smooth,->] (-3,-.75) to[out=280,in=80] (-3,-1.8);
\end{scope}


\begin{scope}[xshift=6cm]
\filldraw[xshift=-5.4cm,yshift=-.2cm,fill=white!90!black,smooth] (0,0) rectangle (0.6,0.4);
\filldraw[very thick,xshift=-4.2cm,yshift=-.2cm,fill=white!90!black,smooth] (0,0) rectangle (0.4,0.4);
\filldraw[very thick,xshift=-3.2cm,yshift=-.2cm,fill=white!90!black,smooth] (0,0) rectangle (0.4,0.4);
\filldraw[xshift=-1.2cm,yshift=-.2cm,fill=white!90!black,smooth] (0,0) rectangle (0.4,0.4);
\draw[smooth] (-5.2,0) to[out=90,in=180] (-3,1.5) to[out=0,in=90] (-1,0); 
\draw[smooth] (-5,0) to[out=90,in=180] (-3,1.2) to[out=0,in=90] (-1.8,0) to[out=270,in=0] (-2,-.2) to[out=180,in=270] (-2.2,0) to[out=90,in=180] (-2.1,.1)to[out=0,in=90] (-2,0); 
\draw[smooth,very thick] (-4,0) to[out=90,in=180] (-3,1) to[out=0,in=90] (-2,0);
\draw[smooth,very thick] (-3,0) to[out=90,in=180] (-2.5,.5) to[out=0,in=90] (-2,0);  
\draw(-5,-.5) node{$2$};
\draw(-4,-.5) node{$0$};
\draw(-2,-.5) node{$1$};
\draw(-4.3,.7) node{$e$};
\draw(-3.5,.5) node{$e'$};
\draw(-2,-1.3) node{$f_{\beta}$};
\draw[smooth,->] (-2.5,-.8) to[out=280,in=80] (-2.5,-1.8);
\end{scope}


\begin{scope}[yshift=-3.5cm]
\filldraw[xshift=-5.4cm,yshift=-.2cm,fill=white!90!black,smooth] (0,0) rectangle (0.6,0.4);
\filldraw[very thick,xshift=-3.2cm,yshift=-.2cm,fill=white!90!black,smooth] (0,0) rectangle (0.4,0.4);
\filldraw[very thick,xshift=-2.2cm,yshift=-.2cm,fill=white!90!black,smooth] (0,0) rectangle (0.4,0.4);
\filldraw[xshift=-1.2cm,yshift=-.2cm,fill=white!90!black,smooth] (0,0) rectangle (0.4,0.4);
\draw[smooth] (-5.2,0) to[out=90,in=180] (-3,1.5) to[out=0,in=90] (-1,0); 
\draw[smooth] (-5,0) to[out=90,in=180] (-4.5,.5) to[out=0,in=180] (-4,-.5) to[out=0,in=180] (-3.5,.4) to[out=0,in=90] (-3,0); 
\draw[smooth,very thick] (-4,0) to[out=90,in=180] (-3,1) to[out=0,in=90] (-2,0);
\draw[xshift=-1cm,smooth,very thick] (-3,0) to[out=90,in=180] (-2.5,.5) to[out=0,in=90] (-2,0);  
\draw(-5,-.5) node{$2$};
\draw(-4,-.2) node{$1$};
\draw(-2,-.5) node{$0$};
\draw(-4.5,.7) node{$e$};
\draw(-3,1.3) node{$e'$};
\end{scope}


\begin{scope}[xshift=6cm,yshift=-3.5cm]
\filldraw[xshift=-5.4cm,yshift=-.2cm,fill=white!90!black,smooth] (0,0) rectangle (0.6,0.4);
\filldraw[very thick,xshift=-3.2cm,yshift=-.2cm,fill=white!90!black,smooth] (0,0) rectangle (0.4,0.4);
\filldraw[very thick,xshift=-2.2cm,yshift=-.2cm,fill=white!90!black,smooth] (0,0) rectangle (0.4,0.4);
\filldraw[xshift=-1.2cm,yshift=-.2cm,fill=white!90!black,smooth] (0,0) rectangle (0.4,0.4);
\draw[smooth] (-5.2,0) to[out=90,in=180] (-3,1.5) to[out=0,in=90] (-1,0); 
\draw[smooth] (-5,0) to[out=90,in=180] (-4.5,.5) to[out=0,in=90] (-4,0);
\draw[smooth,very thick] (-4,0) to[out=90,in=180] (-3,1) to[out=0,in=90] (-2,0);
\draw[xshift=-1cm,smooth,very thick] (-3,0) to[out=90,in=180] (-2.5,.5) to[out=0,in=90] (-2,0);  
\draw(-5,-.5) node{$2$};
\draw(-4,-.5) node{$1$};
\draw(-2,-.5) node{$0$};
\draw(-4.4,.7) node{$e$};
\draw(-3,.7) node{$e'$};
\end{scope}
\end{tikzpicture}
 \caption{
 \label{fig:B.1.1AUTOMATONEDGE}
A detailed edge of the automaton in case B.1.1.}
\end{figure}


 Now consider the lift $\widetilde{\mathbf{D}_{n}}\stackrel{\widetilde{f_{\beta}}}{\to}{\widetilde{\mathbf{D}_{n}}}$ which fixes the point in the fiber over $v_{\rm fix}$ contained in $D_{e_{H}}$. To compute $M(\widetilde{T_{i}})$ we present figure \ref{fig:B.1.1AUTOMATONEDGEupstairs} where $\widetilde{f_{\beta}}(\widetilde{\tau_{i}})$ and $\widetilde{h_{i+1}}(\widetilde{\tau_{i+1}})$ are both depicted at the leaf $D_{e_{H}}$ of level zero. Remark that \emph{at the leaf of level zero}, except for real edges in $\widetilde{f_{\beta}}(\widetilde{\tau_{i}})$ depicted in bold, all real edges in $\widetilde{f_{\beta}}(\widetilde{\tau_{i}})$ are labeled by $\widetilde{\varepsilon_{i}}$ with letters in $\mathbb{A}$. Labels given to the edges in bold by $\widetilde{\varepsilon_{i}}$  are of the form $t_{i}^{-1}x$, 
where $x$ ranges among labels in $\mathbb{A}$ reserved for real edges in $N(T_{i})\cup f$ and  $t(X)=t_{i}$. Here, $X\in\mathbb{A}_{\rm prong}$ is the label (given by $\varepsilon_{i}$)  of the unique infinitesimal edge of $\tau_{i}$ enclosing a puncture and that is incident to $v_{\rm end}$. Now remark that 
the projection of the real edges depicted in bold in figure \ref{fig:B.1.1AUTOMATONEDGEupstairs} to the base $\mathbf{D}_{n}$ are precisely the edges in  $f_{\beta}(\tau_{i})$  contained in $N(T_{i})$. From this data we can perform a straightforward calculation which shows that in this case $M(\widetilde{T_i})=M(T_i){\rm Diag (t(\eta_i(v_i)))}$,  where all entries in the diagonal matrix ${\rm Diag (t(\eta_i(v_i)))}$ different from 1 are equal to $t_{i}$. The case when $T_{i}$ represents the homeomorphism $f_{\beta}$ that standardizes the train track $F_{ee'}(\tau_{i})$ that arises when folding $'e$ over $e$ is treated in the same way. 


\begin{figure}[htbp]
 \begin{tikzpicture}[scale=0.75, inner sep=1mm]


\begin{scope}
\filldraw[xshift=-5.4cm,yshift=-.2cm,fill=white!90!black,smooth] (0,0) rectangle (0.6,0.4);
\filldraw[very thick,xshift=-3.2cm,yshift=-.2cm,fill=white!90!black,smooth] (0,0) rectangle (0.4,0.4);
\filldraw[very thick,xshift=-2.2cm,yshift=-.2cm,fill=white!90!black,smooth] (0,0) rectangle (0.4,0.4);
\filldraw[xshift=-1.2cm,yshift=-.2cm,fill=white!90!black,smooth] (0,0) rectangle (0.4,0.4);
\draw[smooth] (-5.2,0) to[out=90,in=180] (-3,1.7) to[out=0,in=90] (-1,0); 
\draw[smooth] (-5,0) to[out=90,in=180] (-4.5,.5) to[out=0,in=180] (-4.1,-.5);

\draw[very thick,smooth] (-3.9,-.5) 
 to[out=0,in=180] (-3.5,.4) to[out=0,in=90] (-3,0);

\draw[dashed](-4,0) to (-4.1,-1);
\draw[dashed](-4,0) to (-3.9,-1);
 
\draw[smooth,very thick] (-4,0) to[out=90,in=180] (-3,1) to[out=0,in=90] (-2,0);
\draw[xshift=-1cm,smooth,very thick] (-3,0) to[out=90,in=180] (-2.5,.5) to[out=0,in=90] (-2,0);  
\draw(-5,-.5) node{$2$};
\draw(-4,-1.2) node{$1$};
\draw(-2,-.5) node{$0$};
\draw(-4.5,.7) node{$e$};
\draw(-3,1.3) node{$t_{i}^{-1}e'$};
\draw(-3.3,-.6) node{$t_{i}^{-1}e$};
\end{scope}


\begin{scope}[xshift=6cm]
\filldraw[xshift=-5.4cm,yshift=-.2cm,fill=white!90!black,smooth] (0,0) rectangle (0.6,0.4);
\filldraw[very thick,xshift=-3.2cm,yshift=-.2cm,fill=white!90!black,smooth] (0,0) rectangle (0.4,0.4);
\filldraw[very thick,xshift=-2.2cm,yshift=-.2cm,fill=white!90!black,smooth] (0,0) rectangle (0.4,0.4);
\filldraw[xshift=-1.2cm,yshift=-.2cm,fill=white!90!black,smooth] (0,0) rectangle (0.4,0.4);
\draw[smooth] (-5.2,0) to[out=90,in=180] (-3,1.5) to[out=0,in=90] (-1,0); 
\draw[smooth] (-5,0) to[out=90,in=180] (-4.5,.5) to[out=0,in=90] (-4,0);
\draw[smooth,very thick] (-4,0) to[out=90,in=180] (-3,1) to[out=0,in=90] (-2,0);
\draw[xshift=-1cm,smooth,very thick] (-3,0) to[out=90,in=180] (-2.5,.5) to[out=0,in=90] (-2,0);  
\draw(-5,-.5) node{$2$};
\draw(-4,-.5) node{$1$};
\draw(-2,-.5) node{$0$};
\draw(-4.4,.7) node{$e$};
\draw(-3,.7) node{$e'$};
\end{scope}

\end{tikzpicture}
 \caption{
 \label{fig:B.1.1AUTOMATONEDGEupstairs}
Lifting an edge of the automaton in case B.1.1 .}
\end{figure}
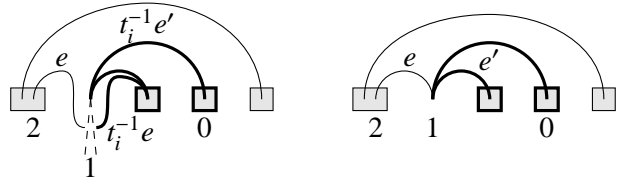


The rest of the cases are treated as follows:
\begin{enumerate}
\item Pick a basic type $\Gamma$ from the list presented in Appendix \ref{APPENDIX:CASES} and consider $(\tau_{i},h_{i},\varepsilon_{i})$ one of the finitely many possible types of train tracks that can be 
constructed from $\Gamma$. 
\item Consider an edge (\ref{E:edgeatproof}) of the folding automaton starting from the train track chosen in the preceding step and where $T_{i}$ comes from either the folding $F_{ee'}$ or $F_{e'e}$ \footnote{Remark that for some simple types only one of these two foldings is possible.}, and represents a standardizing homeomorphism $\mathbf{D}_{n}\stackrel{f_{\beta}}{\to}\mathbf{D}_{n}$. 
\item Pick the lift $\widetilde{\mathbf{D}_{n}}\stackrel{\widetilde{f_{\beta}}}{\to}{\widetilde{\mathbf{D}_{n}}}$ which fixes the point in the fiber over $v_{\rm fix}$ contained in $D_{e_{H}}$ and the corresponding train track map of the form (\ref{E:liftTi}) representing it. Then, depending on the case of definition \ref{DEF:VECTV}, perform analogous calculations to the ones presented in cases A.1 and B.1 .
\end{enumerate}
To finish the proof consider any edge (\ref{E:edgeatproof}) in the decorated folding automaton with extra information $(\pi_{i},v_{i})$. 
If $e$ is an infinitesimal edge in $\tau_{i+1}$, then we can define $\varepsilon_{i+1}(e):=\pi_{i}(\varepsilon_{i}(e))$. 
Hence $w_{i+1}=\pi_i\circ\dots\circ\pi_1(v_{i+1})$. This implies the recursive nature of formula  (\ref{E:DiagFormulaTh}). 
\end{proof}

\begin{remark}
Let $t$ be the variable of $H$ and $u$ corresponds to the lift $\widetilde{\psi}$. Then
$$
\Theta_F(t,u)\in \Z[G]= \Z[t]\oplus \Z[u].
$$
From Remark~\ref{R:dlifts} we conclude that picking different lifts of the standardizing homeomorphism $f_{\beta}$ results in multiplying $M(\widetilde{T})$ by some $t_{0}\in H$. 
This does not affect the expression for the Teichm\"uller polynomial since $\Theta_F(t,u) = \det(uI-t_{0}M(\widetilde{T}))=t_{0}^b\det(t_{0}^{-1}uI-M(\widetilde{T}))=t_{0}^{(n-1)}\Theta_{F}(u',t)$, where $u'=t_{0}^{-1}u$ is the coordinate corresponding to the other lift.
\end{remark}

\begin{remark}
Observe that we can formally apply above theorem to any isotopy class (not necessary pseudo-Anosov).
See the example below.
\end{remark}

\begin{example}
	\label{ex:first}
Let us consider the class $[f_\beta]$ where $\beta=\sigma_2^2$ is a braid in $B_{3}$. (see Figure~\ref{fig:decorated:B3}). The 
corresponding loop is
$$
(\tau,\varepsilon)\stackrel{T_{1}}\longrightarrow(\tau,\varepsilon)\stackrel{T_{2}}{\longrightarrow} (\tau,\varepsilon)
$$
In this example the map $t:\{A,B,C\}\to H$ is given by $t(\alpha)=t_{\alpha}$ for all $\alpha\in {\mathbb{A}_{\mathrm{prong}}} = \{A,B,C\}$ since the permutation on punctures defined by $f_{\beta}$ is the identity. In other words, $H$ is isomorphic to $\Z^{3}$. 
A direct calculation shows that $w_{1}=\eta_{1}(1,C^{+})=(1,C^{+})$ since $\eta_{1}=\mathrm{Id}_{|\mathbb{A}_{\mathrm{prong}}}$. 
Hence $t(w_{1})=(1,t_{C})$. On the other hand $w_{2}=\eta_{2}(1,C^{+})=(1,B^{+})$ since $\eta_{2}(A,B,C)=(A,C,B)$. Hence 
$t(w_{2})=(1,t_{B})$. A direct calculation also shows that the matrix associated to $T_{i}$ is $\left(\begin{smallmatrix} 1 & 1 \\ 0 & 1 \end{smallmatrix}\right)$
for $i=1,2$. Theorem~\ref{thm:second:main} gives
$$
M(\sigma^2_2) = \left(\begin{smallmatrix} 1 & 1 \\ 0 & 1 \end{smallmatrix}\right) \cdot
\left(\begin{smallmatrix} 1 & 0 \\ 0 & t_C \end{smallmatrix}\right) \cdot
\left(\begin{smallmatrix} 1 & 1 \\ 0 & 1 \end{smallmatrix}\right) \cdot
\left(\begin{smallmatrix} 1 & 0 \\ 0 & t_B \end{smallmatrix}\right) = 
\left(\begin{smallmatrix} 1 & t_B+t_Bt_C \\ 0 & t_Bt_C \end{smallmatrix}\right).
$$
Similarly for $\beta=\sigma_1^{-2}$ we have $t(\alpha)=t_{\alpha}$ and $H$ isomorphic to $\Z^{3}$. A direct calculation shows that 
$w_{1}=\eta_{1}(A^{-1},1)=(A^{-1},1)$. Hence $t(w_{1})=(t_{A}^{-1},1)$. On the other hand $w_{2}=\eta_{2}(A^{-1},1)=(B^{-1},1)$
with $\eta_{2}(A,B,C)=(B,A,C)$. Hence $t(w_{2})=(t_{B}^{-1},1)$.
A direct calculation also shows that the matrix associated to $T_{i}$ is $\left(\begin{smallmatrix} 1 & 0 \\ 1 & 1 \end{smallmatrix}\right)$
for $i=1,2$. Theorem~\ref{thm:second:main} then gives
$$
M(\sigma_1^{-2}) = \left(\begin{smallmatrix} 1 & 0 \\ 1 & 1 \end{smallmatrix}\right) \cdot
\left(\begin{smallmatrix} t_A^{-1} & 0 \\ 0 & 1 \end{smallmatrix}\right) \cdot
\left(\begin{smallmatrix} 1 & 0 \\ 1 & 1 \end{smallmatrix}\right) \cdot
\left(\begin{smallmatrix} t_B^{-1} & 0 \\ 0 & 1 \end{smallmatrix}\right) = 
\left(\begin{smallmatrix} t_A^{-1}t_B^{-1} & 0 \\ t_A^{-1}t_B^{-1}+t_B^{-1} & 1 \end{smallmatrix}\right).
$$

\begin{remark}
Compare with formula  in~\cite[\S 11]{Mc} (here the presentation of the group $H$ is different):
$$
M(\sigma_2^{-2})=\left(\begin{smallmatrix} 1 & 0 \\ t^{-1}_2+t^{-1}_2t^{-1}_3 & t^{-1}_2t^{-1}_3 \end{smallmatrix}\right)
\qquad \textrm{and} \qquad
M(\sigma^2_1)=\left(\begin{smallmatrix} t_1t_2 & t_1t_2+t_2 \\ 0&1 \end{smallmatrix}\right).
$$
\end{remark}
For instance our algorithm applied to $\sigma_1^{-2}\sigma_2^6$ given by the loop
$$
\left( (\tau,\varepsilon)\stackrel{T_1}\longrightarrow(\tau,\varepsilon)\right)^3\stackrel{T_2}{\longrightarrow} (\tau,\varepsilon),
$$
(where $T_1$ represents $\sigma_2^2$ and $T_2$ represents $\sigma_1^{-2}$) gives 
that the Teichm\"uller polynomial is 
$$
\Theta_{F}(t_A,t_B,t_C,u)=det\left(
u\cdot \Id-
\left(\begin{smallmatrix} 1 & t_B+t_Bt_C \\ 0 & t_Bt_C \end{smallmatrix}\right)^3
\cdot \left(\begin{smallmatrix} t_A^{-1}t_B^{-1} & 0 \\ t_A^{-1}t_B^{-1}+t_B^{-1} & 1 \end{smallmatrix}\right)
\right).
$$

\end{example}

\section{Topology of a fiber}
\label{ss:TN}

In this section we provide a way to compute the topology (genus, number and type of 
singularities) of a fiber. We begin by introducing some notation that will be used throughout the
sequel. As usual we will consider a mapping torus $M_{\pa}=\mathbf{D}_{n}\times [0,1]/(x,1)\sim(\pa(x),0)$
induced by a pseudo-Anosov braid $\beta\in B_n$. It turns out that there is a natural model for
$M_\pa$ as a link complement $S^{3}\setminus \mathcal N(L)$ where $\mathcal N(L)$ is a regular 
neighborhood of a link $L$ in the $3$-sphere. To construct the link $L=L_\beta$, simply close the braid 
$\beta$ representing $\pa$ while passing it through an unknot $\alpha$ (representing the boundary of the disc
$\mathbf{D}_{n}$).
Let $\fibb$ be a fibration with  
monodromy $\varphi : \Sigma \rightarrow \Sigma$. Recall that, if $\Sigma$ has genus $g$ and $b$ boundary components, then $ \chi_{-}(\Sigma)=2g+b-2$. Hence the Thurston norm does not determine completely the topology of $\Sigma$. To achieve this we have 
to determine one of the numbers $g$ or $b$ (the surface $\Sigma$ is orientable).

\subsection{Computing the number of boundary components}
\label{SS:CNBC}
Since $M_{\pa}$ is homeomorphic to the link complement $S^{3}\setminus \mathcal N(L)$
we can describe its homology group easily. First there is an embedding $i:\mathbf{D}_{n}\hookrightarrow M$ such that the image of the exterior boundary of $\mathbf{D}_{n}$ spans $\alpha$ and $i(\mathbf{D}_{n})$ is punctured by the $n$ strands of $\beta$. 
The boundary of $M$ is a union of tori $T_{1},\ldots,T_{r}$,  where $r=b_{1}(M)$ (viewed as the boundary of a
regular neighborhood of link components $\partial \mathcal N(L_i)$). Let $[S_{1}],\ldots,[S_{r}]$ be a basis of 
$H_{2}(M,\partial M;\R)$ ({\em e.g.} take Seifert surfaces whose boundary is $T_i$). By 
convention we normalize so that $S_r=i(\mathbf{D}_{n})$. This normalization implies that $T_{r}$ comes from the unknot $\alpha$. The meridians 
of components of $L_\beta$ give a natural basis for $H_1(M,\Z)$ \cite{Hillman}.
 
Now let $\{[m_{i}],[l_{i}]\}$ be a meridian and longitude basis for 
$H_{1}(T_{i};\R)$, where the orientation of $l_{i}\subset\partial S_{i}$ is induced by the orientation of $[S_{i}]$.We consider the long exact sequence of the homology groups of the pair $(M,\partial M)$. We write the boundary map:
$$
\partial_{*}: H_{2}(M,\partial M;\R)\to H_{1}(\partial M;\R).
$$
On the chosen basis, for any $i=1,\dots,r$, one has 
$$
\partial_{*}[S_{i}]=\sum_{j=1}^{r}(a_{ij}[m_{j}]+b_{ij}[l_{j}])
$$
with $a_{ij},b_{ij}\in\Z$. We set $A=(a_{ij})_{i,j=1,\dots,r}$ and $B=(b_{ij})_{i,j=1,\dots,r}$.

\begin{proposition}
\label{prop:cc:boundary}
Let $\kappa = \sum_{i=1}^{r}c_{i}[S_{i}]$, where $c=(c_1,\dots,c_r)\in\Z^r$, be an integral homology class
(not necessarily primitive). Then for any embedded surface $S\subset M_\pa$ (not necessarily minimal representative)
such that $[S]=\kappa$, and for any $j=1,\dots,r$, the number of connected components of 
$S \cap T_j$ is $\gcd(a_j,b_j)$ where $(a_1,\dots,a_r)=cA$ and $(b_1,\dots,b_r)=cB$.
\end{proposition}

\begin{proof}[Proof of Proposition~\ref{prop:cc:boundary}] 
Writing $[S]=\sum_{i=1}^{r}c_{i}[S_{i}] \in H_{2}(M,\partial M;\R)$, elementary linear algebra gives
$$
\partial_{*}([S])= \sum_{i=1}^{r}\left( \sum_{j=1}^{r}c_{i}a_{ij}[m_{j}]+c_{i}b_{ij}[l_{j}]\right).
$$
Now $S\cap T_{j}\subset\partial S$ is a union of oriented parallel simple closed
curves. Hence its homology class is given by
$$
(\sum_{i=1}^{r}c_{i}a_{ij})[m_{j}]+(\sum_{i=1}^{r}c_{i}b_{ij})[l_{j}]\in H_{1}(T_{j};\R)
$$
Thus the number of connected components of $S\cap T_{j}$ is given by
$$
\gcd(\sum_{i=1}^{r}c_{i}a_{ij},\sum_{i=1}^{r}c_{i}b_{ij}).
$$
The proposition is proved.

\end{proof}

\begin{remark}
	\label{R:links}
In our situation, since $S_{i}$ is a Seifert surface whose boundary is the torus $T_i$ one has:
$$
\partial_{*}[S_{i}]=[l_i] - \sum_{j=1}^{r}\mathrm{Lk}(L_i,L_j)[m_{j}],
$$
where $\mathrm{Lk}(L_i,L_k)$ is the linking number of the two closed curves $L_i$ and 
$L_j$ with orientations given by the orientations of $[l_{i}]$ and $[l_{j}]$. In other words:
$B=\mathrm{Id}$ and $A=(\mathrm{Lk}(L_i,L_j)))_{i,j=1,\dots,r}$.
\end{remark}

We end this section with the following corollary on the connected components of 
$\Sigma \cap T_i$.

\begin{corollary}
\label{cor:slope}
For any $[\Sigma]=\sum_{i=1}^{r}c_{i}[S_{i}] \in H_{2}(M,\partial M;\R)$ where 
$c=(c_1,\dots,c_r)\in\Z^r$ we let $a=cA$ and $b=cB$ as above. Then each
connected component of $\Sigma \cap T_j$ is identified to a curve (well defined up to isotopy):
$$
c_{\frac{p}{q}}=p[m_{j}]+q[l_{j}]\in H_{1}(T_{j};\R), \qquad \textrm{with} \qquad p=\frac{a_j}{\gcd(a_j,b_j)}, \ q=\frac{b_j}{\gcd(a_j,b_j)}.
$$
\end{corollary}

From the last corollary we make the following definition. If $T$ is a torus, and 
if $H_1(T,\Z)$ is equipped with its preferred basis given by meridian and longitude 
(denoted by $[m]$ and $[l]$) then the {\em slope} of an essential simple closed curve 
$[c]=p[m]+q[l]$ (with $\gcd(p,q)=1$) is $\frac{p}{q}$. Conversely for any {\em slope} 
$r\in \Q\cup\{\infty\}$ one defines the (isotopy class) $c_r$ of the corresponding simple closed curve. \medskip

\subsection{Computing the number and type of singularities of the fiber} Let $\fib$ be a fibration in $\R^{+}\cdot F$ with pseudo-Anosov monodromy $\phi$. In this section we explain how to compute the singularity data of the stable measured foliation of $\Sigma$ that is invariant by $\phi$ using the singularity data of the stable measured foliation of $\psi$. The arguments are based on work by Fried. \medskip

In the sequel we denote by $\F$ the stable measured foliation invariant by $\pa$. 
Up to isotopy one can assume that $\pa(\F)=\F$. This determines a $2$-dimensional lamination
$\mathcal{L}_{\pa}=\F\times\R/\left<(s,t)\sim(\pa(s),t-1)\right>$ obtained as the mapping torus of $\pa: \F\rightarrow \F$.
The ``vertical flow lines'' $\{s\}\times\R\subset \Sigma\times\R$ descend to the leaves of a $1$-dimensional 
foliation whose leaves will be called \emph{the flow lines} of $\pa$. Hence $\mathcal L_\pa$
is swept out by the leaves of the flow lines passing through $\F$. \medskip

We distinguish two cases: $\F$ has no singularities in the interior 
of $\mathbf{D}_{n}$ (see Proposition~\ref{prop:cc:singu}) or $\F$ has some singularities in 
the interior of $\mathbf{D}_{n}$ (see Proposition~\ref{prop:cc:singu:2}).
Any singularity $s$ of $\F$ in the boundary of $\mathbf{D}_{n}$ determines 
a closed curve $\gamma_s$ of slope $\frac{p_s}{q_s}$ on the torus $T_s \subset \partial M$, given by the 
flow line passing throughout $s$. See for example figures 4 and 5 in \cite{KT1}.

\begin{proposition}[$\F$ has no singularities in the interior of $\Dn$]
\label{prop:cc:singu}
 For any $[\Sigma]\in \R^{+}\cdot F\subset H_{2}(M,\partial M;\R)$ with monodromy 
$\phi : \Sigma \rightarrow \Sigma$, the singularity data of the stable foliation $\F_\phi$ of the 
pseudo-Anosov map $\phi$ is given by:
\begin{enumerate}
\item $\F_\phi$ has no singularities in the interior of $\Sigma$.
\item At each connected component of $\partial\Sigma\cap T_{s}$ (of slope $\frac{p}{q}$ given 
by Corollary~\ref{cor:slope}), there is a singularity of $\F_\phi$ of type $k\cdot |p_{s}q-q_{s}p|$-prong if $s$ is a $k-$prong singularity.
\end{enumerate}
\end{proposition}
\begin{remark}
In all examples that we will be treating in this article, every singularity of $\mathcal{F}$ in a boundary of $\mathbf{D}_{n}$ that does not intersect $T_{r}$ is a $1-$prong. Except for the example treated in section~\ref{EXB4} in $B_4$,
$\F_\phi$ has no singularities in the interior of $\Sigma$ and for each $1\leq i< r$, 
writing $\partial\Sigma \cap T_i = a[m_{i}]+b[l_{i}]\in H_{1}(T_{i};\R)$, $\F_\phi$
has $\gcd(a,b)$ singularities in $T_i$, each of which is a $|p_{i}q-q_{i}p|$-prong,
where $p=\frac{a}{\gcd(a,b)}$, $q=\frac{b}{\gcd(a,b)}$ and $\frac{p_i}{q_i}$ is the slope 
of the curve $\gamma_i \subset T_i$. The type of the remaining singularity can be determined using the Euler-Poincar\'e formula. 
\end{remark}

\begin{proof}[Proof of Proposition~\ref{prop:cc:singu}]
Let $[\Sigma]\in\R\cdot F\subset H_{2}(M_{\pa},\partial M_{\pa};\R)$ be a fiber of $M$, with 
monodromy $\phi : \Sigma \rightarrow \Sigma$.
We will use the following result of Fried (see~\cite{F,Mc}): 
After an isotopy we have that
\begin{enumerate}
\item The fiber $\Sigma$ is transverse to the flow lines of $\pa$, and
\item The monodromy of the fibration determined by $[\Sigma]$ coincides with the first return map of the foliation $\mathcal{F}$.
\end{enumerate}
Hence the monodromies of any two points in $\R\cdot F\subset H_{2}(M_{\pa},\partial M_{\pa};\R)$ determine, up to isotopy, 
the same lamination $\mathcal{L}_{\psi}$. Let $\tau\hookrightarrow \mathbf{D}_{n}$ be a train track invariant by our 
given pseudo-Anosov homeomorphism $\pa$. 
Up to isotopy we assume that $\pa(\tau)$ is contained a fibered neighborhood of $\tau$ and transverse to the tie foliation.
We assume that $\tau$ carries the measured foliation 
$\F$. Let $\mathcal L_\tau$ be the mapping torus of $\tau$, namely $\mathcal L_\tau = 
\tau\times[0,1]/\left<(x,1)\sim(\pa(x),0)\right>$. \medskip

The aforementioned result of Fried implies that the intersection $\F_\phi=\Sigma\cap\mathcal{L}_\pa$ 
defines an invariant measured foliation for $\phi$ and  $\F_\phi$ is carried by the train track 
$\tau_\phi=\Sigma\cap \mathcal L_\tau$. By construction 
$\mathcal L_\pa$ is carried by the branched surface $\mathcal L_\tau$. For any singularity $s$ of 
$\F$, one obtains a simple closed curve $\gamma_s \subset M$ which is the closed orbit of the flow line passing 
throughout $s$. Notice that the union of all $\gamma_s$ is the branched loci of $\mathcal L_\tau$. Since $\F$ has no singularities in the interior of $\mathbf{D}_{n}$
all curves $\gamma_s$ lies in $T_i$ for some $i$. Hence all the singularities of $\F_\phi$ lies 
in $\partial \Sigma\cap T_i$ which proves the first point of the proposition. \medskip

Now we determine the number and type of prongs of $\tau_{\phi}$. 
For that we consider the number of prongs of $\F_\phi$ at each component of $\partial \Sigma \cap T_j$
for each $j$ (clearly for a given $j$ the type of the singularity at each component is the same).
Let $c_{\left[\frac{p}{q}\right]}=p[m_{j}]+q[l_{j}]\in H_{1}(T_{j};\R)$ be the corresponding curve 
representing a connected component of $\partial \Sigma \cap T_j$ (see Corollary~\ref{cor:slope}). By the aforementioned result of 
Fried, each intersection between $c_{\left[p/q\right]}$ and $\gamma_s$ contributes 
to $k-$infinitesimal edges (if $s$ is a $k-$prong singularity). Hence total number of prongs of $\F_\phi$ at $\partial \Sigma \cap T_j$
is equal to 
$$
k\cdot i(c_{\left[p/q\right]},\gamma_s).
$$
Since the slope of $\gamma_s$ is $p_s/q_s$ one draws:
$$
i(c_{\left[p/q\right]},\gamma_s) = |p_{s}q-q_{s}p|.
$$
This ends the proof of Proposition~\ref{prop:cc:singu}.
\end{proof}

We now address the case when $\mathcal{F}$ has singularities in the interior of $\mathbf{D}_{n}$. Roughly speaking, the idea is to remove the interior singularities of $\mathcal{F}$ to be in the context of the preceding case.

Note that in the definition of pseudo-Anosov homeomorphism we can 
remove or add punctures while keeping the ``same'' map $\pa : S \rightarrow S$. 
More precisely when $\{\pa^{i}(x)\}$ is a periodic orbit of {\em unpunctured points}, puncturing at
$\{\pa^i(x)\}$ refers to adding them to the puncture set $\{p_i\}$. Conversely, when $\{\pa^i(p)\}$ is a periodic 
orbit of $k$-prong punctured singularities for $k > 1$, capping them off refers to removing them from 
the puncture set. For pseudo-Anosov braids, puncturing or capping off corresponds to
adding or removing some strands.

\begin{proposition}[$\F$ has singularities in the interior of $\mathbf{D}_{n}$]
\label{prop:cc:singu:2}
Puncturing at $\{\pa^i(s)\}$ for any singularity $s$ of $\F$ in the interior of $\mathbf{D}_{n}$ 
gives rise to a pseudo-Anosov $\widetilde{\pa} : \mathbf{D}_{m} \rightarrow \mathbf{D}_{m}$ where $m>n$.
By construction $\F_{\widetilde{\pa}}$ has no interior singularities. Moreover the injection 
$\mathbf{D}_{n} \rightarrow \mathbf{D}_{m}$ induces a map $M_\pa \rightarrow M_{\widetilde{\pa}}=:\widetilde{M}$
and each class $[\Sigma]\in \R^{+}\cdot F\subset H_{2}(M,\partial M;\R)$ (with 
with monodromy $\phi$) determines 
a class $[\widetilde{\Sigma}]\in \R^{+}\cdot F\subset H_{2}(\widetilde{M},\partial \widetilde{M};\R)$ 
with monodromy $\widetilde{\phi}$.

The map $\phi$ is obtained by capping the singularities of $\widetilde{\phi}$ that lie in the interior of $\mathbf{D}_{n}$. In particular 
$\F_\phi$ and $\F_{\widetilde{\phi}}$ share the same number and type of singularities.
\end{proposition}

\begin{proof}[Proof of Proposition~\ref{prop:cc:singu:2}]
The proof is clear from the definition of capping off and puncturing at singularities.
\end{proof}

\subsection{Orientability of singular foliation}

In this section we determine whether or not the measured foliation $\F_\phi$ 
is orientable. For that we will use the following well known theorem of Thurston:
\begin{NoNumberTheorem}
For any pseudo-Anosov homeomorphism $\phi$ on a surface $\Sigma$ 
the following are equivalent:
\begin{enumerate}
\item The stretch factor of $\phi$ is an eigenvalue of the linear map $\phi_\ast$ defined
on $H_1(\Sigma,\mathbb \Z)$.
\item The invariant measured foliation $\F_\phi$ of $\phi$ is orientable.
\end{enumerate}
\end{NoNumberTheorem}

To compute the homological dilatation we will make use of the \emph{Alexander polynomial}. 
As the Teichm\"uller polynomial, the Alexander polynomial of $M$:
$$
\Delta_{M}=\sum_{g\in G} b_{g}\cdot g 
$$
is an element of the group ring $\Z[G]$, where $G=H_{1}(M,\Z)/\Tor$. For a precise definition see~\cite[\S 2]{McA}. The Alexander polynomial can be evaluated in an homology class $[\Sigma]\in H_{2}(M,\partial M;\R)$ using Poincar\'e-Lefschtez duality and then proceeding as with the Teichm\"uller polynomial in the corresponding dual cohomology class. We have the following classical result (see {\em e.g.}~\cite{milnor})
\begin{theorem}
\label{thm:homological}
Let $[\alpha]\in \R^{+}\cdot F\subset H^{1}(M;\R)$ with monodromy 
$\phi : \Sigma \rightarrow \Sigma$. Then the characteristic polynomial of 
$\phi_\ast$ acting on $H_1(\Sigma,\mathbb \Z)$ is given by
the Alexander polynomial $\Delta_M$ evaluated in $[\alpha]$.
\end{theorem}

\section{First examples in $B_{3}$}
\label{sec:examples}

The goal of this section is to revisit classical examples, first studied by 
Hironaka ~\cite{Hi}, Kin-Takasawa~\cite{KT1} and McMullen~\cite{Mc}. 
We stress that the novelty in this section are the methods presented to perform calculations, 
and not the results of these. In the next sections we will address examples in $B_n$ for $n\geq 4$. 

\begin{convention}
The natural order on $\mathbb{R}$ induces an order on the punctures of $\mathbf D_n$ and thus a labeling.
We label the infinitesimal edges enclosing punctures in $\mathbf{D}_{3}$, from left to right, by $A,B,C$ so that
$\mathbb{A}_{\mathrm{prong}}=\{A,B,C\}$. Hence the standard generators $\sigma_{1},\sigma_{2}$ (induced by left Dehn half-twists around loops enclosing the punctures)
define the permutations $(A,B,C)\to (B,A,C)$ and $(A,B,C)\to (A,C,B)$ respectively.
\end{convention}

\subsection{The simplest's pseudo-Anosov braid}
\label{ex:s1invs2}
We consider first the homeomorphism $\psi = f_{\sigma_{1}^{-1}\sigma_{2}}$ and treat this example in detail.

\subsubsection{Invariant train track}

It is well known that the isotopy class of $\psi$ is pseudo-Anosov. Indeed the homeomorphism
$\psi$ leaves invariant the train track $\tau_{0}$ presented in Figure~\ref{fg1}.
The map $f_{\sigma_{1}^{-1}\sigma_{2}}$ is then represented by the train track map $T : \tau_{0} \to \tau_{0}$,
defined by $a\to aab$ and $b\to ab$. The incidence matrix, $\left(\begin{smallmatrix}
2 & 1 \\ 1 & 1 \end{smallmatrix}\right)$, is irreducible.
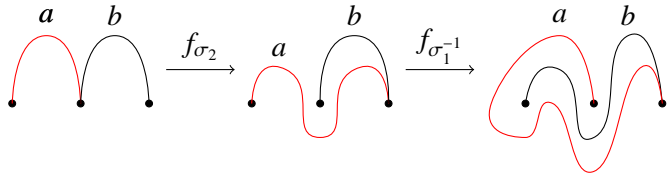
\begin{figure}[htb]

\begin{tikzpicture}[scale=0.9]


\foreach \x in {4.5,3.5,2.5}
	\draw [fill=black] (-\x,0) circle (0.05);
\draw [red, smooth] plot [tension=1.75] coordinates{(-4.5,0) (-4,1) (-3.5,0)};	
\draw [ smooth] plot [tension=1.75] coordinates{(-3.5,0) (-3,1) (-2.5,0)};

\foreach \x in {4.5,3.5,2.5}
	\draw [fill=black, xshift=3.5cm] (-\x,0) circle (0.05);

\draw[red,smooth] plot [tension=1.75]
coordinates {(-1,0) (-.5,.5) (0,-.5) (.5,.5) (1,0) };
\draw[smooth] plot [tension=2]
coordinates {(0,0) (.5,1) (1,0) };	
	
\foreach \x in {4.5,3.5,2.5}
	\draw [fill=black, xshift=7.5cm] (-\x,0) circle (0.05);
	\draw[smooth] plot [tension=1.5] coordinates {(3,0) (3.5,.5) (4,-.5) (4.5,1) (5,0)};
\draw[red, smooth] plot[tension=1] coordinates {(4,0) (3.5,1) (2.5,0)
	(3,-.5) (3.4,0)  (4,-1)  (4.65,.5) (5,0)};

\draw (-4,1.3) node{$a$};
\draw (3.5,1.3) node{$a$};
\draw (-.6,.8) node{$a$};
\draw (-4,1.3) node{$a$};
\foreach \y in {-3,.5,4.5}
\draw (\y,1.3) node{$b$};
\draw [->](-2.25,.5)--(-1.25,.5);
\draw [xshift=3.5cm,->](-2.25,.5)--(-1.25,.5);
\draw (-1.75,.85) node {$f_{\sigma_{2}}$};
\draw [xshift=3.5cm]  (-1.75,.85) node {$f_{\sigma_{1}^{-1}}$};

\end{tikzpicture}

\caption{An invariant train track $\tau_{0}$ for $f_{\sigma_{1}^{-1}\sigma_{2}}$.}
\label{fg1}
\end{figure}

We now quickly review how one can find a sequence of foldings discussed in the previous sections. For this purpose, 
consider the folding automaton and the two folding maps $F_{ba}$, $F_{ab}$ corresponding to the two standardizing homeomorphisms 
$f_{\sigma_{1}^{-1}}$ and $f_{\sigma_{2}}$ depicted in Figure~\ref{isotopy}.

\begin{figure}[htb]
\begin{tikzpicture}[scale=0.9]




\foreach \x in {4.5,3.5,2.5}
	\draw[xshift=3.5cm, yshift=3cm,smooth, red] (-\x,0) to[out=270,in=90] (-\x-.1,-.1) to[out=270,in=180] (-\x,-.25) to[out=0,in=270] (-\x+.1,-.1) to[out=90,in=270] (-\x,0);

\draw[yshift=3cm] (-1,0) to[out=90,in=180] (-.5,1) to[out=0,in=90] (.25,0) to[out=270,in=0] (0,-.35) to[out=180,in=270] (-.25,0) to[out=90,in=180] (-.15,.2) to[out=0,in=90] (0,0);

\draw[yshift=3cm] (0,0) to[out=90,in=0] (-.25,.3)
to[out=180,in=90] (-.5,0) to[out=270,in=180] (0,-.5) to[out=0,in=270] (.5,0) to[out=90,in=180] (.75,1) to[out=0,in=90] (1,0);



\foreach \x in {4.5,3.5,2.5}
	\draw[xshift=3.5cm, yshift=-3cm,smooth, red] (-\x,0) to[out=270,in=90] (-\x-.1,-.1) to[out=270,in=180] (-\x,-.25) to[out=0,in=270] (-\x+.1,-.1) to[out=90,in=270] (-\x,0);

\draw[xscale=-1,yshift=-3cm] (-1,0) to[out=90,in=180] (-.5,1) to[out=0,in=90] (.25,0) to[out=270,in=0] (0,-.35) to[out=180,in=270] (-.25,0) to[out=90,in=180] (-.15,.2) to[out=0,in=90] (0,0);

\draw[xscale=-1,yshift=-3cm] (0,0) to[out=90,in=0] (-.25,.3)
to[out=180,in=90] (-.5,0) to[out=270,in=180] (0,-.5) to[out=0,in=270] (.5,0) to[out=90,in=180] (.75,1) to[out=0,in=90] (1,0);


\foreach \x in {4.5,3.5,2.5}
	\draw[smooth, red] (-\x,0) to[out=270,in=90] (-\x-.1,-.1) to[out=270,in=180] (-\x,-.25) to[out=0,in=270] (-\x+.1,-.1) to[out=90,in=270] (-\x,0);
\draw [smooth] plot [tension=1.75] coordinates{(-4.5,0) (-4,1) (-3.5,0)};	

\draw  (-2.5,0) to[out=90,in=0] 
(-3.5,1.5)  to[out=180,in=90]
(-4.7,0)  to[out=270,in=180]
(-4.5,-.4)  to[out=0,in=270]
(-4.3,0)  to[out=90,in=0]
(-4.4,.2)    to[out=180,in=90]
(-4.5,0);

\filldraw[fill=white!90!black,smooth]
(-1.2,-.4) rectangle (1.2,1.5);

\foreach \x in {4.5,3.5,2.5}
	\draw[thick, xshift=-3.5cm,smooth, red] (\x,0) to[out=270,in=90] (\x-.1,-.1) to[out=270,in=180] (\x,-.25) to[out=0,in=270] (\x+.1,-.1) to[out=90,in=270] (\x,0);

\draw [xshift=8cm,smooth] plot [tension=1.75] coordinates{(-4.5,0) (-4,1) (-3.5,0)};	
\draw [thick,xshift=3.5cm,smooth] plot [tension=1.75] coordinates{(-4.5,0) (-4,1) (-3.5,0)};	
\draw [thick, xshift=3.5cm,smooth] plot [tension=1.75] coordinates{(-3.5,0) (-3,1) (-2.5,0)};	


\foreach \x in {4.5,3.5,2.5}
	\draw[smooth, red] (\x,0) to[out=270,in=90] (\x-.1,-.1) to[out=270,in=180] (\x,-.25) to[out=0,in=270] (\x+.1,-.1) to[out=90,in=270] (\x,0);
\draw [xshift=8cm,smooth] plot [tension=1.75] coordinates{(-4.5,0) (-4,1) (-3.5,0)};	

\draw [xscale=-1,yscale=1] (-2.5,0) to[out=90,in=0] 
(-3.5,1.5)  to[out=180,in=90]
(-4.7,0)  to[out=270,in=180]
(-4.5,-.4)  to[out=0,in=270]
(-4.3,0)  to[out=90,in=0]
(-4.4,.2)    to[out=180,in=90]
(-4.5,0);



\draw (-4,.8) node{$a$};
\draw (3.5,1.3) node{$a$};
\draw (-.5,1.3) node{$a$};

\draw (-3,1) node{$b$};
\draw (.5,1.3) node{$b$};
\draw (4,.7) node{$b$};
\draw [ ->] (-3,-1) to[out=300,in=180] (-1.5,-2.5);
\draw [xscale=-1,yscale=-1, ->] (-2.5,-1) to[out=270,in=180] (-1.5,-2.5);
\draw (-2,-1.5) node{$f_{\sigma_{1}^{-1}}$};
\draw [xscale=1,yscale=-1,yshift=-.3cm,<-] (-2.2,-.2) to[out=340,in=200] (-1.4,-.2);
\draw (-1.65,.9) node{$F_{ba}$};
\draw (2.5,2.5) node{$f_{\sigma_{2}}$};
\draw [xshift=3.5cm,xscale=1,yscale=-1,yshift=-.3cm,->] (-2.2,-.2) to[out=340,in=200] (-1.2,-.2);
\draw [xshift=3.5cm] (-1.65,.9) node{$F_{ab}$};

\draw[smooth,->] (1.1,-1.9) to[out=70,in=290] (1.1,-.6);

\draw[yshift=.5cm,xscale=-1,yscale=-1,smooth,->] (1.4,-2.3) to[out=70,in=290] (1.4,-.6);

\draw (1.7,-1) node{$f_{\mathrm rot}$};
\draw (-2,2) node{$f^{-1}_{\mathrm rot}$};
\draw (.5,-1.7) node{$b$};
\draw (-0.8,-1.7) node{$a$};
\draw[yshift=6cm] (.5,-1.7) node{$b$};
\draw[yshift=6cm]  (-0.5,-1.7) node{$a$};

\end{tikzpicture}

\caption{The folding automaton for $B_{3}$. The map $f_{\mathrm rot}$ is 
an isotopy (rotation in the neighborhood near punctures). Observe that the folding 
$F_{ba}$ induces a train track map $T_{ab}$ that represents $f_{\sigma_1^{-1}}$.
The same is true for $F_{ba}$ with $f_{\sigma_2}$.}
\label{isotopy}
\end{figure}
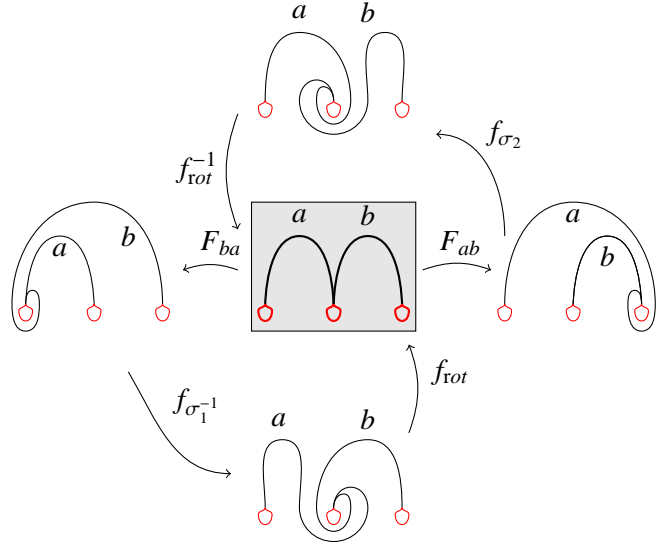

\begin{remark}
In the sequel, we will encode the folding automaton by representing the isotopy near the punctures by a permutation
(see Definition~\ref{def:standardizing} and Example~\ref{ex:standardizing}).
This defines a simpler automaton: see Figure~\ref{fg0}.
\end{remark}
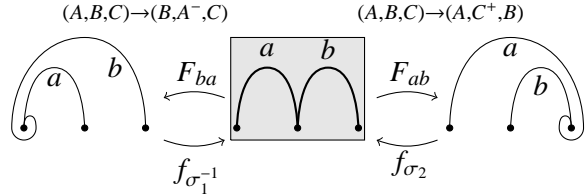
\begin{figure}[htb]

\begin{tikzpicture}[scale=0.8]


\foreach \x in {4.5,3.5,2.5}
	\draw [fill=black] (-\x,0) circle (0.05);
\draw [smooth] plot [tension=1.75] coordinates{(-4.5,0) (-4,1) (-3.5,0)};	

\draw  (-2.5,0) to[out=90,in=0] 
(-3.5,1.5)  to[out=180,in=90]
(-4.7,0)  to[out=270,in=180]
(-4.5,-.2)  to[out=0,in=270]
(-4.3,0)  to[out=90,in=0]
(-4.4,.2)    to[out=180,in=90]
(-4.5,0);


\filldraw[fill=white!90!black,smooth]
(-1.1,-.2) -- (1.1,-.2) -- (1.1,1.5) -- (-1.1,1.5)--(-1.1,-.2);

\foreach \x in {4.5,3.5,2.5}
	\draw [xshift=3.5cm, fill=black] (-\x,0) circle (0.05);
\draw [thick, xshift=3.5cm,smooth] plot [tension=1.75] coordinates{(-4.5,0) (-4,1) (-3.5,0)};	
\draw [thick, xshift=3.5cm,smooth] plot [tension=1.75] coordinates{(-3.5,0) (-3,1) (-2.5,0)};	


\foreach \x in {4.5,3.5,2.5}
	\draw [xshift=7cm, fill=black] (-\x,0) circle (0.05);

\draw [xshift=8cm,smooth] plot [tension=1.75] coordinates{(-4.5,0) (-4,1) (-3.5,0)};	

\draw [xscale=-1,yscale=1] (-2.5,0) to[out=90,in=0] 
(-3.5,1.5)  to[out=180,in=90]
(-4.7,0)  to[out=270,in=180]
(-4.5,-.2)  to[out=0,in=270]
(-4.3,0)  to[out=90,in=0]
(-4.4,.2)    to[out=180,in=90]
(-4.5,0);



\draw (-4,.8) node{$a$};
\draw (3.5,1.3) node{$a$};
\draw (-.5,1.3) node{$a$};

\draw (-3,1) node{$b$};
\draw (.5,1.3) node{$b$};
\draw (4,.7) node{$b$};
\draw [->] (-2.2,-.2) to[out=340,in=200] (-1.2,-.2);
\draw (-1.65,-.8) node{$f_{\sigma_{1}^{-1}}$};
\draw [xscale=1,yscale=-1,yshift=-.3cm,<-] (-2.2,-.2) to[out=340,in=200] (-1.2,-.2);
\draw (-1.65,.9) node{$F_{ba}$};
\draw (-2.5,1.9) node{$\scriptstyle (A,B,C)\to(B,A^{-},C)$};

\draw [xshift=3.5cm,<-] (-2.2,-.2) to[out=340,in=200] (-1.2,-.2);
\draw [xshift=3.5cm](-1.65,-.6) node{$f_{\sigma_{2}}$};
\draw [xshift=3.5cm,xscale=1,yscale=-1,yshift=-.3cm,->] (-2.2,-.2) to[out=340,in=200] (-1.2,-.2);
\draw [xshift=3.5cm] (-1.65,.9) node{$F_{ab}$};
\draw [xshift=4cm] (-1.65,1.9) node{$\scriptstyle (A,B,C)\to(A,C^+,B)$};

\end{tikzpicture}


\caption{The folding automaton for $B_{3}$. Note that the graphs $\mathcal{N}^{\mathrm{lab}}(\tau,h,\varepsilon)$ and 
$\mathcal{N}(\tau,h)$ coincide: it has only one vertex and two edges represented by the two folding maps
$F_{ab}$ and $F_{ba}$. For each edge we have represented the action of the 
standardizing braids on the punctures.
}
\label{fg0}
\end{figure}

The two foldings $F_{ba}$ and $F_{ab}$ define two train track maps $T_{ba}$ and $T_{ab}$ 
(representing the two homeomorphisms $f_{\sigma_{1}^{-1}}$ and $f_{\sigma_{2}}$, respectively):
$$
\begin{array}{ccc}
\tau_{0}\stackrel{T_{ba}}{\longrightarrow}\tau_{0} \hspace{2cm}& \hspace{2cm} \tau_{0}\stackrel{T_{ab}}{\longrightarrow}\tau_{0} \\
a\to a \hspace{2cm} & \hspace{2cm} a\to ab\\
b\to ba \hspace{2cm} & \hspace{2cm} b\to a
\end{array}
$$
whose incidence matrices are $M_{ba} := M(T_{ba})=\left(\begin{smallmatrix}
1 & 0 \\ 1 & 1 \end{smallmatrix}\right)$ and $M_{ab}:= M(T_{ab})= \left(\begin{smallmatrix}
1 & 1 \\ 0 & 1 \end{smallmatrix}\right)$.
To be more precise, one sees that in this very particular example $\tau_0$ is both 
invariant by $f_{\sigma_{2}}$ and $f_{\sigma_{1}^{-1}}$ and the associated train track maps 
are given by $T_{ab}$ and $T_{ba}$. Hence the path in the automaton representing 
$f_{\sigma_{1}^{-1}\sigma_{2}} = f_{\sigma_{1}^{-1}} \circ f_{\sigma_{2}}$ has the train track map given by 
$T_{ba}\circ T_{ab}$. Therefore the incidence matrix is 
$$
M(T_{ba}\circ T_{ab}) = M_{ab}\cdot M_{ba}= \begin{pmatrix}
2 & 1 \\
1 & 1
\end{pmatrix}.
$$
Observe that in this case the relabeling map involved is equal to the identity map.
In the above situation, matrices belong to $\mathrm{GL}(\Z^{\mathbb{A}_{\mathrm{real}}})=\mathrm{GL}(\Z^{\{a,b\}})$.

\subsubsection{The Teichm\"uller polynomial}
	\label{subsec:simplteichpoly}

We now compute the Teichm\"uller polynomial of the fibered face containing the fibration defined by the suspension of 
$f_{\sigma_{1}^{-1}\sigma_{2}}$. Recall that $\mathbf{D}_{3}$ is the complement of $3$ round discs $D_{A}$, $D_{B}$ and $D_{C}$ 
lying along a diameter of the closed unit disc. The rank of the group $H$ is given by the number of cycles of the permutation 
induced by the action of $f_{\sigma_{1}^{-1}\sigma_{2}}$ on the boundary $\{[\partial D_{\alpha}]\}_{\alpha \in \mathbb{A}_{\mathrm{prong}}}$. Since the braid $\beta=\sigma_{1}^{-1}\sigma_{2}$ permutes $3$ strands cyclically, $H$ is isomorphic to $\Z$. Therefore $\pi:\widetilde{\mathbf{D}_{3}}\to\mathbf{D}_{3}$ is a $\Z$-covering. The infinite surface $\widetilde{\mathbf{D}_{3}}$ can be constructed by glueing $\Z$ 
copies of the simply connected domain obtained by cutting $\mathbf{D}_{3}$ along three disjoint segments going from $D_{\alpha}$ to the exterior boundary of $\mathbf{D}_{3}$. 
These are the so called leaves of the covering  $\pi:\widetilde{\mathbf{D}_{3}}\to\mathbf{D}_{3}$ (see \S \ref{SLEO}). For our computations, we fix a labeling by $t\in \Z$ of the set of leaves forming $\widetilde{\mathbf{D}_{3}}$ that is coherent with the action of ${\rm Deck}(\pi)$. This labeling induces a labeling for the edges and vertices of the infinite train track $(\widetilde{\tau_{0}},\widetilde{h})$.
 
As noted before, the path in the automaton $\mathcal{N}(h,\tau_{0})$  representing $f_{\sigma_{1}^{-1}\sigma_{2}}$ is 
$T_{ba} \circ T_{ab}$. In Figure~\ref{fg3} we depict the lift of each factor in this path to $\widetilde{\mathbf{D}_{3}}$. 

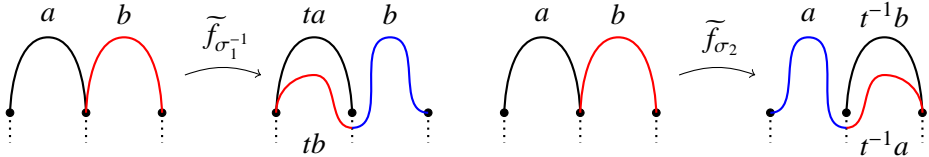
\begin{figure}[here]


\begin{tikzpicture}

\foreach \x in {4.5,3.5,2.5}
	\draw [xshift=-1.5cm, fill=black] (-\x,0) circle (0.05);
	\foreach \x in {4.5,3.5,2.5}
\draw [thick, xshift=-1.5cm,dotted] (-\x,0)--(-\x,-.4);\draw [thick,xshift=-1.5cm,smooth] plot [tension=1.75] coordinates{(-4.5,0) (-4,1) (-3.5,0)};	
\draw [red, thick, xshift=-1.5cm,smooth] plot [tension=1.75] coordinates{(-3.5,0) (-3,1) (-2.5,0)};	


	\foreach \x in {4.5,3.5,2.5}
\draw [thick, xshift=2cm,dotted] (-\x,0)--(-\x,-.4);

\foreach \x in {4.5,3.5,2.5}
	\draw [xshift=2cm, fill=black] (-\x,0) circle (0.05);

\draw [thick, xshift=1cm,smooth] plot [tension=1.75] coordinates{(-3.5,0) (-3,1) (-2.5,0)};	
\draw[thick,xshift=-1.5cm,red,smooth] (-1,0) to[out=90,in=180] (-.5,.5) to[out=0,in=180] (0,-.2);
\draw[thick,xshift=-1.5cm,blue,smooth]
(0,-.2) to[out=0,in=180]
(.5,1) to[out=0,in=180] (1,0);


\foreach \x in {4.5,3.5,2.5}
	\draw [xshift=5cm, fill=black] (-\x,0) circle (0.05);
	\foreach \x in {4.5,3.5,2.5}
\draw [thick, xshift=5cm,dotted] (-\x,0)--(-\x,-.4);\draw [thick,xshift=5cm,smooth] plot [tension=1.75] coordinates{(-4.5,0) (-4,1) (-3.5,0)};	
\draw [red, thick, xshift=5cm,smooth] plot [tension=1.75] coordinates{(-3.5,0) (-3,1) (-2.5,0)};	


	\foreach \x in {4.5,3.5,2.5}
\draw [xshift=6.5cm, thick, xshift=2cm,dotted] (-\x,0)--(-\x,-.4);

\foreach \x in {4.5,3.5,2.5}
	\draw [xshift=8.5cm, fill=black] (-\x,0) circle (0.05);

\draw [thick, xshift=8.5cm,smooth] plot [tension=1.75] coordinates{(-3.5,0) (-3,1) (-2.5,0)};	
\draw[xshift=3.5cm,xscale=-1,thick,xshift=-1.5cm,red,smooth] (-1,0) to[out=90,in=180] (-.5,.5) to[out=0,in=180] (0,-.2);
\draw[xshift=3.5cm,xscale=-1,thick,xshift=-1.5cm,blue,smooth]
(0,-.2) to[out=0,in=180]
(.5,1) to[out=0,in=180] (1,0);



\draw (-5.5,1.3) node{$a$};
\draw[xshift=6.5cm] (-5.5,1.3) node{$a$};
\draw[xshift=10cm] (-5.5,1.3) node{$a$};
\draw (-2,1.3) node{$ta$};
\draw (5.5,-.4) node{$t^{-1}a$};


\draw (-4.5,1.3) node{$b$};
\draw (-1,1.3) node{$b$};
\draw (-2,-.4) node{$tb$};
\draw (2,1.3) node{$b$};
\draw (5.5,1.3) node{$t^{-1}b$};


\draw [xshift=-1.5cm,xscale=1,yscale=-1,yshift=-.3cm,->] (-2.2,-.2) to[out=340,in=200] (-1.2,-.2);
\draw [xshift=5cm,xscale=1,yscale=-1,yshift=-.3cm,->] (-2.2,-.2) to[out=340,in=200] (-1.2,-.2);

\draw (-3.15,1) node{$\widetilde{f}_{\sigma_{1}^{-1}}$};

\draw [xshift=6.5cm] (-3.15,1) node{$\widetilde{f}_{\sigma_{2}}$};

\end{tikzpicture}


\caption{
The lift of homeomorphisms induced by folding operations.}
\label{fg3}
\end{figure}
The first train track map $T_{ab}$ corresponds to the homeomorphism $f_{\sigma_{2}}$. We choose the lift 
$\widetilde{f_{\sigma_{2}}}$ of $f_{\sigma_{2}}$ that fixes the vertex $v_{\rm fix}$, which in figure \ref{fg3} is the vertex to the left. Equipped with this choice we get
a train track map $\widetilde{T_{ab}} : \widetilde{\tau_0} \rightarrow \widetilde{\tau_0}$ induced by $\widetilde{f_{\sigma_{2}}}(\widetilde{\tau_0}) \prec \widetilde{\tau_0}$. Similarly the train track map 
$T_{ba}$ represents the homeomorphism $f_{\sigma^{-1}_{1}}$ and we choose the lift 
$\widetilde{f_{\sigma^{-1}_{1}}}$ of $f_{\sigma^{-1}_{1}}$ that fixes $v_{fix}$, which in figure  \ref{fg3} is the vertex to the right. This lift is represented by train track map $\widetilde{T_{ba}} : \widetilde{\tau_0} \rightarrow \widetilde{\tau_0}$. As we did in Example~\eqref{ex:first}, 
a direct calculation shows that
$$
\begin{array}{l}
w_{1}=\eta_{1}(1,C^{+})=(1,C^{+}), \textrm{ since } \eta_1(A,B,C) =(A,B,C),\\
w_{2}=\eta_{2}(A^{-1},1)=(A^{-1},1) \textrm{ since }\eta_{2}(A,B,C)=\pi_1\circ \eta_{1}(A,B,C)= (A,C,B),\\
\end{array}
$$
In the particular case of this example, all punctures are permuted cyclically, hence $t(w_{1})=(1,t)$ and $t(w_{2})=(t^{-1},1)$. Theorem~\ref{thm:second:main} then gives that the incidence matrix of the train track map $\widetilde{T}$ representing 
$\widetilde{f}_{\sigma_{1}^{-1}\sigma_{2}}$ is
$$
	\label{prods}
M(\widetilde{T})=M(\widetilde{T_{ba}} \circ \widetilde{T_{ab}})	= \begin{pmatrix}
1 & 1 \\
0 & 1
\end{pmatrix}
\begin{pmatrix}
1 & 0 \\
0 & t
\end{pmatrix}
\begin{pmatrix}
1 & 0 \\
1 & 1
\end{pmatrix}
\begin{pmatrix}
t^{-1} & 0 \\
0 & 1
\end{pmatrix}
=\begin{pmatrix}
1+t^{-1} & t \\
1 & t
\end{pmatrix}.
$$
Therefore the characteristic polynomial is
$$
\Theta_{F}(t,u)=u^{2}-(1+t+t^{-1})u+1
$$

\begin{remark}
From the preceding calculations it is easy to compute the Teichm\"uller polynomials associated to braids in $B_{3}$ that permute the strands cyclically (by considering products of $M(\widetilde{T_{ba}})$  and $M(\widetilde{T_{ab}}$).
Compare with~\cite[\S 11]{Mc}.
\end{remark}

\subsubsection{Evaluating the Teichm\"uller polynomial of $\sigma_{1}^{-1}\sigma_{2}$}
First given a class $f_{\beta}\in {\rm Mod}(\mathbf{D}_{n})$ we explain how to assign coordinates on 
$H^{1}(M;\Z)$ such that the cohomology class corresponding to the fibration defined $f_{\beta}$ is $(0,\ldots, 0,1)$. 
Following Section~\ref{SS:CNBC}, we choose an ordered basis $B=\{[m_{1}],\ldots, [m_{r}]\}$ of $H_{1}(M;\Z)$ 
formed by the meridians of the tori $T_{1},\ldots, T_{r}$ respectively. Given that 
$H_{1}(M;\Z)$ is torsion free, the base $B$ defines a base $B^{*}=\{[s_{1}],\ldots [s_{r-1}],[y]\}$ for $H^{1}(M;\Z)\simeq {\rm Hom}(H_{1}(M,\Z),\Z)$ by duality. Here $s_{i}=m_{i}^{*}$ for $i=1,\ldots,r-1$ and $m_{r}=y$. Let $[S_{r}]=i(\mathbf{D_{n}})$ denote the class of the fiber of the fibration defined by $f_{\beta}$. The intersection of $[S_{r}]$ with $[m_{i}]$ is given by $\delta_{rj}$ and hence, using the Universal Coefficient Theorem and Poincar\'e duality, we deduce that the coordinates of $[S_{r}]^{*}\in H^{1}(M;\Z)\simeq H_{2}(M,\partial M;\Z)$ for the basis $B^{*}$ are precisely $(0,\ldots,0,1)$. In the rest of the examples presented in this text we always choose the ordered basis $B^{*}$. Remark that $\{[m_{1}],\ldots,[m_{r-1}]\}$ generate the $f_{\beta}$-invariant homology of the $r-1$ punctured disc and $[m_{r}]$ corresponds to the natural lifting  $\widetilde{f_{\beta}}$ of $f_{\beta}$, hence we can identify $\{[m_{1}],\ldots,[m_{r-1}],[m_{r}]\}$ with the variables $\{t_{1},\ldots,t_{r-1},u\}$ of the Teichm\"uller polynomial  (see Section~\ref{SS:TPF}).
In the following figure we depict the link complement defined by  $\sigma_{1}^{-1}\sigma_{2}$. 

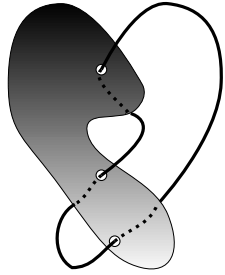
\begin{figure}[here]
  \begin{center}


\begin{tikzpicture}[scale=0.7]


\shadedraw[very thin, top color=black, bottom color=white!100!black, smooth]
plot [tension=1] coordinates{(.5,.5) (-1,2)  (-2,.5)  (-1,-1.5) (.5,-3) (1,-2) (-.5,-.5) (.4,0) (.5,.5)};

\draw [fill=white]  (-.25,.75) circle (.1);

\draw[very thick,smooth] plot [tension=1]
coordinates {(-.25,.75)  (1,2)  (2,.5)  (.85,-1.75)};

\draw [fill=white]  (-.25,-1.25) circle (.1);

\draw[very thick,smooth] plot [tension=1]
coordinates {(0.3,-0.08) (.5,-.5) (-.25,-1.25)};

\draw[dotted, very thick, smooth] (-.25,-1.25) to[out=220,in=30] (-.75,-1.8);

\draw [fill=white]   (0,-2.5) circle (.1);

\draw[smooth,very thick]   (-.75,-1.8) to[out=210,in=150] (-.8,-3) to[out=0,in=220]  (0,-2.5);

\draw[dotted,very thick, smooth] (0,-2.5) to[out=40,in=250] (.8,-1.8);

\draw[dotted,very thick] (-.25,.75) to[out=220, in=340] (0.2,-0.08);

\end{tikzpicture}


 \caption{The link $6_{2}^2$ and the fiber of the fibration defined by $\sigma^{-1}_{1}\sigma_{2}$.} \label{F:622}
 \end{center}
 \end{figure}
\par
We now determine the Thurston norm for the case $\beta=\sigma_{1}^{-1}\sigma_{2}$. We achieve 
this by computing first the Alexander norm of $M=M_{f_{\beta}}$. Direct computation 
shows that the Alexander polynomial of $\beta$ is
\begin{equation}
	\label{E:AlexPol}
\Delta_{M}(t,u)=u+u^{-1}-(-t^{-1}+1-t)
\end{equation}
(well defined up to multiplication by a unit in $\Z[G]$ (see \cite{McA}). 
The Newton polygon $N(\Delta_{M})$ of this polynomial is the symmetric diamond forming the convex hull of the points 
$\{(0,\pm1),(\pm1,0),(0,0)\}$; its Newton polytope is the square of vertices $\{(\pm\frac{1}{2},\pm\frac{1}{2})\}$.
By Remark~\ref{R:pdual} the unit ball of the Alexander norm is the square of vertices $\{(\pm\frac{1}{2},\pm\frac{1}{2})\}$
(in $H^1(M;\Z)$). Hence
$$
||(s,y)||_{A}=\max(|2s|,|2y|) \qquad \textrm{for all } (s,y)\in H^{1}(M;\Z).
$$
By theorems~\ref{T:TvsA}-\ref{T:FA} we conclude that the segment joining the points $(-\frac{1}{2},\frac{1}{2})$ and $(\frac{1}{2},\frac{1}{2})$ is the fibered face $F$ of the Thurston norm ball whose cone $\R^{+}\cdot F$ contains the fibration defined by $\sigma_{1}^{-1}\sigma_{2}$. Hence
$$
||(s,y)||_{T}=\max(|2s|,|2y|)
$$
for all fibrations $(s,y)\in \R^{+}\cdot F\cap H^{1}(M;\Z)$. \medskip

We finally explain how to evaluate $\Theta_{F}$ in a point $(s,y)\in H^{1}(M;\Z)$. Let 
$f_{s},f_{y}:H_{1}(M;\Z)\to\Z$ be the duals of $[s]$ and $[y]$ respectively. Hence, the dual of a point $(s,y)\in H^{1}(M;\Z)$ is given by $f_{(s,y)}:=sf_{s}+yf_{y}$. Since by definition $f_{(s,y)}(u)=y$ and $f_{(s,y)}(t)=s$ one has
$$
\begin{array}{ccc}
\Theta_{F}(s,y)& = &X^{f_{(s,y)}(u^{2})}-(1+X^{f_{(s,y)}(t)}+X^{f_{(s,y)}(t^{-1})})X^{f_{(s,y)}(u)}+1\\
\\
			& = & X^{2y}-(1+X^{s}+X^{-s})X^{y}+1
\end{array}
$$

\subsubsection{The topology of the fiber} 
Let $\Sigma$ be the fiber of the fibration determined by the point $(s,y)\in \R^{+}\cdot F$, where $F$ is the segment joining the points  $(-\frac{1}{2},\frac{1}{2})$ and $(\frac{1}{2},\frac{1}{2})$. Since every fiber is (Thurston) norm minimizing in its homology class, we have:
\begin{equation}
	\label{E:TOPFIB}
||(s,y)||_{T}=|y|=-\chi(\Sigma)=2\text{genus}(\Sigma)-2+\#\{\text{boundary components of }\Sigma\}.
\end{equation}
We calculate now the number of boundary components of $\Sigma$ as follows. 
We choose a basis $\{[S_{1}], [S_{2}]\}$ of $H_{2}(M,\partial M;\Z)$ by taking two 
Seifert surfaces of the components of the $6_{2}^{2}$ link shown in Figure~\ref{F:622}.
By Remark~\ref{R:links}, we have that
$$
\partial_{*}[S_{1}] = l_{1}-{\rm Lk}(L_{1},L_{2})m_{2}\hspace{5mm} \partial_{*}[S_{2}] = l_{2}-{\rm Lk}(L_{2},L_{1})m_{1}.
$$
A straightforward computation shows that $|{\rm Lk}(L_{1},L_{2})|=|{\rm Lk}(L_{2},L_{1})|=3$.
Letting $A=\left(\begin{smallmatrix} 0 & 3\\ 3 & 0 \end{smallmatrix}\right)$,
Proposition~\ref{prop:cc:boundary} implies that the number of connected components of $\partial \Sigma \cap T_j$
is $\gcd(a_j,b_j)$ where $a=(s,y)A=(3y,3s)$ and $b=(s,y)$. Therefore the total number of connected components of 
$\partial \Sigma$ is $\gcd(s,3y)+\gcd(3s,y)=\gcd(3,s)+\gcd(3,y)$. Plugging this data into~\eqref{E:TOPFIB} we get
$$
\text{genus}(\Sigma)=|y|+1-\frac{\gcd(3,s)+\gcd(3,y)}{2}.
$$
With notations of Corollary~\ref{cor:slope} the slope of the boundary components
are $3/s$ and $3/y$.

\subsubsection{The singularities of the fiber}
We already observed that $\Sigma$ has $\gcd(3,s)$ boundary components at $T_1$ and
$\gcd(3,y)$ boundary components at $T_2$. \medskip

For any singularity $s$ of $\F$ one needs to determine the slope of $\gamma_s \subset M$ where $\gamma_s$
is the closed orbit of the flow line passing throughout $s$. Since $\F$ has no singularities in the interior of $\mathbf{D}_{n}$
all curves $\gamma_s$ lie in $T_i$ for some $i$. We label the prongs with the capital letters $A,B,C$. One sees that the braid $\beta$
permutes the prongs $(A,B,C)$ to $(C,A,B)$. We denote the corresponding permutation $\pi(\beta)$. 
Since $\pi(\beta)$ has only one cycle, there is only one torus component (see Figure~\ref{torus}). Now when performing the pseudo-Anosov 
braid and isotopy, one needs to understand the rotation in the neighborhood near punctures (see Definition~\ref{def:standardizing}
and Example~\ref{def:standardizing}). \medskip

As usual one can obtain the permutation $\pi(\beta)$ as follows: for each elementary step we have 
a permutation encoding how the isotopy (rotation) acts in the neighborhood near punctures. More precisely:
$$
\pi(\sigma_2) : (A,B,C) \to (A,C^+,B) \qquad \textrm{and} \qquad \pi(\sigma^{-1}_1) : (A,B,C) \to (B,A^-,C).
$$
Composition gives the desired slope:
$$
\pi(\beta) = \pi(\sigma^{-1}_1\sigma_2) = \pi(\sigma_2) \circ \pi(\sigma^{-1}_1).
$$
Hence $\pi(\beta) : (A,B,C) \to (C^+,A^-,B)$ and so $\gamma = [l]$ {\em i.e.} its slope is $0/1$ (no Dehn twist around the meridian).

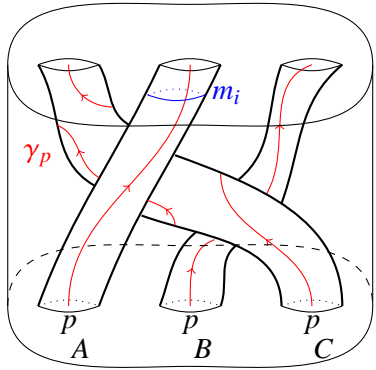
\begin{figure}[htb]
\tikzset{->-/.style={decoration={
  markings,
  mark=at position .5 with {\arrow{>}}},postaction={decorate}}}

\begin{tikzpicture}[scale=0.8]


\draw[smooth,dotted] (-2.5,0) to[out=20,in=160] (-1.5,0);
\draw[smooth] (-2.5,0) to[out=340,in=200] (-1.5,0);
\draw[smooth] (-.5,4) to[out=20,in=160] (.5,4);
\draw[smooth] (-.5,4) to[out=340,in=200] (.5,4);

\draw[red,->-=.8] (-2,0) to[out=90,in=270]  (0,4);

\draw[thick,smooth] (-1.5,0) to[out=70,in=245]  (.5,4);
\draw[thick,smooth] (-2.5,0) to[out=70,in=245] (-.5,4);


\draw[xshift=2cm,smooth,dotted] (-2.5,0) to[out=20,in=160] (-1.5,0);
\draw[xshift=2cm,smooth] (-2.5,0) to[out=340,in=200] (-1.5,0);
\draw[xshift=2cm,smooth] (-.5,4) to[out=20,in=160] (.5,4);
\draw[xshift=2cm,smooth] (-.5,4) to[out=340,in=200] (.5,4);

\draw[smooth,thick]  (-.5,0) to[out=70,in=220] (.1,1.23);
\draw[smooth,thick]  (.8,2.1) to[out=40,in=250] (1.5,4);
\draw[smooth,thick]  (.5,0) to[out=70,in=220] (.75,1);
\draw[smooth,thick]  (1.5,1.7) to[out=70,in=220] (2.5,4);

\draw[red,->-=.8] (0,0) to[out=90,in=200]  (0.4,1.1);
\draw[red,->-=.8] (1.25,1.85) to[out=60,in=200]  (2,4);


\draw[xshift=4cm,smooth,dotted] (-2.5,0) to[out=20,in=160] (-1.5,0);
\draw[xshift=4cm,smooth] (-2.5,0) to[out=340,in=200] (-1.5,0);
\draw[xshift=-2cm,smooth] (-.5,4) to[out=20,in=160] (.5,4);
\draw[xshift=-2cm,smooth] (-.5,4) to[out=340,in=200] (.5,4);

\draw[smooth,thick] (-2.5,4) to[out=300,in=150] (-1.58,2);
\draw[smooth,thick] (-.8,1.5) to[out=340,in=90] (1.5,0);

\draw[smooth,thick] (-1.5,4) to[out=300,in=150] (-1,3);
\draw[smooth,thick] (-.25,2.5) to[out=340,in=90] (2.5,0);

\draw[red,->-=.8] (2,0) to[out=90,in=270]  (0.5,2.2);
\draw[red,->-=.8] (-.25,1.33) to[out=90,in=340]  (-.7,1.75);
\draw[red,->-=.8] (-1.5,2.13) to[out=150,in=340]  (-2.2,3);
\draw[red,->-=.8] (-1.28,3.3) to[out=180,in=270]  (-2,4);


\draw[smooth] (-3,0) to[out=270, in=180] (0,-1) to[out=0,in=270] (3,0);
\draw[dashed, yscale=-1,smooth] (-3,0) to[out=270, in=180] (0,-1) to[out=0,in=270] (3,0);
\draw (-3,0) to (-3,4);


\draw[xshift=6cm] (-3,0) to (-3,4);
\draw[yshift=4cm,smooth] (-3,0) to[out=270, in=180] (0,-1) to[out=0,in=270] (3,0);
\draw[yshift=4cm,  yscale=-1,smooth] (-3,0) to[out=270, in=180] (0,-1) to[out=0,in=270] (3,0);


\draw[blue] (-.7,3.5) to[out=340,in=200 ](0.25,3.5);
\draw[dotted,blue] (-.7,3.5) to[out=20,in=160](0.25,3.5);


\draw[blue] (0.6,3.5) node {$m_{i}$};
\draw (0,-.3) node {$ p$};
\draw (-2,-.3) node {$ p$};
\draw (2,-.3) node {$ p$};

\draw (0.2,-.7) node {$ B$};
\draw (-1.8,-.7) node {$ A$};
\draw (2.2,-.7) node {$ C$};

\draw[red] (-2.5,2.5) node{$\gamma_{p}$};
\end{tikzpicture}

\caption{Computing the slope of the curve $\gamma_p$}
\label{torus}
\end{figure}

Concretely the slope of $\gamma_p$ is $p/q=0/1$. The slope for the other component
$T_2$ is $1/0$. We apply Proposition~\ref{prop:cc:singu}
at each connected component of $\partial\Sigma\cap T_{j}$ as follows.
\begin{enumerate}
\item For $T_1$ (coordinate $t$): one has $(a_1,b_1)=(3y,s)$. Thus at each connected component
(of the $\gcd(3y,s)$ components) there is a singularity of $\F_\phi$ of type $3y/\gcd(3,s)$-prong.
\item For $T_2$ (coordinate $t$): one has $(a_2,b_2)=(3s,y)$. Thus at each connected component
(of the $\gcd(3s,y)$ components) there is a singularity of $\F_\phi$ of type $y/\gcd(3,y)$-prong.
\end{enumerate}

\subsubsection{Orientability of singular foliation}

To compute the homological dilatation we will use the Alexander polynomial
$\Delta_{M}(t,u)=u^2-u(-t^{-1}+1-t) + 1$ (up to a factor). By Theorem~\ref{thm:homological},
the homological dilatation is the maximal root of $\Delta_M$ (in absolute value) evaluated in 
$(s,y)$, namely it is the maximal root (in absolute value) of the polynomial:
$$
Q(X) = X^{2y}-(1-X^{s}-X^{-s})X^{y}+1, \qquad y>s
$$
Recall that the stretch factor is the maximal root of
$$
P(X) = X^{2y}-(1+X^{s}+X^{-s})X^{y}+1.
$$
Since $Q(-X)=P(X)$ when $s$ is odd and $y$ is even, we draw that the invariant 
measured foliation is orientable if $s$ is odd and $y$ is even.

In the rest of this paper we will focus only on computing the Teichm\"uller polynomials, 
for the rest (Thurston norm, topology of fivers and type of singularities) can be performed using the methods presented for the simplest pseudo-Anosov braid. 
\subsection{The Teichm\"uller polynomial of $\sigma_{2}\sigma_{1}^{-1}\sigma_{2}\in B_{3}$} 

The link complement $M=\mathbf{S}^{3}\setminus L(\beta)$ of the braid $\beta=\sigma_{2}\sigma_{1}^{-1}\sigma_{2}$ is homeomorphic to the \emph{magic manifold} (see \cite{KT1} for more details). This braid fixes one strand and permutes the other two, hence the $H$-covering $\widetilde{\mathbf{D}_{3}}$ is a $\Z^{2}$-covering. Let us denote by $(t_{A},t_{B})$ the variables of the deck transformation group of $\pi:\widetilde{\mathbf{D}_{3}}\to\mathbf{D}_{3}$ corresponding to the permuted and fixed strands, respectively.
From the automaton presented in Figure~\ref{fg0} one sees that the path in the automaton $\mathcal{N}(h,\tau_{0})$ representing $f_{\sigma_{2}\sigma_{1}^{-1}\sigma_{2}}$ is the composition of three folding maps. By Theorem~\ref{thm:second:main}, one has
$$
\begin{array}{l}
w_{1}=\eta_{1}(1,C^{+})=(1,C^{+}), \textrm{ since } \eta_1(A,B,C) =(A,B,C),\\
w_{2}=\eta_{2}(A^{-1},1)=(A^{-1},1) \textrm{ since }\eta_{2}(A,B,C)= \pi_1\circ\eta_{1}(A,B,C)= (A,C,B),\\
w_{3}=\eta_{3}(1,C^{+})=(1,B^{+}),  \textrm{ since } \eta_{3}(A,B,C)=\pi_2\circ\pi_{1}(A,B,C) = (C,A,B).
\end{array}
$$
Hence $t(w_{1})=(1,t_{A})$, $t(w_{2})=(t^{-1}_{A},1)$ and $t(w_{3})=(1,t_{B})$ and the incidence matrix 
of the train track map $M(\widetilde{T})$ representing a lift of $f_{\sigma_{2}\sigma_{1}^{-1}\sigma_{2}}$ is
$$
M(\widetilde{T})=
\begin{pmatrix}
1 & t_{A} \\
0 & t_{A}
\end{pmatrix}
\begin{pmatrix}
t_{A}^{-1} & 0 \\
t_{A}^{-1} & 1
\end{pmatrix}
\begin{pmatrix}
1 & t_{B} \\
0 & t_{B}
\end{pmatrix}=
\begin{pmatrix}
t_{A}^{-1}+1 & t_{A}t_{B}+t_{B}+t_{B}t_{A}^{-1} \\
1 & t_{A}t_{B}+t_{B}
\end{pmatrix}.
$$
Taking the characteristic polynomial we get 
$$
 \Theta_{F}(t_{A},t_{B},u)=u^{2}-(t_{A}t_{B}+t_{B}+1+t_{A}^{-1})u+t_{B}.
$$

\section{The Teichm\"uller polynomial of $\sigma^{-1}_1\sigma_2 \sigma_3\in B_{4}$} 
\label{EXB4}

\subsection{Invariant train track}

The homeomorphism $f_{\sigma^{-1}_1\sigma_2 \sigma_3}$ is a pseudo-Anosov 
homeomorphism: it leaves invariant the train track $\tau_{1}$ (see Figure~\ref{fg6}) and 
the train track map $T : \tau_1 \rightarrow \tau_1$ induced by $f_{\sigma^{-1}_1\sigma_2 \sigma_3}(\tau_1)\prec\tau_1$
is given by
$$
\begin{array}{l}
a\to cbaa \\
b\to c \\
c\to d \\
d\to ba \\
\end{array}
$$
Its incidence matrix $M(T)=
\left(\begin{smallmatrix}
2 & 1 & 1 & 0 \\
0 & 0 & 1 & 0 \\
0 & 0 & 0 & 1 \\
1 & 1 & 0 & 0
\end{smallmatrix}\right)
$
is irreducible. In Figure~\ref{fg6} we depict part of the folding automaton $\mathcal{N}(\tau_{1},h)$. 

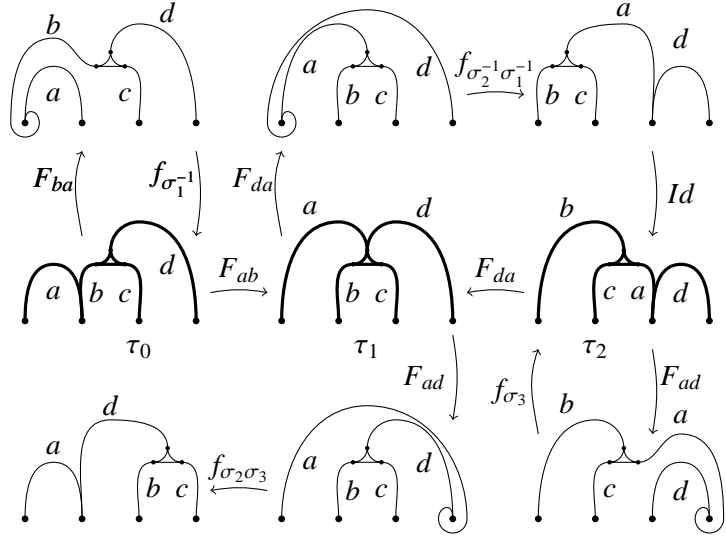
\begin{figure}[htbp]
 \begin{tikzpicture}[scale=0.75, inner sep=1mm]

\foreach \x in {1.5,.5,-.5,-1.5}
	\draw [fill=black] (-\x,0) circle (0.05);
\draw [fill=black] (-.25,1) circle (0.03);
\draw [fill=black] (.25,1) circle (0.03);
\draw [fill=black] (0,1.25) circle (0.03);
\draw [very thick, smooth] (-.25,1) to[out=0,in=270] (0,1.25) to[out=270,in=180] (.25,1) to[out=180,in=0] (-.25,1);
\draw [very thick, smooth] (-.5,0) to[out=90,in=180] (-.25,1); 
\draw [very thick, xscale=-1,smooth] (-.5,0) to[out=90,in=180] (-.25,1); 
\draw [very thick, smooth] (-1.5,0) to[out=90,in=180] (-.5,1.75) to[out=0,in=90] (0,1.25);
\draw [very thick, xscale=-1,smooth] (-1.5,0) to[out=90,in=180] (-.5,1.75) to[out=0,in=90] (0,1.25);
\draw (-1,2) node {$a$};
\draw (-.25,.5) node {$b$};
\draw (.25,.45) node {$c$};
\draw (1,2) node {$d$};
\draw (0,-.5) node {$\tau_1$};
 

\foreach \x in {1.5,.5,-.5,-1.5}
	\draw [xshift=-4.5cm,fill=black] (-\x,0) circle (0.05);
\draw [xshift=-4.5cm,fill=black] (-.25,1) circle (0.03);
\draw [xshift=-4.5cm,fill=black] (.25,1) circle (0.03);
\draw [xshift=-4.5cm,fill=black] (0,1.25) circle (0.03);
\draw [very thick, xshift=-4.5cm,smooth] (-.25,1) to[out=0,in=270] (0,1.25) to[out=270,in=180] (.25,1) to[out=180,in=0] (-.25,1);
\draw [very thick, smooth,xshift=-4.5cm] (-.5,0) to[out=90,in=180] (-.25,1); 
\draw [very thick, xshift=-4.5cm,xscale=-1,smooth] (-.5,0) to[out=90,in=180] (-.25,1); 
\draw[very thick, smooth] (-6,0) to[out=90,in=180] (-5.5,1) to[out=0,in=90] (-5,0);
\draw[very thick, smooth] (-4.5,1.25) to[out=90,in=180] (-4,1.75) to[out=0,in=90] (-3,0);
\draw (-5.5,.5) node {$a$};
\draw (-4.75,.5) node {$b$};
\draw (-4.25,.45) node {$c$};
\draw (-3.5,1) node {$d$};
\draw [thick] (-4,-.5) node {$\tau_0$};


\foreach \x in {1.5,.5,-.5,-1.5}
	\draw [xshift=4.5cm,fill=black] (-\x,0) circle (0.05);
\draw [xshift=4.5cm,fill=black] (-.25,1) circle (0.03);
\draw [xshift=4.5cm,fill=black] (.25,1) circle (0.03);
\draw [xshift=4.5cm,fill=black] (0,1.25) circle (0.03);
\draw [very thick, xscale=-1,xshift=-4.5cm,smooth] (-.25,1) to[out=0,in=270] (0,1.25) to[out=270,in=180] (.25,1) to[out=180,in=0] (-.25,1);
\draw [very thick, xscale=-1,smooth,xshift=-4.5cm] (-.5,0) to[out=90,in=180] (-.25,1); 
\draw [very thick, xscale=-1,xshift=-4.5cm,xscale=-1,smooth] (-.5,0) to[out=90,in=180] (-.25,1); 
\draw[very thick, xscale=-1,smooth] (-6,0) to[out=90,in=180] (-5.5,1) to[out=0,in=90] (-5,0);
\draw[very thick, xscale=-1,smooth] (-4.5,1.25) to[out=90,in=180] (-4,1.75) to[out=0,in=90] (-3,0);
\draw[xscale=-1,] (-5.5,.5) node {$d$};
\draw[xscale=-1,] (-4.75,.5) node {$a$};
\draw[xscale=-1,] (-4.25,.5) node {$c$};
\draw[xscale=-1,] (-3.5,2) node {$b$};
\draw (4,-.5) node {$\tau_2$};

\foreach \x in {1.5,.5,-.5,-1.5}
	\draw [yshift=3.5cm,xshift=-4.5cm,fill=black] (-\x,0) circle (0.05);
\draw [yshift=3.5cm,xshift=-4.5cm,fill=black] (-.25,1) circle (0.03);
\draw [yshift=3.5cm,xshift=-4.5cm,fill=black] (.25,1) circle (0.03);
\draw [yshift=3.5cm,xshift=-4.5cm,fill=black] (0,1.25) circle (0.03);
\draw [yshift=3.5cm,xshift=-4.5cm,smooth] (-.25,1) to[out=0,in=270] (0,1.25) to[out=270,in=180] (.25,1) to[out=180,in=0] (-.25,1);
\draw [smooth] (-4.75,4.5) to[out=180,in=0] (-5.5,5) to[out=180,in=90] (-6.25,3.5) to[out=270,in=180] (-6,3.25) to[out=0,in=270] (-5.75,3.5) to[out=90,in=0] (-5.90,3.7) to[out=180,in=90](-6,3.5);

\draw [yshift=3.5cm,xshift=-4.5cm,xscale=-1,smooth] (-.5,0) to[out=90,in=180] (-.25,1); 
\draw[yshift=3.5cm,smooth] (-6,0) to[out=90,in=180] (-5.5,1) to[out=0,in=90] (-5,0);
\draw[yshift=3.5cm,smooth] (-4.5,1.25) to[out=90,in=180] (-4,1.75) to[out=0,in=90] (-3,0);
\draw (-5.5,4) node {$a$};
\draw (-5.5,5.25) node {$b$};
\draw (-4.25,4) node {$c$};
\draw (-3.5,5.45) node {$d$};


\foreach \x in {1.5,.5,-.5,-1.5}
	\draw [yshift=3.5cm,fill=black] (-\x,0) circle (0.05);
\draw [yshift=3.5cm,fill=black] (-.25,1) circle (0.03);
\draw [yshift=3.5cm,fill=black] (.25,1) circle (0.03);
\draw [yshift=3.5cm,fill=black] (0,1.25) circle (0.03);
\draw [yshift=3.5cm,smooth] (-.25,1) to[out=0,in=270] (0,1.25) to[out=270,in=180] (.25,1) to[out=180,in=0] (-.25,1);
\draw [yshift=3.5cm,smooth] (-.5,0) to[out=90,in=180] (-.25,1); 
\draw [yshift=3.5cm,xscale=-1,smooth] (-.5,0) to[out=90,in=180] (-.25,1); 
\draw [yshift=3.5cm,smooth] (-1.5,0) to[out=90,in=180] (-.5,1.75) to[out=0,in=90] (0,1.25);
\draw [smooth] (1.5,3.5) to[out=90,in=0] (0,5.5) to[out=180,in=90](-1.75,3.5) to[out=270,in=180](-1.5,3.25) to[out=0,in=270](-1.25,3.5) to[out=90,in=0](-1.4,3.7) to[out=180,in=90](-1.5,3.5);

\draw (-1,4.5) node {$a$};
\draw (-.25,4) node {$b$};
\draw (.25,3.95) node {$c$};
\draw (1,4.5) node {$d$};


\foreach \x in {1.5,.5,-.5,-1.5}
	\draw [yshift=3.5cm,xshift=4.5cm,fill=black] (-\x,0) circle (0.05);
\draw [yshift=3.5cm,xshift=3.5cm,fill=black] (-.25,1) circle (0.03);
\draw [yshift=3.5cm,xshift=3.5cm,fill=black] (.25,1) circle (0.03);
\draw [yshift=3.5cm,xshift=3.5cm,fill=black] (0,1.25) circle (0.03);
\draw [yshift=3.5cm,xshift=3.5cm,smooth] (-.25,1) to[out=0,in=270] (0,1.25) to[out=270,in=180] (.25,1) to[out=180,in=0] (-.25,1);
\draw [yshift=3.5cm,xshift=3.5cm,smooth] (-.5,0) to[out=90,in=180] (-.25,1); 
\draw [yshift=3.5cm,xshift=3.5cm,xscale=-1,smooth] (-.5,0) to[out=90,in=180] (-.25,1); 
\draw[yshift=3.5cm,xshift=11cm,smooth] (-6,0) to[out=90,in=180] (-5.5,1) to[out=0,in=90] (-5,0);
\draw [xshift=-.5cm,smooth] (4,4.75) to[out=90,in=180] (5,5.25) to[out=0,in=90] (5.5,3.5);

\draw [yshift=3.5cm,xshift=5.5cm] (-1,2) node {$a$};
\draw [yshift=3.5cm,xshift=4.25cm]  (-1,.5) node {$b$};
\draw [yshift=3.5cm,xshift=2.75cm]  (1,.45) node {$c$};
\draw [yshift=3cm,xshift=4.5cm]  (1,2) node {$d$};


\foreach \x in {1.5,.5,-.5,-1.5}
	\draw [xscale=-1,yshift=-3.5cm,xshift=4.5cm,fill=black] (-\x,0) circle (0.05);
\draw [xscale=-1,yshift=-3.5cm,xshift=3.5cm,fill=black] (-.25,1) circle (0.03);
\draw [xscale=-1,yshift=-3.5cm,xshift=3.5cm,fill=black] (.25,1) circle (0.03);
\draw [xscale=-1,yshift=-3.5cm,xshift=3.5cm,fill=black] (0,1.25) circle (0.03);
\draw [xscale=-1,yshift=-3.5cm,xshift=3.5cm,smooth] (-.25,1) to[out=0,in=270] (0,1.25) to[out=270,in=180] (.25,1) to[out=180,in=0] (-.25,1);
\draw [xscale=-1,yshift=-3.5cm,xshift=3.5cm,smooth] (-.5,0) to[out=90,in=180] (-.25,1); 
\draw [xscale=-1,yshift=-3.5cm,xshift=3.5cm,xscale=-1,smooth] (-.5,0) to[out=90,in=180] (-.25,1); 
\draw[xscale=-1,yshift=-3.5cm,xshift=11cm,smooth] (-6,0) to[out=90,in=180] (-5.5,1) to[out=0,in=90] (-5,0);
\draw [yshift=-7cm,xscale=-1,xshift=-.5cm,smooth] (4,4.75) to[out=90,in=180] (5,5.25) to[out=0,in=90] (5.5,3.5);

\draw [xscale=-1,yshift=-3.5cm,xshift=5.5cm] (-1,2) node {$d$};
\draw [xscale=-1,yshift=-3.5cm,xshift=4.25cm]  (-1,.5) node {$c$};
\draw [xscale=-1,yshift=-3.4cm,xshift=2.75cm]  (1,.45) node {$b$};
\draw [xscale=-1,yshift=-4.25cm,xshift=4.5cm]  (1,2) node {$a$};


\foreach \x in {1.5,.5,-.5,-1.5}
	\draw [xscale=-1,yshift=-3.5cm,fill=black] (-\x,0) circle (0.05);
\draw [xscale=-1,yshift=-3.5cm,fill=black] (-.25,1) circle (0.03);
\draw [xscale=-1,yshift=-3.5cm,fill=black] (.25,1) circle (0.03);
\draw [xscale=-1,yshift=-3.5cm,fill=black] (0,1.25) circle (0.03);
\draw [xscale=-1,yshift=-3.5cm,smooth] (-.25,1) to[out=0,in=270] (0,1.25) to[out=270,in=180] (.25,1) to[out=180,in=0] (-.25,1);
\draw [xscale=-1,yshift=-3.5cm,smooth] (-.5,0) to[out=90,in=180] (-.25,1); 
\draw [xscale=-1,yshift=-3.5cm,xscale=-1,smooth] (-.5,0) to[out=90,in=180] (-.25,1); 
\draw [xscale=-1,yshift=-3.5cm,smooth] (-1.5,0) to[out=90,in=180] (-.5,1.75) to[out=0,in=90] (0,1.25);
\draw [xscale=-1,smooth,yshift=-7cm] (1.5,3.5) to[out=90,in=0] (0,5.5) to[out=180,in=90](-1.75,3.5) to[out=270,in=180](-1.5,3.25) to[out=0,in=270](-1.25,3.5) to[out=90,in=0](-1.4,3.7) to[out=180,in=90](-1.5,3.5);

\draw[xscale=-1,yshift=-7cm] (-1,4.5) node {$d$};
\draw[xscale=-1,yshift=-7cm] (-.25,4) node {$c$};
\draw[xscale=-1,yshift=-7cm] (.25,4) node {$b$};
\draw[xscale=-1,yshift=-7cm] (1,4.5) node {$a$};


\foreach \x in {1.5,.5,-.5,-1.5}
	\draw [yshift=-3.5cm,xshift=4.5cm,fill=black] (-\x,0) circle (0.05);
\draw [yshift=-3.5cm,xshift=4.5cm,fill=black] (-.25,1) circle (0.03);
\draw [yshift=-3.5cm,xshift=4.5cm,fill=black] (.25,1) circle (0.03);
\draw [yshift=-3.5cm,xshift=4.5cm,fill=black] (0,1.25) circle (0.03);
\draw [yshift=-3.5cm,xshift=4.5cm,smooth] (-.25,1) to[out=0,in=270] (0,1.25) to[out=270,in=180] (.25,1) to[out=180,in=0] (-.25,1);
\draw [smooth,xscale=-1,yshift=-7cm] (-4.75,4.5) to[out=180,in=0] (-5.5,5) to[out=180,in=90] (-6.25,3.5) to[out=270,in=180] (-6,3.25) to[out=0,in=270] (-5.75,3.5) to[out=90,in=0] (-5.90,3.7) to[out=180,in=90](-6,3.5);

\draw [yshift=-3.5cm,xshift=4.5cm,smooth] (-.5,0) to[out=90,in=180] (-.25,1); 
\draw[yshift=-3.5cm,xshift=11cm,smooth] (-6,0) to[out=90,in=180] (-5.5,1) to[out=0,in=90] (-5,0);
\draw[xscale=-1,yshift=-3.5cm,smooth] (-4.5,1.25) to[out=90,in=180] (-4,1.75) to[out=0,in=90] (-3,0);
\draw (5.5,-3) node {$d$};
\draw (5.5,-1.7) node {$a$};
\draw (4.25,-3) node {$c$};
\draw (3.5,-1.45) node {$b$};


\draw [->] (1.75,4) to[out=10,in=170] (2.75,4);
\draw (2.25,4.45) node{$f_{\sigma_{2}^{-1}\sigma_{1}^{-1}}$};

\draw [ ->, smooth] (-5,1.5) to[out=100,in=250] (-5,3); 
\draw (-5.5,2.5) node {$F_{ba}$};
\draw [<-, smooth] (-3,1.5) to[out=80,in=280] (-3,3);
\draw (-3.4,2.5) node {$f_{\sigma_{1}^{-1}}$};
\draw [xshift=3.5cm,->, smooth] (-5,1.5) to[out=100,in=250] (-5,3); 
\draw (-2,2.5) node {$F_{da}$};
\draw [xshift=8cm,<-, smooth] (-3,1.5) to[out=80,in=280] (-3,3);
\draw (5.5,2.25) node{$Id$};
\draw [->] (-2.75,.5) to[out=10,in=170] (-1.75,.5);
\draw (-2.25,.9) node{$F_{ab}$};
\draw [->] (2.75,.5) to[out=170,in=10] (1.75,.5);
\draw (2.25,.9) node{$F_{da}$};
\draw [ ->, xshift=8cm, yshift=-3.5cm,smooth] (-5,1.5) to[out=100,in=250] (-5,3); 
\draw (2.5,-1.25) node {$f_{\sigma_{3}}$};
\draw (-5.5,2.5) node {$F_{ba}$};
\draw [->] (1.5,-.25) to[out=280,in=80] (1.5,-1.75);
\draw (1,-1) node{$F_{ad}$};
\draw [<-, smooth,xshift=8cm, yshift=-3.5cm] (-3,1.5) to[out=80,in=280] (-3,3);
\draw (5.5,-1) node {$F_{ad}$};
\draw (-3.4,2.5) node {$f_{\sigma_{1}^{-1}}$};
\draw [yshift=-3.5cm,<-] (-2.75,.5) to[out=10,in=170] (-1.75,.5);
\draw  [yshift=-3.5cm] (-2.25,.9) node{$f_{\sigma_{2}\sigma_{3}}$};

\end{tikzpicture}
 \caption{
 \label{fg6}
Detail of foldings and standardizing braids in the automaton in $B_4$.}
\end{figure}

In this part of the automaton we see three vertices ($\tau_i$ for $i=1,2,3$, bolder train tracks).
More precisely the (standard) folding $F_{ab}$ induces a train track map $T_1 : \tau_1 \rightarrow \tau_2$ that represents 
$[\Id]\in{\rm Mod}(\mathbf{D}_{4})$. On the other hand the folding $F_{ad}$ induces a train track map 
$T_2 : \tau_2 \rightarrow \tau_1$ that represents $[f_{\sigma_2 \sigma_3}] \in{\rm Mod}(\mathbf{D}_{4})$. 
Finally $F_{ba}$ induces a train track map $T_3 : \tau_1 \rightarrow \tau_1$ that represents 
$[f_{\sigma^{-1}_1}]\in{\rm Mod}(\mathbf{D}_{4})$.  Hence the closed path representing $f_{\sigma^{-1}_1\sigma_2 \sigma_3}$ is given
by the sequence of train track maps in the labelled automaton
$$
(\tau_{1},\varepsilon_1)\stackrel{T_1}{\longrightarrow}(\tau_{2},\varepsilon_2)\stackrel{T_2}{\longrightarrow}(\tau_{1},\varepsilon_3)
\stackrel{T_3}{\longrightarrow}(\tau_{1},\varepsilon_3) \stackrel{R}{\longrightarrow}(\tau_{1},\varepsilon_1)
$$
with the relabeling $R:(\tau_{1},\varepsilon_3) \longrightarrow (\tau_{1},\varepsilon_1)$. Direct computations gives
(in the ordered basis $(a,b,c,d)$):
$$
M(T_1) = \left( \begin{smallmatrix} 1&1&0&0\\ 0&1&0&0 \\ 0&0&1&0 \\ 0&0&0&1  \end{smallmatrix} \right), \
M(T_2) = \left( \begin{smallmatrix} 1&0&0&1\\ 0&1&0&0 \\ 0&0&1&0 \\ 0&0&0&1  \end{smallmatrix} \right), \
M(T_3) = \left( \begin{smallmatrix} 1&0&0&0\\ 0&1&0&0 \\ 0&0&1&0 \\ 1&0&0&1  \end{smallmatrix} \right), \
M(R) = \left( \begin{smallmatrix} 1&0&0&0\\ 0&0&1&0 \\ 0&0&0&1 \\ 0&1&0&0  \end{smallmatrix} \right).
$$
We recover our incidence matrix $M(T)$ as $M(R\circ T_3\circ T_2\circ T_1)= M(T_1)M(T_2)M(T_3)M(R)$.

\begin{figure}[here]
\begin{tikzpicture}[scale=0.8]


\draw[thick,->] (0,-4.5) to[out=180,in=270] (-4,-1);

\filldraw[fill=white!90!black,smooth] (-1.7,-.2) rectangle (1.7,2.2);
\filldraw[yshift=-3.5cm,fill=white!90!black,smooth] (-1.7,-.2) rectangle (1.7,2.2);
\filldraw[yshift=-7cm,fill=white!90!black,smooth] (-1.7,-.2) rectangle (1.7,2.2);
\filldraw[xshift=-4.5cm,yshift=-3.5cm,fill=white!90!black,smooth] (-1.7,-.2) rectangle (1.7,2.2);
\filldraw[xshift=-4.5cm,yshift=-7cm,fill=white!90!black,smooth] (-1.7,-.2) rectangle (1.7,2.2);
\filldraw[xshift=-4.5cm,fill=white!90!black,smooth] (-1.7,-.2) rectangle (1.7,2.2);





\draw[yshift=.5cm, thick, ->,smooth] (-6.5,1) to[out=145,in=90] (-7,.5) to[out=270,in=210] (-6.5,0);

\draw[thick, smooth,->] (-3,1.5) to[out=45,in=180] (-2.25,1.75) to[out=0,in=135] (-1.5,1.5);


\draw[thick, smooth,->] (-1.75,-.3) to[out=225,in=45] (-2.75,-1.3) ;


\draw[yshift=-3cm, thick, ->,smooth] (-6.5,1) to[out=145,in=90] (-7,.5) to[out=270,in=210] (-6.5,0);

\draw[yshift=-3.5cm, thick, smooth,->] (-3,1.5) to[out=45,in=180] (-2.25,1.75) to[out=0,in=135] (-1.5,1.5);


\draw[yshift=-3.5cm,thick, smooth,->] (-1.75,-.3) to[out=225,in=45] (-2.75,-1.3) ;


\draw[yshift=-6.5cm, thick, ->,smooth] (-6.5,1) to[out=145,in=90] (-7,.5) to[out=270,in=210] (-6.5,0);

\draw[yshift=-7cm,thick, smooth,->] (-3,1.5) to[out=45,in=180] (-2.25,1.75) to[out=0,in=135] (-1.5,1.5);


\draw[dashed,thick,->] (0,-4.5) to[out=180,in=270] (-4,-1);


\draw (-7.5,1) node{$F_{ba}$};
\draw[yshift=-3.5cm] (-7.5,1) node{$F_{da}$};
\draw[yshift=-7cm] (-7.5,1) node{$F_{ca}$};

\draw (-2.25,1.4) node{$F_{ab}'$};
\draw[yshift=-3.5cm] (-2.25,1.4) node{$F_{ad}^{''}$};
\draw[yshift=-7cm] (-2.25,1.4) node{$F_{ac}^{''}$};

\draw (-1.6,-.8) node{$F_{ad}'$};
\draw (-3,-4.4) node{$F_{ac}'$};
\draw (.4,-4.4) node{$F_{ab}^{''}$};


\foreach \x in {1.5,.5,-.5,-1.5}
	\draw [fill=black] (-\x,0) circle (0.05);
\draw [fill=black] (-.25,1) circle (0.03);
\draw [fill=black] (.25,1) circle (0.03);
\draw [fill=black] (0,1.25) circle (0.03);
\draw [very thick, smooth] (-.25,1) to[out=0,in=270] (0,1.25) to[out=270,in=180] (.25,1) to[out=180,in=0] (-.25,1);
\draw [very thick, smooth] (-.5,0) to[out=90,in=180] (-.25,1); 
\draw [very thick, xscale=-1,smooth] (-.5,0) to[out=90,in=180] (-.25,1); 
\draw [very thick, smooth] (-1.5,0) to[out=90,in=180] (-.5,1.75) to[out=0,in=90] (0,1.25);
\draw [very thick, xscale=-1,smooth] (-1.5,0) to[out=90,in=180] (-.5,1.75) to[out=0,in=90] (0,1.25);
\draw (-1,2) node {$a$};
\draw (-.25,.5) node {$b$};
\draw (.25,.45) node {$c$};
\draw (1,2) node {$d$};


\foreach \x in {1.5,.5,-.5,-1.5}
	\draw [yshift=-3.5cm,fill=black] (-\x,0) circle (0.05);
\draw [yshift=-3.5cm,fill=black] (-.25,1) circle (0.03);
\draw [yshift=-3.5cm,fill=black] (.25,1) circle (0.03);
\draw [yshift=-3.5cm,fill=black] (0,1.25) circle (0.03);
\draw [yshift=-3.5cm,very thick, smooth] (-.25,1) to[out=0,in=270] (0,1.25) to[out=270,in=180] (.25,1) to[out=180,in=0] (-.25,1);
\draw [yshift=-3.5cm,very thick, smooth] (-.5,0) to[out=90,in=180] (-.25,1); 
\draw [yshift=-3.5cm,very thick, xscale=-1,smooth] (-.5,0) to[out=90,in=180] (-.25,1); 
\draw [yshift=-3.5cm,very thick, smooth] (-1.5,0) to[out=90,in=180] (-.5,1.75) to[out=0,in=90] (0,1.25);
\draw [yshift=-3.5cm,very thick, xscale=-1,smooth] (-1.5,0) to[out=90,in=180] (-.5,1.75) to[out=0,in=90] (0,1.25);
\draw[yshift=-3.5cm] (-1,2) node {$a$};
\draw[yshift=-3.5cm] (-.25,.5) node {$d$};
\draw[yshift=-3.5cm] (.25,.45) node {$b$};
\draw[yshift=-3.5cm] (1,2) node {$c$};


\foreach \x in {1.5,.5,-.5,-1.5}
	\draw [yshift=-7cm,fill=black] (-\x,0) circle (0.05);
\draw [yshift=-7cm,fill=black] (-.25,1) circle (0.03);
\draw [yshift=-7cm,fill=black] (.25,1) circle (0.03);
\draw [yshift=-7cm,fill=black] (0,1.25) circle (0.03);
\draw [yshift=-7cm,very thick, smooth] (-.25,1) to[out=0,in=270] (0,1.25) to[out=270,in=180] (.25,1) to[out=180,in=0] (-.25,1);
\draw [yshift=-7cm,very thick, smooth] (-.5,0) to[out=90,in=180] (-.25,1); 
\draw [yshift=-7cm,very thick, xscale=-1,smooth] (-.5,0) to[out=90,in=180] (-.25,1); 
\draw [yshift=-7cm,very thick, smooth] (-1.5,0) to[out=90,in=180] (-.5,1.75) to[out=0,in=90] (0,1.25);
\draw [yshift=-7cm,very thick, xscale=-1,smooth] (-1.5,0) to[out=90,in=180] (-.5,1.75) to[out=0,in=90] (0,1.25);
\draw[yshift=-7cm] (-1,2) node {$a$};
\draw[yshift=-7cm] (-.25,.5) node {$c$};
\draw[yshift=-7cm] (.25,.45) node {$
d$};
\draw[yshift=-7cm] (1,2) node {$b$};



\foreach \x in {1.5,.5,-.5,-1.5}
	\draw [xshift=-4.5cm,fill=black] (-\x,0) circle (0.05);
\draw [xshift=-4.5cm,fill=black] (-.25,1) circle (0.03);
\draw [xshift=-4.5cm,fill=black] (.25,1) circle (0.03);
\draw [xshift=-4.5cm,fill=black] (0,1.25) circle (0.03);
\draw [very thick, xshift=-4.5cm,smooth] (-.25,1) to[out=0,in=270] (0,1.25) to[out=270,in=180] (.25,1) to[out=180,in=0] (-.25,1);
\draw [very thick, smooth,xshift=-4.5cm] (-.5,0) to[out=90,in=180] (-.25,1); 
\draw [very thick, xshift=-4.5cm,xscale=-1,smooth] (-.5,0) to[out=90,in=180] (-.25,1); 
\draw[very thick, smooth] (-6,0) to[out=90,in=180] (-5.5,1) to[out=0,in=90] (-5,0);
\draw[very thick, smooth] (-4.5,1.25) to[out=90,in=180] (-4,1.75) to[out=0,in=90] (-3,0);
\draw (-5.5,.5) node {$a$};
\draw (-4.75,.5) node {$b$};
\draw (-4.25,.45) node {$c$};
\draw (-3.5,1) node {$d$};
\draw (-6,-.4) node{$A$};
\draw (-5,-.4) node{$B$};
\draw (-4,-.4) node{$C$};
\draw (-3,-.4) node{$D$};


\foreach \x in {1.5,.5,-.5,-1.5}
	\draw [yshift=-3.5cm,xshift=-4.5cm,fill=black] (-\x,0) circle (0.05);
\draw [yshift=-3.5cm,xshift=-4.5cm,fill=black] (-.25,1) circle (0.03);
\draw [yshift=-3.5cm,xshift=-4.5cm,fill=black] (.25,1) circle (0.03);
\draw [yshift=-3.5cm,xshift=-4.5cm,fill=black] (0,1.25) circle (0.03);
\draw [yshift=-3.5cm,very thick, xshift=-4.5cm,smooth] (-.25,1) to[out=0,in=270] (0,1.25) to[out=270,in=180] (.25,1) to[out=180,in=0] (-.25,1);
\draw [yshift=-3.5cm,very thick, smooth,xshift=-4.5cm] (-.5,0) to[out=90,in=180] (-.25,1); 
\draw [yshift=-3.5cm,very thick, xshift=-4.5cm,xscale=-1,smooth] (-.5,0) to[out=90,in=180] (-.25,1); 
\draw[yshift=-3.5cm,very thick, smooth] (-6,0) to[out=90,in=180] (-5.5,1) to[out=0,in=90] (-5,0);
\draw[yshift=-3.5cm,very thick, smooth] (-4.5,1.25) to[out=90,in=180] (-4,1.75) to[out=0,in=90] (-3,0);
\draw[yshift=-3.5cm] (-5.5,.5) node {$a$};
\draw[yshift=-3.5cm] (-4.75,.5) node {$d$};
\draw[yshift=-3.5cm] (-4.25,.45) node {$b$};
\draw[yshift=-3.5cm] (-3.5,1) node {$c$};


\foreach \x in {1.5,.5,-.5,-1.5}
	\draw [yshift=-7cm,xshift=-4.5cm,fill=black] (-\x,0) circle (0.05);
\draw [yshift=-7cm,xshift=-4.5cm,fill=black] (-.25,1) circle (0.03);
\draw [yshift=-7cm,xshift=-4.5cm,fill=black] (.25,1) circle (0.03);
\draw [yshift=-7cm,xshift=-4.5cm,fill=black] (0,1.25) circle (0.03);
\draw [yshift=-7cm,very thick, xshift=-4.5cm,smooth] (-.25,1) to[out=0,in=270] (0,1.25) to[out=270,in=180] (.25,1) to[out=180,in=0] (-.25,1);
\draw [yshift=-7cm,very thick, smooth,xshift=-4.5cm] (-.5,0) to[out=90,in=180] (-.25,1); 
\draw [yshift=-7cm,very thick, xshift=-4.5cm,xscale=-1,smooth] (-.5,0) to[out=90,in=180] (-.25,1); 
\draw[yshift=-7cm,very thick, smooth] (-6,0) to[out=90,in=180] (-5.5,1) to[out=0,in=90] (-5,0);
\draw[yshift=-7cm,very thick, smooth] (-4.5,1.25) to[out=90,in=180] (-4,1.75) to[out=0,in=90] (-3,0);
\draw[yshift=-7cm] (-5.5,.5) node {$a$};
\draw[yshift=-7cm] (-4.75,.5) node {$c$};
\draw[yshift=-7cm] (-4.25,.45) node {$d$};
\draw[yshift=-7cm] (-3.5,1) node {$b$};

\end{tikzpicture}

\caption{Detail of the labeled folding automaton in $B_4$.}
\label{fg666}
\end{figure}
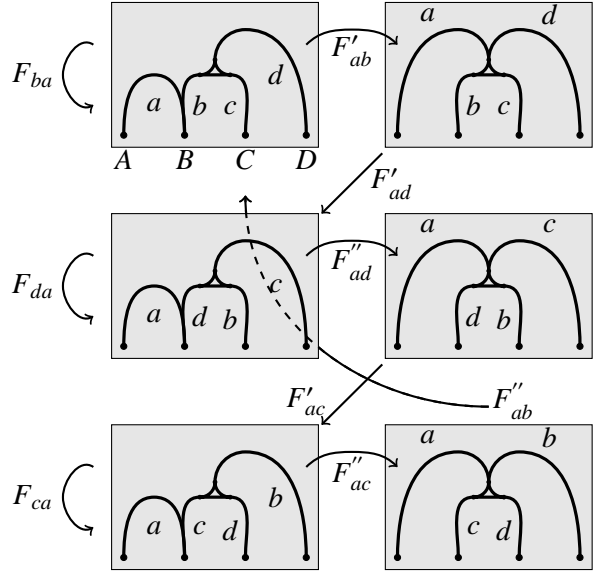

\subsection{Teichm\"uller polynomial}
We now calculate the Teichm\"uller polynomial of the fibered face containing the fibration defined by the suspension of 
$f_{\sigma^{-1}_1\sigma_2 \sigma_3}$. Since the braid permutes all the strands cyclically, $\pi:\widetilde{\mathbf{D}_{4}}\to\mathbf{D}_{4}$ is a 
$\Z$-covering. We now apply Theorem~\ref{thm:second:main} step by step.
\begin{enumerate}
\item The folding $F_{ab}$ is standard, hence $v_{1}=(1,1,1,1)$ and this implies that $D_{1}$ is the identity matrix.
\item The folding $F_{ad}$ is not standard. Following definition \ref{DEF:VECTV}, in $(\tau_{2},\varepsilon_{2})$ we
have $f=a$, $f'=d\in N(T_{2})$ hence we are in case 2. We conclude that $v_{2}=(1,D,D,D)$. Since $\sigma^{-1}_1\sigma_2 \sigma_3$ permutes the strands cyclically, we have that $w_{2}=(1,t,t,t)$. Therefore the matrix $D_{2}=\left( \begin{smallmatrix} 1&0&0&0\\ 0&t&0&0 \\ 0&0&t&0 \\ 0&0&0&t  \end{smallmatrix} \right)$.
\item The folding $F_{ba}$ is not standard. Following definition \ref{DEF:VECTV}, in$(\tau_{3},\varepsilon_{3})$ we
have $f=b$, $f'=a\notin N(T_{2})$ hence we are in case 1. We conclude that $v_{3}=(A^{-1},1,1,1)$. Since $\sigma^{-1}_1\sigma_2 \sigma_3$ permutes the strands cyclically , we have that $w_{2}=(t^{-1},1,1,1)$. Therefore the matrix $D_{3}=\left( \begin{smallmatrix} t^{-1}&0&0&0\\ 0&1&0&0 \\ 0&0&1&0 \\ 0&0&0&1  \end{smallmatrix} \right)$.

\end{enumerate}

Hence, the matrix whose characteristic polynomial is $\Theta_{F}$ is given by:
$$
M(\widetilde{T})=M(T_{1})\left( \begin{smallmatrix} 1&0&0&0\\ 0&1&0&0 \\ 0&0&1&0 \\ 0&0&0&1  \end{smallmatrix} \right)
M(T_{2})\left( \begin{smallmatrix} t^{-1}&0&0&0\\ 0&1&0&0 \\ 0&0&1&0 \\ 0&0&0&1  \end{smallmatrix} \right)M(T_{3})\left( \begin{smallmatrix} 1&0&0&0\\ 0&t&0&0 \\ 0&0&t&0 \\ 0&0&0&t  \end{smallmatrix} \right)M(R)=\begin{pmatrix}
1+t^{-1} & t & t & 0 \\
0 & 0 & t & 0 \\
0 & 0 & 0 & t \\
1 & t & 0 &0
\end{pmatrix}
$$
We conclude that the Teichm\"uller polynomial of $f_{\sigma^{-1}_1\sigma_2 \sigma_3}$ is:
$$
\Theta_{F}(t,u)=u^{4}-(1+t^{-1})u^{3}-(t^{2}+t^{3})u+t^{2}.
$$
This calculation can also be performed without the use of elementary operations:
$\widetilde{\tau_{1}}$ is $\widetilde{f_{\sigma^{-1}_1\sigma_2 \sigma_3}}$-invariant and the corresponding incidence matrix is precisely
$M(\widetilde{T})$ (see~Figure~\ref{fg9}).
 \begin{figure}[here]
\begin{tikzpicture}[scale=0.9]


\foreach \x in {1.5,.5,-.5,-1.5}
	\draw [xshift=-2.5cm,fill=black] (-\x,0) circle (0.05);
\draw [xshift=-2.5cm,fill=black] (-.25,1) circle (0.03);
\draw [xshift=-2.5cm,fill=black] (.25,1) circle (0.03);
\draw [xshift=-2.5cm,fill=black] (0,1.25) circle (0.03);
\draw [xshift=-2.5cm,smooth] (-.25,1) to[out=0,in=270] (0,1.25) to[out=270,in=180] (.25,1) to[out=180,in=0] (-.25,1);
\draw [ smooth,xshift=-2.5cm] (-.5,0) to[out=90,in=180] (-.25,1); 
\draw [xshift=-2.5cm,xscale=-1,smooth] (-.5,0) to[out=90,in=180] (-.25,1); 
\draw[xshift=2cm, smooth] (-6,0) to[out=90,in=180] (-5.5,1) to[out=0,in=90] (-5,0);
\draw[xshift=2cm, smooth] (-4.5,1.25) to[out=90,in=180] (-4,1.75) to[out=0,in=90] (-3,0);
\draw[xshift=2cm] (-5.5,.5) node {$a$};
\draw[xshift=2cm] (-4.75,.5) node {$b$};
\draw[xshift=2cm] (-4.25,.48) node {$c$};
\draw[xshift=2cm] (-3.5,1) node {$d$};
\foreach \x in {1,2,3,4}
\draw [thick,dotted] (-\x,0)--(-\x,-.4);


\foreach \x in {1.5,.5,-.5,-1.5}
	\draw [xshift=2.5cm,fill=black] (-\x,0) circle (0.05);
\draw [xshift=3.5cm,fill=black] (-.25,1) circle (0.03);
\draw [xshift=3.5cm,fill=black] (.25,1) circle (0.03);
\draw [xshift=3.5cm,fill=black] (0,1.25) circle (0.03);

\draw [xshift=3.5cm,smooth] (-.25,1) to[out=0,in=270] (0,1.25) to[out=270,in=180] (.25,1) to[out=180,in=0] (-.25,1);

\draw [ smooth,xshift=3.5cm] (-.5,0) to[out=90,in=180] (-.25,1); 
\draw [xshift=3.5cm,xscale=-1,smooth] (-.5,0) to[out=90,in=180] (-.25,1);
\draw[xscale=-1,xshift=+1cm, smooth] (-4.5,1.25) to[out=90,in=180] (-4,1.75) to[out=0,in=90] (-3,0);
\draw [red, smooth]
(1,0) to[out=90,in=180] (1.5,1) to[out=0,in=180] (2,-.25);
\draw [blue,smooth] (2,-.25) to[out=0,in=180]  (3,1.5) to[out=0, in=0 ] (3,1.25) to[out=180,in=90] (3,0);
\foreach \x in {1,2,3,4}
\draw [xshift=5cm,thick,dotted] (-\x,0)--(-\x,-.4);


\draw (1.5,1.25) node{$a$};
\draw (2.5,-.4) node{$\substack{\scalebox{0.8}{$t^{-1}a$}}$};
\draw (3.4,.5) node{$\substack{\scalebox{0.8}{$t^{-1}b$}}$}; 
\draw (4.4,.5) node{$\substack{\scalebox{0.8}{$t^{-1}c$}}$}; 
\draw (3,2) node{$\substack{\scalebox{0.8}{$t^{-1}d$}}$}; 


\foreach \x in {1.5,.5,-.5,-1.5}
	\draw [yshift=-3cm,xshift=-2.5cm,fill=black] (-\x,0) circle (0.05);
\draw [yshift=-3cm,xshift=-1.5cm,fill=black] (-.25,1) circle (0.03);
\draw [yshift=-3cm,xshift=-1.5cm,fill=black] (.25,1) circle (0.03);
\draw [yshift=-3cm,xshift=-1.5cm,fill=black] (0,1.25) circle (0.03);

\draw [yshift=-3cm,xshift=-1.5cm,smooth] (-.25,1) to[out=0,in=270] (0,1.25) to[out=270,in=180] (.25,1) to[out=180,in=0] (-.25,1);

\draw [smooth,yshift=-3cm,xshift=-1.5cm] (-.5,0) to[out=90,in=180] (-.25,1); 
\draw[xscale=-1,xshift=1.5cm,yshift=-3cm,smooth] (-.5,0) to[out=90,in=180] (-.25,1);
\draw[xscale=-1,xshift=+1cm, smooth] (-4.5,1.25) to[out=90,in=180] (-4,1.75) to[out=0,in=90] (-3,0);
\draw [yshift=-3cm,xshift=-5cm,red, smooth]
(1,0) to[out=90,in=180] (1.5,1) to[out=0,in=180] (2,-.25);
\draw [yshift=-3cm,xshift=-5cm,blue,smooth] (2,-.5) to[out=0,in=180]  (3,1.5) to[out=0, in=0 ] (3,1.25) to[out=180,in=90] (3,0);
\draw [blue,smooth] (-1.5,-1.75) to[out=90,in=0] (-2,-1) to[out=180,in=90] (-2.75,-2) to[out=270,in= 0] (-3,-3.25);
\draw [red,smooth] (-3,-3.5) to[out=180,in=0] (-3.75,-2.5) to[out=180,in=0] (-4,-3.25); 
\draw [blue,smooth] (-4,-3.25) to                    [out=180,in=180] (-3.5,-1.5) to[out=0,in=90] (-3,-3);

\foreach \x in {1,2,3,4}
\draw [yshift=-3cm,thick,dotted] (-\x,0)--(-\x,-.4);


\draw (-3.5,-1.25) node{$ta$};
\draw (-3.5,-2.25) node{$d$};
\draw (-3.5,-3.25) node{$a$};
\draw (-2.4,-3.5) node{$t^{-1}a$};
\draw (-1.2,-1) node{$t^{-1}d$};
\draw (-1.75,-2.5) node{$\substack{\scalebox{0.6}{$t^{-1}b$}}$}; 
\draw (-1.25,-2.5) node{$\substack{\scalebox{0.6}{$t^{-1}c$}}$};


\foreach \x in {1.5,.5,-.5,-1.5}
	\draw [yshift=-3cm,xshift=2.5cm,fill=black] (-\x,0) circle (0.05);
\draw [yshift=-3cm,xshift=2.5cm,fill=black] (-.25,1) circle (0.03);
\draw [yshift=-3cm,xshift=2.5cm,fill=black] (.25,1) circle (0.03);
\draw [yshift=-3cm,xshift=2.5cm,fill=black] (0,1.25) circle (0.03);

\draw [yshift=-3cm,xshift=2.5cm,smooth] (-.25,1) to[out=0,in=270] (0,1.25) to[out=270,in=180] (.25,1) to[out=180,in=0] (-.25,1);

\draw [yshift=-3cm,red, smooth]
(1,0) to[out=90,in=180] (1.5,1) to[out=0,in=180] (2,-.25);
\draw [xshift=5cm,red,smooth] (-3,-3.5) to[out=180,in=0] (-3.75,-2.5) to[out=180,in=0] (-4,-3.25); 
\draw [xshift=5cm,blue,smooth] (-4,-3.25) to[out=180,in=180] (-3.5,-1.5) to[out=0,in=90] (-3,-3);
\draw [blue,smooth] (2,-3.25) to[out=0,in=180] (2.25,-2);
\draw[smooth] (2,-3.5) to[out=0,in=180] (2.5,-2.6) to[out=0,in=90] (3,-3);
\draw[blue,smooth] (2.75,-2) to[out=0,in=90] (3,-3);
\draw[xshift=7cm, yshift=-3cm, smooth] (-4.5,1.25) to[out=90,in=180] (-4,1.75) to[out=0,in=90] (-3,0);

\foreach \x in {1,2,3,4}
\draw [yshift=-3cm, xshift=5cm,thick,dotted] (-\x,0)--(-\x,-.4);


\draw[xshift=5cm] (-3.5,-1.25) node{$ta$};
\draw[xshift=5cm] (-3.5,-2.25) node{$d$};
\draw[xshift=5cm] (-3.5,-3.25) node{$a$};
\draw[xshift=5cm] (-2.4,-3.5) node{$t^{-1}a$};
\draw (2.45,-2.25) node{$\substack{\scalebox{0.8}{$t^{-1}d$}}$};
\draw (3.4,-2.25) node{$\substack{\scalebox{0.8}{$t^{-1}b$}}$};
\draw (3.6,-1.1) node{$\substack{\scalebox{0.8}{$t^{-1}c$}}$};

\draw [->]  (-.75,.5) to[out=10,in=170] (.75,.5);
\draw (0,.9) node{$f_{\sigma_{2}\sigma_{3}}$};

\draw [->]  (-6,-2.5) to[out=10,in=170] (-4.5,-2.5);
\draw (-5,-2.1) node{$f_{\sigma_{1}^{-1}}$};

\draw [->]  (-.75,-2.5) to[out=10,in=170] (.6,-2.5);
\draw (0,-2.1) node{$\scriptstyle {\mathrm Isotopy}$};

\end{tikzpicture}

\caption{The lift of $f_{\sigma_3^{-1} \sigma_2^{-1}\sigma_1}$ to $\widetilde{\mathbf{D}_{4}}$}
\label{fg9}
\end{figure}
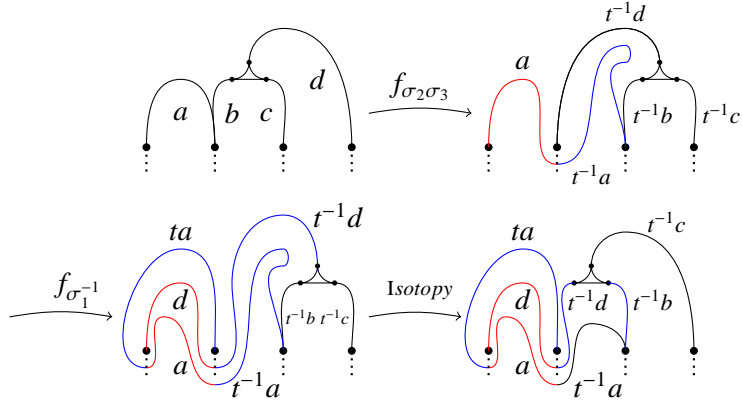

\appendix
\section{An infinite family of braids. }
\label{APENDIXA}
We consider, for each $n\in\N$, $n>0$, the braid $\beta_{n}\in B_{n+4}$ given by
$$
\beta_{n}=\delta_{n}\delta_{3}\sigma_{1}
$$
where $\delta_{n}=(\sigma_{1}\sigma_{2}\cdots\sigma_{n+3})^{-1}$. Consider the train track $(\tau_1,\varepsilon_1)$ given by Figure~\ref{fig10}. We have a train track map $(\tau_1,\varepsilon_1)\stackrel{T}{\to}(\tau_1,\varepsilon_1)$
 representing $f_{\beta_n}$. The loop
$$
(\tau_{1},\varepsilon_1)\stackrel{T_{1}}\longrightarrow(\tau_{2},\varepsilon_2)\stackrel{T_{2}}{\longrightarrow}\cdots
\stackrel{T_{n}}\longrightarrow(\tau_{n+1},\varepsilon_{n+1}) \stackrel{T_{n+1}}\longrightarrow(\tau_{1},\varepsilon_{n+1})
\stackrel{T_{n+2}}\longrightarrow (\tau_{1},\varepsilon_{n+1}) \stackrel{R}\longrightarrow
(\tau_{1},\varepsilon_1)
$$
where:
\begin{enumerate}
\item The train track map $T_{1}$ is induced by folding the edge labelled $a_{1}$ onto
the edge labelled $a_2$: It represents the braid $\sigma_{1}$, 
\item For every $i=2,\dots,n+1$ the train track morphism $T_{i}$ is induced by folding the edge labelled $a_{n+5-i}$ onto
the edge labelled $a_1$ and then applying
 a standardizing braid, and
\item $T_{n+2}$ is induced by the braid $\delta_n$ since $\delta^{-1}_n\circ h(\tau_1)$ is standard,
\end{enumerate}
represents the train track map T, that is, $T=R\circ T_{n+2}\circ T_{n+1}\circ\dots \circ T_2\circ T_1$. 

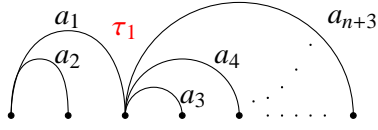
\begin{figure}[here]
\begin{tikzpicture}[scale=0.75]
\draw[smooth,xshift=1cm,yshift=3.5cm] (-8,0) to[out=90,in=180] (-7.5,1) to[out=0,in=90] (-7,0);
\draw[smooth,yshift=3.5cm,xshift=1cm] (-8,0) to[out=90,in=180] (-7,1.5) to[out=0,in=90] (-6,0);
\draw[smooth,yshift=3.5cm,xshift=1cm] (-6,0) to[out=90,in=180] (-5.5,0.5) to[out=0,in=90] (-5,0);
\draw[smooth,yshift=3.5cm,xshift=1cm] (-6,0) to[out=90,in=180] (-5,1) to[out=0,in=90] (-4,0);
\draw[smooth,yshift=3.5cm,xshift=1cm] (-6,0) to[out=90,in=180] (-4,2) to[out=0,in=90] (-2,0);


\draw[red,very thick] (-5,5) node{$\tau_{1}$};
\draw (-6,5.2) node{$a_{1}$};
\draw (-6,4.5) node{$a_{2}$};
\draw (-1,5.2) node{$a_{n+3}$};
\draw (-3.2,4.5) node{$a_{4}$};
\draw (-3.8,3.8) node{$a_{3}$};

\foreach \x in {1,3,4,5,6,7}
	\draw[fill=black] (-\x,3.5) circle (0.05);
\foreach \y in {0,0.25,0.5,0.75,1}
	\draw[xshift=-2.5cm, yshift=3.5cm,fill=black] (\y,0) circle (0.01);
\foreach \x in{0,0.25,0.5,0.75,1}
	\draw[xshift=-2.75cm, yshift=3.75cm,fill=black] (\x,\x*\x) circle (0.01);
	

\end{tikzpicture}
\caption{The train track $(\tau_1,\varepsilon_1)$.}
\label{fig10}
\end{figure}
We easily obtain:
$$
\begin{array}{l}
M(T_1) = \mathrm{Id}_{\mathbb{A}_{\mathrm{real}}} + E_{a_{1}a_{2}}, \\
M(T_i) = \mathrm{Id}_{\mathbb{A}_{\mathrm{real}}} + E_{a_{n+5-i}a_{1}} \textrm{ for } i=2,\dots,n+1,\\
M(T_{n+2}) = \mathrm{Id}_{\mathbb{A}_{\mathrm{real}}}
\end{array}
$$
where $E_{\alpha\beta}$ is a matrix having zeros in all entries except at position $(\alpha,\beta)$ where the entry is $1$.
We also draw (in the ordered basis $(a_1,a_2,\dots,a_{n+3})$):
$$
M(R)  =\left(
 \begin{array}{ccc|cccc}
1&0&0&0&\cdots&\cdots&0\\
0&0&0&0&\cdots&\cdots&1\\
0&1&0 &0&\cdots&\cdots & 0\\ \hline
0&0&1 & 0 & \cdots &\cdots & 0  \\
\vdots & \vdots & 0 & 1 &\cdots &\cdots & 0 \\
\vdots           &   \vdots         & \vdots & 0 & \ddots & 0 & 0\\
0 & 0 & 0 & 0 & \cdots  & 1 & 0 \\
\end{array}\right)_{(a_1,a_2,\dots,a_{n+3})}.
$$
Therefore the incidence matrix $M(T)$ of the train track map $T$ is $M(T_{1})M(T_{2})\dots M(T_{n+1})M(T_{n+2})M(R)$, 
whose characteristic polynomial is:
$$
P(X)=X^{n+3}-X^{n+2}-\ldots-X+1.
$$
We now calculate the Teichm\"uller polynomial of the fibered face $F$ containing the fibration defined by the suspension of $f_{\beta_{n}}$. 
Since the braid permutes all the strands cyclically, $\pi:\widetilde{\mathbf{D}_{n}}\to\mathbf{D}_{n}$ is a $\Z$-covering and we fix a labeling by $t\in\Z$ of the set of leaves forming $\widetilde{\mathbf{D}_{n}}$ that is coherent with the action of ${\rm Deck}(\pi)$.

One needs to compute the vectors $w_i=\eta_i(v_i)$ for $i=1,\dots,n+1$. The first case is similar to the situation discussed
in other examples: $w_1 = v_1 = (1,B,1,\dots,1)$. Hence $t(w_1) = (1,t,1,\dots,1)$. \medskip

For the map $T_2$ one has $f=a_{n+3}$ and $f'=a_1$. On the other hand $f'\in N(T_2)$ thus we are in case $2$ 
hence $v_2 = (X^{-1},\dots,X^{-1},1)$. For the others vectors, for each $i=3,\dots,n+1$, we have $f=a_{n+5-i}$ and $f'=a_1$.
Since $N(T_i)=\emptyset$, $f'\not \in N(T_i)$ and we are in case $1$. Hence $(v_i)_{a_{n+5-i}} = X$ and 
$(v_i)_{\alpha}=1$ otherwise. Finally for $T_{n+2}$, one has $(v_{n+2})_\alpha = X^{-1}$ for every $\alpha$.
In conclusion, a straightforward computation shows:
$$
\left\{\begin{array}{l}
t(w_1) = (1,t,1,\dots,1), \\
t(w_2) = (t^{-1},\dots,t^{-1},1),\\
t(w_i) = (1,\dots,1,t,1,\dots,1), \textrm{ for } i=3,\dots,n+1,  \textrm{ and} \\
t(w_{n+2}) = (t^{-1},\dots,t^{-1}).
\end{array}
\right.
$$
where the entry $t$ in $t(w_i)$ occurs at position $n+5-i$.

We can apply Equation~\eqref{E:DetFormulaTh} to obtain $\Theta_{F}(t,u)=\det(u\cdot \Id-M)$
where $M = M(T_{1})D_1 \cdots M(T_{n+2})D_{n+2} M(R)$ where $D_i= {\rm Diag}(t(w_i))$.
Therefore we obtain
$$
M=\left(
 \begin{array}{ccc|cccc}
t^{-2}&0&0&0&\cdots&\cdots&t^{-1}\\
0&0&0&0&\cdots&\cdots&t^{-1}\\
0&t^{-2}&0 &0&\cdots&\cdots & 0\\ \hline
t^{-2}&0&t^{-1} & 0 & \cdots &\cdots & 0  \\
\vdots & \vdots & 0 & t^{-1} &\cdots &\cdots & 0 \\
\vdots           &   \vdots         & \vdots & 0 & \ddots & 0 & 0\\
t^{-2} & 0 & 0 & 0 & \cdots  & t^{-1} & 0 \\
\end{array}\right)
$$
and its associated characteristic polynomial, that is the Teichm\"uller polynomial of $f_{\beta}$:
$$
\Theta_{F}(t,u)=u^{n+3}-t^{-2}u^{n+2}-t^{-3}u^{n+1}-\ldots-t^{-(n+3)}u+t^{-(n+5)}.
$$

\begin{remark}
In figure \ref{fig11} we depict lifts of the maps $f_{\sigma_{1}}$ and $f_{\delta_{n}}$, that we denote by $\widetilde{f_{\sigma_{1}}}$, 
$\widetilde{f_{\delta_{n}}}$ respectively.
\end{remark}
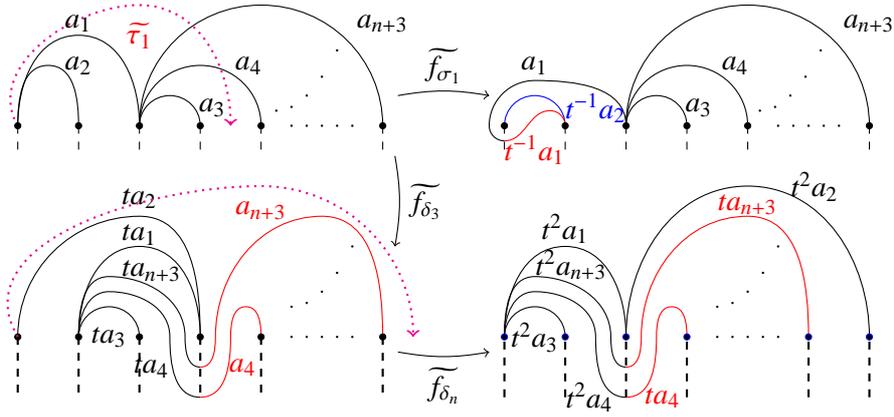
\begin{figure}[here]
\begin{tikzpicture}[scale=0.8]


\draw[thick,magenta,dotted,smooth,->] (-7,3.5) to [out=120,in=180] (-5,5.5) to[out=0,in=90] (-3.5,3.5);

\draw[red] (-5,5) node{$\widetilde{\tau_{1}}$};

\draw[smooth,xshift=1cm,yshift=3.5cm] (-8,0) to[out=90,in=180] (-7.5,1) to[out=0,in=90] (-7,0);
\draw[smooth,yshift=3.5cm,xshift=1cm] (-8,0) to[out=90,in=180] (-7,1.5) to[out=0,in=90] (-6,0);
\draw[smooth,yshift=3.5cm,xshift=1cm] (-6,0) to[out=90,in=180] (-5.5,0.5) to[out=0,in=90] (-5,0);
\draw[smooth,yshift=3.5cm,xshift=1cm] (-6,0) to[out=90,in=180] (-5,1) to[out=0,in=90] (-4,0);
\draw[smooth,yshift=3.5cm,xshift=1cm] (-6,0) to[out=90,in=180] (-4,2) to[out=0,in=90] (-2,0);


\draw (-6,5.2) node{$a_{1}$};
\draw (-6,4.5) node{$a_{2}$};
\draw (-1,5.2) node{$a_{n+3}$};
\draw (-3.2,4.5) node{$a_{4}$};
\draw (-3.8,3.8) node{$a_{3}$};

\foreach \x in {1,3,4,5,6,7}
	\draw[fill=black] (-\x,3.5) circle (0.05);
\foreach \y in {0,0.25,0.5,0.75,1}
	\draw[xshift=-2.5cm, yshift=3.5cm,fill=black] (\y,0) circle (0.01);
\foreach \x in{0,0.25,0.5,0.75,1}
	\draw[xshift=-2.75cm, yshift=3.75cm,fill=black] (\x,\x*\x) circle (0.01);
	
	\foreach \x in {1,3,4,5,6,7}
	\draw[dashed] (-\x,3.5) to (-\x,3);
	
\draw[smooth,->] (-.75,4)  to[out=10,in=170] (.75,4);
\draw (0,4.5) node {$\widetilde{f_{\sigma_{1}}}$};	
\draw[smooth,->] (-.8,3)  to[out=280,in=80] (-.8,1.5);\draw (-0.3,2.25) node {$\widetilde{f_{\delta_3}}$};		
	
\draw[blue,smooth,xshift=9cm,yshift=3.5cm] (-8,0) to[out=90,in=180] (-7.5,.5) to[out=0,in=90] (-7,0);
\draw[smooth,xshift=-1cm,yshift=3.5cm] (4,0) to[out=90,in=0] (2.5,.75) to[out=180,in=90] (1.75,0) to[out=270,in=180] (2,-0.25);
\draw[red,smooth,xshift=-1cm,yshift=3.5cm]
(2,-0.25) to[out=0,in=180] (2.75,.25) to[out=0,in=90] (3,0); 

\draw[smooth, xshift=9cm,yshift=3.5cm] (-6,0) to[out=90,in=180] (-5.5,0.5) to[out=0,in=90] (-5,0);
\draw[smooth,xshift=9cm,yshift=3.5cm] (-6,0) to[out=90,in=180] (-5,1) to[out=0,in=90] (-4,0);
\draw[smooth,xshift=9cm,yshift=3.5cm] (-6,0) to[out=90,in=180] (-4,2) to[out=0,in=90] (-2,0);

\draw [blue,xshift=8cm,yshift=-2.4cm](-5.5,6.2) node{$t^{-1}a_{2}$}; 
\draw [red,xshift=8cm,yshift=-2.4cm](-6.5,5.5) node{$t^{-1}a_{1}$};
\draw[xshift=7.5cm] (-6,4.5) node{$a_{1}$};
\draw [xshift=8cm](-1,5.2) node{$a_{n+3}$};
\draw[xshift=8cm] (-3.2,4.5) node{$a_{4}$};
\draw [xshift=8cm](-3.8,3.8) node{$a_{3}$};

\foreach \x in {1,3,4,5,6,7}
	\draw[xshift=8cm,fill=black] (-\x,3.5) circle (0.05);
\foreach \y in {0,0.25,0.5,0.75,1}
	\draw[xshift=5.5cm, yshift=3.5cm,fill=black] (\y,0) circle (0.01);
\foreach \x in{0,0.25,0.5,0.75,1}
	\draw[xshift=5.2cm, yshift=3.75cm,fill=black] (\x,\x*\x) circle (0.01);
	\foreach \x in {1,3,4,5,6,7}
	\draw[xshift=8cm,dashed] (-\x,3.5) to (-\x,3);



\draw[smooth] (-7,0) to[out=90,in=180] (-5,2) to[out=0,in=90] (-4,0);
\draw[smooth] (-6,0) to[out=90,in=180] (-5,1.5) to[out=0,in=90] (-4,0);
\draw[smooth] (-6,0) to[out=90,in=180] (-5.5,.5) to[out=0,in=90] (-5,0);

\draw[smooth] (-6,0) to[out=90,in=180] (-5.5,.75) to[out=0,in=90]  (-4.5,0) to[out=270,in=180]  (-4,-1); 
\draw[red, smooth] (-4,-1) to[out=0,in=270]  (-3.5,0) to[out=90,in=180]  (-3.25,.5) to[out=0,in=90]  (-3,0);


\draw[smooth] (-6,0) to[out=90,in=180] (-5.5,1) to[out=0,in=90]  (-4.25,0) to[out=270,in=180]  (-4,-.5);
\draw[red,smooth] (-4,-.5) to[out=0,in=270]  (-3.75,0) to[out=90,in=180]  (-2,2) to[out=0,in=90]  (-1,0);
\foreach \x in {1,3,4,5,6,7}
\draw[fill=black] (-\x,0) circle (0.05);

	\foreach \x in {1,3,4,5,6,7}
	\draw[yshift=-4cm,dashed,thick] (-\x,4) to (-\x,3);

\foreach \y in {0,0.25,0.5,0.75,1}
	\draw[xshift=-2.5cm,fill=black] (\y,0) circle (0.01);
\foreach \x in{0,0.25,0.5,0.75,1}
	\draw[xshift=-2.5cm,yshift=.5cm,fill=black] (\x,\x*\x) circle (0.01);


\draw (-5,2.3) node {$ta_{2}$};
\draw (-5,1.7) node {$ta_{1}$};
\draw (-5.5,0) node {$ta_{3}$};
\draw (-4.8,-.5) node {$ta_{4}$};
\draw (-4.8,1.1) node {$ta_{n+3}$};
\draw[red] (-3.3,-.5) node {$a_{4}$};
\draw[red] (-3,2.1) node {$a_{n+3}$};


\draw[thick,dotted, smooth,magenta,->] (-7,0) to[out=120,in=180] (-3,2.5) to[out=0,in=90] (-.5,0);


\draw[smooth,->] (-.75,-.25)  to[out=350,in=190] (.75,-.25);
\draw (0,-.8) node {$\widetilde{f_{\delta_{n}}}$};



\draw[smooth] (3,0) to[out=90,in=180] (5,2.5) to[out=0,in=90] (7,0);

\draw[xshift=7cm,smooth] (-6,0) to[out=90,in=180] (-5,1.5) to[out=0,in=90] (-4,0);
\draw[xshift=7cm,smooth] (-6,0) to[out=90,in=180] (-5.5,.5) to[out=0,in=90] (-5,0);

\draw[xshift=7cm,smooth] (-6,0) to[out=90,in=180] (-5.5,.75) to[out=0,in=90]  (-4.5,0) to[out=270,in=180]  (-4,-1); 
\draw[xshift=7cm,red, smooth] (-4,-1) to[out=0,in=270]  (-3.5,0) to[out=90,in=180]  (-3.25,.5) to[out=0,in=90]  (-3,0);


\draw[xshift=7cm,smooth] (-6,0) to[out=90,in=180] (-5.5,1) to[out=0,in=90]  (-4.25,0) to[out=270,in=180]  (-4,-.5);
\draw[xshift=7cm,red,smooth] (-4,-.5) to[out=0,in=270]  (-3.75,0) to[out=90,in=180]  (-2,2) to[out=0,in=90]  (-1,0);



\draw (6.1,2.5) node {$t^{2}a_{2}$};
\draw[xshift=7cm] (-5,1.8) node {$t^{2}a_{1}$};
\draw[xshift=7cm] (-5.5,0) node {$t^{2}a_{3}$};
\draw (2.4,-1) node {$t^{2}a_{4}$};
\draw[xshift=7cm] (-4.9,1.2) node {$t^{2}a_{n+3}$};
\draw[red] (3.6,-1) node {$ta_{4}$};
\draw[xshift=7cm, red] (-2,2.2) node {$ta_{n+3}$};

\foreach \x in {1,2,3,4,6,7}
\draw[blue, fill=black] (\x,0) circle (0.05);

	\foreach \x in {1,2,3,4,6,7}
	\draw[dashed,thick] (\x,-1) to (\x,0);

\foreach \y in {0,0.25,0.5,0.75,1}
	\draw[xshift=4.5cm,fill=black] (\y,0) circle (0.01);
\foreach \x in{0,0.25,0.5,0.75,1}
	\draw[xshift=4.5cm,yshift=.5cm,fill=black] (\x,\x*\x) circle (0.01);
\end{tikzpicture}
\caption{The lifts $\widetilde{f_{\sigma_{1}}}$, $\widetilde{f_{\delta_{3}}}$ and $\widetilde{f_{\delta_{n}}}$. Dotted lines indicate the action of $\widetilde{f_{\delta_{3}}}$ and $\widetilde{f_{\delta_{n}}}$. }
\label{fig11}
\end{figure}

\section{Computing the Thurston norm} 	
	
The Thurston norm of a link complement can be computed directly in some simple examples (see for example \cite{Thur}). Our calculations will make use of the \emph{Alexander norm}. The definition of this norm makes use of the \emph{Alexander polynomial}.
As the Teichm\"uller polynomial, the \emph{Alexander polynomial} $\Delta_{M}=\sum_{g\in G} b_{g}\cdot g $ of $M$
is an element of the group ring $\Z[G]$, where $G=H_{1}(M,\Z)/\Tor$. The \emph{Alexander norm} 
is defined on $H^{1}(M;\R)$ by
\begin{equation}
   \label{E:AN}
   ||\alpha||_{A}:=\sup_{b_{g}\neq 0\neq b_{h}}\alpha(g-h)
\end{equation}
The next two theorems explain how the Alexander norm and Thurston norm are related.
\begin{theorem}[\cite{McA}]
\label{T:TvsA}
Let $M$ be a compact, orientable 3-manifold whose boundary, if any, is a union of tori. If $b_{1}(M)\geq 2$
then for all $[\alpha]\in H^{1}(M,\Z)$:
$$
||\alpha||_{A}\leq ||\alpha||_{T}.
$$
	Moreover, equality holds when $\alpha:\pi_{1}(M)\to\Z$ is represented by a fibration $\fib$, where $\Sigma$ has non-positive Euler characteristic. 
\end{theorem}
\begin{theorem}[\cite{Mc}]
	\label{T:FA}
Let $F$ be a fibered face in $H^{1}(M,\R)$ with $b_{1}(M)\geq 2$. Then we have:
\begin{enumerate}
\item  $F\subset A$ for a unique face $A$ of the Alexander unit norm ball, and
\item  $F=A$ and $\Delta_{M}$ divides $\Theta_{F}$ if the lamination $\mathcal{L}$ associated to $F$ is transversally orientable.
\end{enumerate}
\end{theorem}
In particular,  \emph{the Thurston and Alexander norms agree on integer classes in the cone over a fibered face of the Thurston norm ball}. The condition ``the lamination $\mathcal{L}$ associated to $F$ is transversally orientable'' is equivalent the following condition: there exist a fibration $\fib$ whose pseudo-Anosov monodromy fixes a projective measured lamination $[(l,\mu)]\in\mathbb{P}\mathcal{ML}(\Sigma)$ which is transversally orientable. Equivalently, this last condition is equivalent to the orientability of a train track $\tau$ carrying $l$. From these theorems we can deduce the following simple fact: if $b_{1}(M)=2$ and all faces of $B_{T}$ are fibered, then the Thurston and Alexander norms coincide. The effective calculation of the Alexander norm is possible thanks to the 
following obvious remark:
\begin{remark}\label{R:pdual}
Since the Alexander polynomial of a $3$-manifold is symmetric the Alexander norm ball is dual to 
the scale by of factor of $2$ of the Newton polytope of the Alexander polynomial.
\end{remark}

For the sake of completeness we end this section discussing the \emph{Teichm\"uller norm} and how it can also be used to calculate the Thurston norm. Fix a fibered face $F\subset H^{1}(M,\R)$ and let $\Theta_{F}=\sum_{g\in G} a_{g}\cdot g$ be the corresponding Teichm\"uller polynomial. The Teichm\"uller norm (relative to $F$) is defined by:
\begin{equation}
   \label{E:TN}
   ||\alpha||_{\Theta_{F}}:=\sup_{a_{g}\neq 0\neq a_{h}}\alpha(g-h)
\end{equation}
Compare with $\ref{E:AN}$. The unit ball $B_{\Theta_{F}}$ of  the Teichm\"uller norm is dual to the Newton polytope $N(\Theta_{F})$ of the Teichm\"uller polynomial \cite{Mc}. Moreover,
\begin{theorem}[\cite{Mc}]
	\label{T:FFTN}
	For any fibered face $F$ of the Thurston norm ball, there exists a face $D$ of the Teichm\"uller norm ball,
	$$
	D\subset \{[\alpha]\hspace{1mm}|\hspace{1mm} || [\alpha] ||_{\Theta_{F}}=1\}
	$$
	such that $\R^{+}\cdot F = \R^{+}\cdot D$.
\end{theorem}

\section{Basic types}
	\label{APPENDIX:CASES}

In figure \ref{fig:ListBasicCases} we present the basic types that remain to complete the proof of theorem \ref{thm:second:main}. To understand the picture it is important to consider:
\begin{enumerate}
\item For each basic type depicted in the figure we omit the basic type obtained by performing a reflection with respect to a vertical line. We have to take this  'reflected' basic types into consideration for the proof. 
\item At most three infinitesimal edges are depicted, nevertheless the types presented can live in any punctured disc.  
\item The little black dot on which  in some basic types the real edges are incident needs to be changed, when constructing a train track from the basic type, by either a vertex or a multigon formed by infinitesimal edges.  
\end{enumerate}

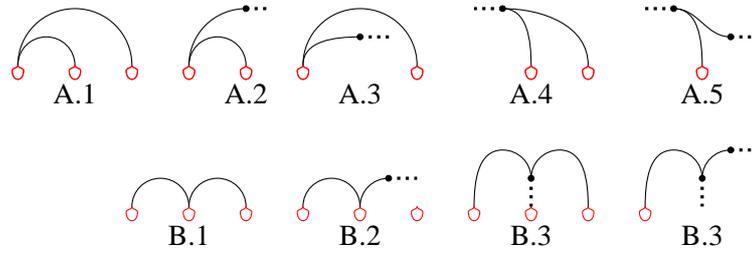
\begin{figure}[htbp]
 \begin{tikzpicture}[scale=0.75, inner sep=1mm]

\begin{scope}[xshift=-2cm]
\foreach \x in {-6,-5,-4,-3,-2,-1,1,3,4,6}
	\draw[smooth, red] (\x,0) to[out=270,in=90] (\x-.1,-.1) to[out=270,in=180] (\x,-.25) to[out=0,in=270] (\x+.1,-.1) to[out=90,in=270] (\x,0);

\foreach \x in {-6,-5,-4,-3,-2,-1,1,3,4,6}
	\draw[smooth, red] (\x,0) to[out=270,in=90] (\x-.1,-.1) to[out=270,in=180] (\x,-.25) to[out=0,in=270] (\x+.1,-.1) to[out=90,in=270] (\x,0);

\draw[smooth] (-6,0) to[out=90,in=180] (-5.5,.5)  to[out=0,in=90] (-5,0) ;
\draw[smooth] (-6,0) to[out=90,in=180] (-5,1)  to[out=0,in=90] (-4,0) ;
\draw(-5,-.5) node{A.1};

\draw[xshift=3cm,smooth] (-6,0) to[out=90,in=180] (-5.5,.5)  to[out=0,in=90] (-5,0) ;
\draw[smooth] (-3,0) to[out=90,in=180] (-2,1);
\draw [fill=black] (-2,1) circle (0.05);
\draw[ very thick, dotted,smooth] (-2,1) to (-1.6,1);
\draw(-2,-.5) node{A.2};
\end{scope}
\begin{scope}[xshift=-4cm]
\draw[smooth] (1,0) to[out=90,in=180] (2,1)  to[out=0,in=90] (3,0) ;
\draw[smooth] (1,0) to[out=90,in=180] (2,.5);
\draw [fill=black] (2,.5) circle (0.05);
\draw[ very thick, dotted,smooth] (2,.5) to (2.6,.5);
\draw(2,-.5) node{A.3};

\draw[smooth] (4.5,1) to[out=0,in=90] (5,0);
\draw[smooth] (4.5,1) to[out=0,in=90] (6,0);
\draw [fill=black] (4.5,1) circle (0.05);
\draw[ very thick, dotted,smooth] (4.5,1) to (4,1);
\draw(5,-.5) node{A.4};
\end{scope}
\begin{scope}[yshift=-2.5cm]
\foreach \x in {-6,-5,-4,-3,-2,-1,0,1,2,3}
	\draw[smooth, red] (\x,0) to[out=270,in=90] (\x-.1,-.1) to[out=270,in=180] (\x,-.25) to[out=0,in=270] (\x+.1,-.1) to[out=90,in=270] (\x,0);
\end{scope}
\begin{scope}[xshift=9cm,yshift=2.5cm]
\draw[smooth](-5.5,-1.5) to[out=0,in=90](-5,-2.5);
\draw[smooth](-5.5,-1.5) to[out=0,in=180](-4.5,-2);
\draw[very thick, dotted] (-6,-1.5) to (-5.5,-1.5);
\draw[very thick, dotted] (-4.5,-2) to (-4,-2);
\draw [fill=black] (-4.5,-2) circle (0.05);
\draw [fill=black] (-5.5,-1.5) circle (0.05);
\draw(-5,-3) node{A.5};
\end{scope}

\begin{scope}[xshift=-3cm]
\draw[smooth] (-3,-2.5) to[out=90,in=180] (-2.5,-2)  to[out=0,in=90] (-2,-2.5) ;
\draw[xshift=1cm,smooth] (-3,-2.5) to[out=90,in=180] (-2.5,-2)  to[out=0,in=90] (-2,-2.5) ;
\draw(-2,-3) node{B.1};
\end{scope}
\begin{scope}[xshift=-4cm]
\draw[smooth] (1,-2.5) to[out=90,in=180] (1.5,-2)  to[out=0,in=90] (2,-2.5) ;
\draw[smooth] (2,-2.5) to[out=90,in=180] (2.5,-2);
\draw [fill=black] (2.5,-2) circle (0.05);
\draw[very thick, dotted] (2.5,-2) to (3,-2);
\draw(2,-3) node{B.2};


\draw[smooth] (4,-2.5) to[out=90,in=180] (4.5,-1.5)
to[out=0,in=90] (5,-2);
\draw[smooth] (5,-2) to [out=90,in=180] (5.5,-1.5)
to [out=0,in=90] (6,-2.5);
\draw [fill=black] (5,-2) circle (0.05);
\draw[very thick, dotted] (5,-2) to (5,-2.5);
\draw(5,-3) node{B.3};
\end{scope}


\draw(3,-2.5) to[out=90,in=180] (3.5,-1.5)to[out=0,in=90] (4,-2);
\draw(4,-2) to[out=90,in=180] (4.5,-1.5);
\draw [fill=black] (4,-2) circle (0.05);
\draw[very thick, dotted] (4,-2) to (4,-2.5);
\draw [fill=black] (4.5,-1.5) circle (0.05);
\draw[very thick, dotted] (4.5,-1.5) to (5,-1.5);
\draw(4,-3) node{B.3};

\end{tikzpicture}
 \caption{
 \label{fig:ListBasicCases}
Basic graphs.}
\end{figure}



\end{document}